\documentclass[12pt,reqno]{NumPDEsArticle}

\usepackage[utf8]{inputenc}
\usepackage[T1]{fontenc}
\usepackage[english]{babel}

\usepackage{NumPDEsMacros}

\usepackage{csquotes}
\usepackage{enumitem}
\usepackage{amsmath,amssymb}

\usepackage{todonotes}

\usepackage[export]{adjustbox}%

\usepackage{graphicx}
\usepackage{subcaption}
\usepackage{pgfplots}
\pgfplotsset{compat=1.9}
\usepackage{pgfplotstable}

\usepackage{bbm}

\newcommand{\exact}{{\star}}
\newcommand{\g}{\boldsymbol{g}}

\newcommand\mapsfrom{\mathrel{\reflectbox{\ensuremath{\mapsto}}}}

\def\hook{\hookrightarrow}
\def\CGNS{C_{\mathrm{GNS}}}

\def\kmax{{\underline{k}}}
\def\kmaxl{{\underline{k}(\ell)}}

\def\kMAXL{\underline{k}(\ell)}

\def\kk{{\underline{k}}}
\def\lmax{{\underline{\ell}}}

\def\Cest{C_{\mathrm{est}}}
\def\qest{q_{\mathrm{est}}}

\def\qN{q_{\mathrm{N}}}
\def\qE{q_{\mathcal{E}}}

\let\div\relax
\DeclareMathOperator{\div}{div}

\def\coarse{H}
\def\fine{h}

\def\k{{\underline{k}}}

\let\div\relax
\DeclareMathOperator{\div}{div}

\def\opt{{\rm opt}}

\def\lambdaopt{\lambda_{\rm opt}}

\def\A{\boldsymbol{A}}

\def\f{\boldsymbol{f}}
\def\n{\boldsymbol{n}}

\newcommand{\dist}{\text{{\usefont{U}{DSSerif}{m}{n}d}}}

\newcommand{\bbA}{\mathbb{A}}

\title[AILFEM for semilinear PDEs]{Cost-optimal adaptive iterative linearized FEM\\ for semilinear elliptic PDEs}
\author{Roland Becker}
\address{Université de Pau et des Pays de l’Adour, IPRA-LMAP, Avenue de l’Université BP 1155, 64013 PAU Cedex, France}
\email{roland.becker@univ-pau.fr}

\address{TU Wien, Institute of Analysis and Scientific Computing, Wiedner Hauptstr. 8--10/E101/4, 1040 Vienna, Austria}
\author{Maximilian Brunner}
\email{maximilian.brunner@asc.tuwien.ac.at \quad \rm (corresponding author)}

\author{Michael Innerberger}
\email{michael.innerberger@asc.tuwien.ac.at}%

\author{Jens Markus Melenk}
\email{melenk@asc.tuwien.ac.at}%

\author{Dirk Praetorius}
\email{dirk.praetorius@asc.tuwien.ac.at}

\thanks{The authors thankfully acknowledge support by the Austrian Science Fund (FWF) through the doctoral school \emph{Dissipation and dispersion in nonlinear PDEs} (grant W1245) and the stand-alone projects \emph{Computational nonlinear PDEs} (grant P33216) and \emph{Analysis of $\HH$-matrices} (grant
		P28367). Michael Innerberger, Jens Markus Melenk, and Dirk Praetorius are supported by the SFB \emph{Taming complexity in partial differential systems} (grant SFB F65). Additionally, Maximilian Brunner and Michael Innerberger are supported by the \emph{Vienna School of Mathematics}. The authors thank the anonymous reviewers for their constructive feedback which helped to improve results and presentation.}

\makeatother

\advance\footskip0.4cm
\advance\textheight-0.4cm
\calclayout
\makeatletter
\def\@seccntformat#1{\hspace*{4mm}%
  \protect\text{}{\protect\@secnumfont
    \ifnum\pdfstrcmp{subsection}{#1}=0 \bfseries\fi
    \csname the#1\endcsname
    \protect\@secnumpunct
  }%
}
\makeatother

\begin{document}


\maketitle
\thispagestyle{fancy}

\begin{abstract}
	We consider scalar semilinear elliptic PDEs where the nonlinearity is strongly monotone, but only locally Lipschitz continuous. We formulate an adaptive iterative linearized finite element method (AILFEM) which steers the local mesh refinement as well as the iterative linearization of the arising nonlinear discrete equations. To this end, we employ a damped Zarantonello iteration so that, in each step of the algorithm, only a linear Poisson-type equation has to be solved. We prove that the proposed AILFEM strategy guarantees convergence with optimal rates, where rates are understood with respect to the overall computational complexity (i.e., the computational time). Moreover, we formulate and test an adaptive algorithm where also the damping parameter of the Zarantonello iteration is adaptively adjusted. Numerical experiments underline the theoretical findings.
\end{abstract}



\section{Introduction}

\subsection{State of the art}

Cost-optimal computation of a discrete solution with an error below a given tolerance is the prime aim of any numerical method. 
Since convergence of numerical schemes is usually (but not necessarily) spoiled by singularities of the (given) data or the (unknown) solution, \textsl{a~posteriori} error estimation and adaptive mesh refinement schemes are pivotal to reliable and efficient numerical approximation.
This is the foundation of  adaptive finite element methods (AFEM), for which the mathematical understanding of convergence and optimality is fairly mature; we refer to~\cite{bv1984, doerfler1996, mns2000, bdd2004, stevenson2007, msv2008, ckns2008, ks2011, cn2012, ffp2014} for linear elliptic equations, to~\cite{veeser2002, dk2008, bdk2012, gmz2012, ghps2018} for certain quasi-linear PDEs, and to~\cite{axioms} for an overview of available results on rate-optimal AFEM. 

In particular, for nonlinear PDEs, the arising discrete equations must be solved iteratively. The interplay of adaptive mesh refinement and iterative solvers has been treated extensively in the literature; we refer, e.g., to~\cite{stevenson2007, bms2010, agl2013, MR3095916} for algebraic solvers for linear PDEs, to~\cite{aev2011, gmz2011,   aw2015, hw2018, ghps2018, hw2020:convergence, hw2020:ailfem} for the iterative linearization of nonlinear PDEs, and to~\cite{ev2013, hpsv2021} for fully adaptive schemes including linearization and algebraic solver. For the latter works, the consideration is usually restricted to the class of strongly monotone and globally Lipschitz continuous nonlinearities; see~\cite{gmz2011} for the first plain convergence result,~\cite{hw2020:convergence} for an abstract framework for plain convergence of adaptive iteratively linearized finite element methods (AILFEM), \cite{ghps2018,ghps2021} for rate-optimality of AILFEM based on the Zarantonello iteration (as proposed in~\cite{cw2017}), and \cite{hpw2021} for rate-optimality for other linearization strategies including the Ka\v{c}anov iteration as well as the damped Newton method. In particular, we note that~\cite{ghps2021, hpw2021, hpsv2021} prove optimal convergence rates with respect to the overall computational cost.
For more general nonlinear operators, optimal convergences rates are empirically observed (e.g.,\allowbreak~\cite{ev2013}), but the quest for a sound mathematical analysis is still ongoing.

\subsection{Contributions of the present work}

We prove optimal convergence of AILFEM for strongly monotone, but only locally Lipschitz continuous operators, where our interest stems from the treatment of semilinear elliptic PDEs. For $d \in \{1, 2, 3\}$ and a bounded Lipschitz domain $\Omega \subset \R^d$, our model problem reads: Find the (unique) solution $u^\exact \in H^1_0(\Omega)$ to the (scalar) semilinear elliptic PDE
\begin{align}\label{eq:strongform:primal}
	-\div(\A \nabla u^\exact) + b(u^\exact) = f - \div \f
	\text{ \ in } \Omega
	\quad \text{subject to} \quad 
	u^\exact = 0 \text{ \ on }  \partial \Omega,
\end{align}
where we refer to Section~\ref{section:modelproblem} for a discussion of the precise assumptions on the diffusion matrix $\A$, the semilinearity $b$, and the given data $f$ and $\f$. The presented AILFEM algorithm employs the Zarantonello linearization with a damping parameter $\delta>0$, requiring only to solve a \emph{linear} Poisson-type problem in each linearization step. The AILFEM algorithm takes the form
\vspace{0.5\baselineskip}
\begin{center}
\begin{tikzpicture}
	\tikzstyle{afemnode} = [draw, very thick, color=lightgray, text=black, minimum width=5em, rounded corners]
	\tikzstyle{afemarrow} = [very thick, color=lightgray, -stealth]
	\tikzstyle{dummynode} = [draw=none]
	
	\node[afemnode] (S) at (0,0) {\textsc{Iteratively Solve and Estimate}};
	\node[afemnode, right=3em of S] (M) {\textsc{Mark}};
	\node[afemnode, right=3em of M] (R) {\textsc{Refine}};
	\coordinate[above=1.5em of S] (D1);
	\coordinate[above=1.5em of R] (D2);
	
	\draw[afemarrow] (S) -- (M);
	\draw[afemarrow] (M) -- (R);
	
	\draw[afemarrow, rounded corners] (R) -- (D2) -- (D1) -- (S);
\end{tikzpicture}
\vspace{0.5\baselineskip} 
\end{center}
where the first step represents an inner loop of the Zarantonello iteration and error estimation by a residual \textsl{a posteriori} error estimator. This inner loop is stopped when the linearization error (measured in terms of the energy difference of discrete Zarantonello iterates) is small with respect to the discretization error (measured in terms of the error estimator). However, since the PDE operator is only locally Lipschitz continuous, the stopping criterion must be slightly extended when compared to that of~\cite{hw2020:convergence, ghps2021, hpw2021} for globally Lipschitz continuous operators. As usual in this context, we employ the D\"orfler marking to single out elements for refinement, and mesh refinement relies on newest vertex bisection. 

We prove that the solver iterates are uniformly bounded, provided that the Zarantonello parameter $\delta$ is chosen appropriately (Corollary~\ref{cor:crucial}). For arbitrary adaptivity parameters ($\theta$ for marking and $\lambda$ for stopping the Zarantonello iteration), we then prove \emph{full} linear convergence (Theorem~\ref{theorem:fulllinear}), i.e., linear convergence regardless of the algorithmic decision for yet another solver step or mesh refinement. For sufficiently small marking parameters, this even guarantees 
\emph{rate-optimality with respect to the number of degrees of freedom} (Theorem~\ref{theorem:rates}) and \emph{cost-optimality}, i.e., rate-optimality with respect to the overall computational cost (Corollary~\ref{cor:cost}).

\subsection{Outline}

This work is organized as follows: In Section~\ref{section:monlip}, we present our adaptive iterative linearized finite element method~(Algorithm~\ref{algorithm:idealized}) and the details of its individual steps. This includes the discussion of the abstract Hilbert space setting, the precise assumptions for the iterative solver, and a discussion of the extended stopping criterion. Finally, we prove full linear convergence of the proposed AILFEM algorithm (Theorem~\ref{theorem:fulllinear}) and optimal rates both with respect to the degrees of freedom (Theorem~\ref{theorem:rates}) as well as the overall computational cost (Corollary~\ref{cor:cost}).
	In Section~\ref{section:modelproblem}, we introduce and discuss semilinear elliptic PDEs, which fit into the abstract framework of Section~\ref{section:monlip}. 	
	Section~\ref{section:practical_algorithm} presents a practical extension of our AILFEM strategy (Algorithm~\ref{algorithm:practical}), which includes the adaptive choice of the Zarantonello damping parameter $\delta$. In Section~\ref{section:numerical}, we support our theoretical findings with numerical experiments. 
	Finally, Appendix~\ref{section:appendix} concludes the work by providing additional material, which allows us to apply the abstract setting to a wider range of problems like non-scalar semilinear PDEs.
	
\subsection{General notation}

Without ambiguity, we use $| \, \cdot \, |$ to denote the absolute value $|\lambda|$ of a scalar $\lambda \in \R$, the Euclidean norm $|x|$ of a vector $x \in \R^d$, and the Lebesgue measure $|\omega|$ of a set $\omega \subseteq \R^d$, depending on the respective context. Furthermore, $\# \UU$ denotes the cardinality of a finite set $\UU$.


\section{Strongly monotone operators} 
\label{section:monlip}

In this section, we present the mathematical heart of our analysis, which will later be applied to strongly monotone semilinear PDEs.
\subsection{Abstract model problem} \label{subsection:modelproblem}

Let $\XX$ be a Hilbert space over $\R$ with scalar product $\sprod{\cdot}{\cdot}$ and induced norm $\enorm{\, \cdot \,}$. Let $\XX_\coarse \subseteq \XX$ be a closed subspace. Let $\XX'$ be the dual space with norm $\norm{\cdot }{\XX'}$ and denote by $\dual{\cdot}{\cdot}$ the duality bracket on $\XX' \times \XX$. Let $\AA \colon \XX \to \XX'$ be a nonlinear operator. We suppose that $\AA$ is {\bf strongly monotone}, i.e., there exists $\alpha > 0$ such that
\begin{align}\label{eq:strongly-monotone}
 \alpha \, \enorm{v - w}^2 \le \dual{\AA v - \AA w}{v - w}%
 \quad \text{for all } v, w \in \XX. \tag{SM}
\end{align}
Moreover, we suppose that $\AA$ is {\bf locally Lipschitz continuous}, i.e., for all $\vartheta > 0$, there exists $L[\vartheta] >0$ such that
\begin{align}\label{eq:locally-lipschitz}
\! \! \!\dual{\AA v \!-\! \AA w\!}{\!\varphi} \! \le \!  L[\vartheta] \enorm{v\! -\! w} \enorm{\varphi}
 \text{ for all } v, w, \varphi \in \XX 
\! \text{ with }\! \max\big\{ \enorm{v}, \enorm{v \!-\! w}\big\} \!\le \! \vartheta. \tag{LIP}
\end{align}
\begin{remark}
\cite[p.~565]{zeidler} defines local Lipschitz continuity as follows: For all $\Theta > 0$, there exists $L'[\Theta] >0$ such that
\begin{align}\label{eq:locally-lipschitz:variant}
	\! \! \!\dual{\AA v \!-\! \AA w\!}{\!\varphi} \! \le \! L'[\Theta] \enorm{v \! -\! w} \enorm{\varphi}
	\text{ for all } v, w, \varphi \in \XX 
	\text{ with } \max\big\{ \enorm{v}, \enorm{w}\big\} \! \le \! \Theta.
\end{align} 
Conditions~\eqref{eq:locally-lipschitz} and~\eqref{eq:locally-lipschitz:variant} are indeed equivalent in the sense that
\begin{align*}
	\max\big\{ \enorm{v}, \enorm{w}\big\} & \le \max\big\{ \enorm{v}, \enorm{v- w} + \enorm{v}\big\} \le 2 \,\vartheta, \\
	\max\big\{ \enorm{v}, \enorm{v- w} \big\}  &\le  \max \big\{\enorm{v}, \enorm{v} + \enorm{w}\big\} \le 2\, \Theta.
\end{align*}
However,~\eqref{eq:locally-lipschitz} is better suited for the inductive structure in the proof of~Corollary~\ref{cor1:zarantonello}. 
	\end{remark}
Without loss of generality, we may suppose that $\AA0 \neq F \in \XX'$. We consider the operator equation
\begin{align}\label{eq:weakform}
 \AA u^\exact = F.
\end{align}
For any closed subspace $\XX_\coarse \subseteq \XX$, we consider the corresponding Galerkin discretization
\begin{align}\label{eq:weakform:discrete}
 \dual{\AA u_\coarse^\exact}{v_\coarse} = \dual{F}{v_\coarse}
 \quad \text{ for all } v_\coarse \in \XX_\coarse.
\end{align}
We observe that the setting of strongly montone and locally Lipschitz operators yields existence and uniqueness of the solutions to~\eqref{eq:weakform}--\eqref{eq:weakform:discrete} as well as a C\'{e}a-type estimate.

\begin{proposition}\label{prop:existence}
	Suppose that $\AA$ satisfies~\eqref{eq:strongly-monotone} and~\eqref{eq:locally-lipschitz}. Then,~\eqref{eq:weakform}--\eqref{eq:weakform:discrete} admit unique solutions $u^\exact \in \XX$ and $u_\coarse^\exact \in \XX_\coarse$, respectively, and it holds that
\begin{align}\label{eq:exact:bounded}
 \max\big\{ \enorm{u^\exact}, \enorm{u^\exact_\coarse} \big\} \le M \coloneqq \frac{1}{\alpha} \, \norm{F - \AA 0}{\XX'} \neq 0
\end{align}
as well as
\begin{align}\label{eq:lemma:cea}
	\enorm{u^\exact - u_\coarse^\exact} \le \Ccea \min_{v_\coarse \in \XX_\coarse}  \enorm{u^\exact -v_\coarse} \quad \text{ with } \quad \Ccea = L[2M]/\alpha.
\end{align}
\end{proposition}
\begin{proof}
Since $\AA$ is (even locally Lipschitz) continuous, existence of $u_\coarse^\exact$ follows from the Browder--Minty theorem on monotone operators~\cite[Theorem~26.A]{zeidler}. Uniqueness of $u_\coarse^\exact$ follows from strong monotonicity, since any two solutions $u_\coarse^\exact, u_\coarse \in \XX_\coarse$ to~\eqref{eq:weakform:discrete} satisfy
\begin{align*}
 \alpha \, \enorm{u_\coarse^\exact - u_\coarse}^2 
 \eqreff{eq:strongly-monotone}\le 
 \dual{\AA u_\coarse^\exact - \AA u_\coarse}{u_\coarse^\exact - u_\coarse}\eqreff{eq:weakform:discrete}= 0
\end{align*}
and hence $u_\coarse^\exact = u_\coarse$. Boundedness~\eqref{eq:exact:bounded} follows from
\begin{align*}
 \alpha \, \enorm{u_\coarse^\exact}^2 
 \eqreff{eq:strongly-monotone}\le 
 \dual{\AA u_\coarse^\exact - \AA 0}{u_\coarse^\exact}
 = \dual{F - \AA 0}{u_\coarse^\exact}
 \le \norm{F - \AA 0}{\XX'} \enorm{u_\coarse^\exact}.
\end{align*}
Since~\eqref{eq:weakform} is equivalent to~\eqref{eq:weakform:discrete} with $\XX = \XX_\coarse$, the foregoing results also cover $u^\exact \in \XX$. This concludes the proof of~\eqref{eq:exact:bounded}. To see the C\'{e}a-type estimate~\eqref{eq:lemma:cea}, recall the Galerkin orthogonality
\begin{align}\label{eq:aux:gal-ortho}
	\dual{\AA u^\exact -\AA u^\exact_\coarse}{v_\coarse} = 0 \quad \text{ for all } v_\coarse \in \XX_\coarse.
\end{align}
For $v_\coarse \in \XX_\coarse$, standard reasoning leads us to
\begin{align*}
	\alpha \, \enorm{u^\exact- u^\exact_\coarse}^2 \, \,\,  \stackrel{\mathclap{\eqref{eq:strongly-monotone}}}{\le}  \, \,\prod{\AA u^\exact -\AA u^\exact_\coarse}{u^\exact-u^\exact_\coarse} &\stackrel{\mathclap{\eqref{eq:aux:gal-ortho}}}{=} \prod{\AA u^\exact -\AA u^\exact_\coarse}{u^\exact-v_\coarse} \\
	& \stackrel{\mathclap{\eqref{eq:locally-lipschitz}}}{\le}  \, \, L[2M] \,\enorm{u^\exact-u^\exact_\coarse} \enorm{u^\exact-v_\coarse}.
\end{align*}
Rearranging the last estimate, we prove~\eqref{eq:lemma:cea}, where the minimum is attained since $\XX_\coarse$ is closed. This concludes the proof. 
\end{proof}

Finally, we suppose that the operator $\AA$ possesses a potential $\PP$: there exists a G\^{a}teaux differentiable function $\PP\colon \XX \to \R$ such that its derivative $\d{\PP}\colon \XX \to \XX'$ coincides with $\AA$, i.e., it holds that
	\begin{align}\label{eq:potential}
		\prod{\AA w}{v} = \prod{\d{\PP} (w)}{v} = \lim_{\substack{t \to 0 \\ t \in \R}}\, \frac{\PP(w+tv) - \PP(w)}{t} \quad \text{ for all } v, w\in \XX. \tag{POT}
	\end{align}
We define the energy $\EE (v) \coloneqq (\PP-F) v$, where $F$ is the right-hand side from~\eqref{eq:weakform}. 

Note that the energy $\EE$ trivially satisfies that
\begin{align}\label{eq:qo:abstract}
	\EE(v_{\coarse}) - \EE(u^\exact) = 	\big[\EE(v_{\coarse}) - \EE(u^\exact_{\coarse})\big] +\big[\EE(u^\exact_{\coarse}) - \EE(u^\exact)\big] \quad \text{ for all } v_\coarse \in \XX_\coarse
\end{align}
and all these energy differences are non-negative; see~\eqref{eq:emin}. 

Moreover, assumption~\eqref{eq:potential} admits the following classical equivalence:
\begin{lemma}[\phantom{}{see, e.g.,~\cite[Lemma~5.1]{ghps2018}}]\label{lemma:equivalence}
	Suppose that $\AA$ satisfies~\eqref{eq:strongly-monotone}, \eqref{eq:locally-lipschitz}, and~\eqref{eq:potential}. Let $\vartheta \ge M$. Let $v_{\coarse} \in \XX_{\coarse}$ with $\enorm{ v_\coarse - u^\exact_{\coarse}} \le \vartheta$. Then, it holds that
	\begin{align}\label{eq:equivalence}
		\frac{\alpha}{2} \, \enorm{v_{\coarse} - u^\exact_{\coarse}}^2 \le  \EE(v_{\coarse}) - \EE(u^\exact_{\coarse}) \le  \frac{L[\vartheta]}{2}\, \enorm{v_{\coarse} - u^\exact_{\coarse}}^2.
	\end{align}
	In particular, the solution $u^\exact_\coarse$ of~\eqref{eq:weakform:discrete} is indeed the unique minimizer of $\EE$ in $\XX_\coarse$, i.e.,
	\begin{align}\label{eq:emin}
		\EE(u^\exact_\coarse) \le \EE(v_\coarse) \quad \text{ for all } v_\coarse \in \XX_\coarse,
	\end{align}
	and, therefore,~\eqref{eq:weakform:discrete} can equivalently be reformulated as an energy minimization problem:
	\begin{align} \tag*{\qed}
		\text{ Find } \quad u^\exact_{\coarse} \in \XX_{\coarse} \quad  \text{ such that } \quad \EE(u^\exact_{\coarse}) =\min_{v_{\coarse} \in \XX_{\coarse}} \EE(v_{\coarse}). 
	\end{align}
\end{lemma}

\subsection{Zarantonello iteration}
\label{subsection:zarantonello}
Let $\XX_\coarse \subseteq \XX$ be a closed subspace. For given damping parameter $\delta > 0$, we define the Zarantonello mapping $\Phi_H(\delta; \cdot)\colon \XX_\coarse \to \XX_\coarse$ by
\begin{align}\label{eq:zarantonello}
	\sprod{\Phi_\coarse(\delta; w_\coarse)}{v_\coarse}
	= \sprod{w_\coarse}{v_\coarse} + \delta \, \dual{F - \AA w_\coarse}{v_\coarse}
	\quad \text{for all } v_\coarse \in \XX_\coarse.
\end{align}
Clearly, existence and uniqueness of $\Phi_\coarse(\delta; w_\coarse) \in \XX_\coarse$ and hence well-posedness of $\Phi_\coarse(\delta; \cdot)$ follows from the Riesz theorem.
The following two estimates are obvious: first,
\begin{align}\label{eq:zarantonello:iterate}
	\enorm{\Phi_\coarse(\delta; w_\coarse) - w_\coarse}
	\le \delta \, \norm{F - \AA w_\coarse}{\XX'} = \delta \mkern-9mu\sup_{v \in \XX \setminus \{0\}} \frac{\prod{F-\AA w_\coarse}{v}}{\enorm{v}}
	\text{ for all } w_\coarse \in \XX_\coarse;
\end{align}
second,
\begin{align}\label{eq:zarantonello:lipschitz}
	\enorm{\Phi_\coarse(\delta; v_\coarse) \!-\! \Phi_\coarse(\delta; w_\coarse)} 
	\le \enorm{v_\coarse \!-\! w_\coarse} + \delta \, \norm{\AA v_\coarse \!-\! \AA w_\coarse}{\XX'}
	\, \text{ for all }  v_\coarse, w_\coarse \in \XX_\coarse.
\end{align}
Due to the local Lipschitz continuity~\eqref{eq:locally-lipschitz} of $\AA$, this proves that also $\Phi_\coarse(\delta; \cdot)$ is locally Lipschitz continuous.
By definition, $u_\coarse^\exact \in \XX_\coarse$ solves~\eqref{eq:weakform:discrete} if and only it is a fixed point of $\Phi_\coarse(\delta; \cdot)$, i.e., $u_\coarse^\exact = \Phi_\coarse(\delta; u_\coarse^\exact)$. 

\subsection{Zarantonello iteration and norm contraction} Let $\XX_\coarse \subseteq \XX$ be a closed subspace. 
The next proposition~\cite[Section 25.4]{zeidler} proves local contraction of $\Phi_\coarse(\delta;\cdot)$ with respect to the energy norm. For the convenience of the reader, we include the proof to highlight that local Lipschitz continuity suffices.
\begin{proposition}[norm contraction]\label{prop:zarantonello}
Suppose that $\AA$ satisfies~\eqref{eq:strongly-monotone} and~\eqref{eq:locally-lipschitz}. Let $\vartheta > 0$ and $v_\coarse, w_\coarse \in \XX_\coarse$ with $\max \big\{ \enorm{v_\coarse}, \enorm{v_\coarse-w_\coarse} \big\} \le \vartheta$. Then, for all $0 < \delta < 2\alpha / L[\vartheta]^2$ and $0 < \qN[\delta]^2 \coloneqq 1 - \delta(2\alpha - \delta L[\vartheta]^2) < 1$, it holds that
	\begin{align}\label{eq:prop:zarantonello}
		\enorm{\Phi_\coarse(\delta; v_\coarse) - \Phi_\coarse(\delta; w_\coarse)} 
		\le \qN[\delta] \, \enorm{v_\coarse - w_\coarse}.
	\end{align}
	We note that $\qN[\delta] \to 1$ as $\delta \to 0$. Moreover, for known $\alpha$ and $L[\vartheta]$, the contraction constant $\qN[\delta]^2 = 1-\alpha^2/L[\vartheta]^2 = 1- \alpha \, \delta $ is minimal and only attained for $\delta = \alpha/L[\vartheta]^2$.
\end{proposition}

\begin{proof}
	Recall that the Riesz mapping 
	\begin{align} \label{eq:isometry}
		I_\coarse \colon \XX_\coarse \to \XX_\coarse^\prime, \quad v_\coarse \mapsto I_\coarse(v_\coarse) \coloneqq \sprod{\cdot \,}{v_\coarse} \quad \text{ for all } v_\coarse \in \XX_\coarse
	\end{align}
	is an isometric isomorphism; cf., e.g.,~\cite[Chapter~III.6]{yosida}. Therefore, a reformulation of the Zarantonello iteration reads
	\begin{align*}
		\sprod{\Phi_\coarse(\delta; w_\coarse)}{\varphi_\coarse} = \sprod{w_\coarse}{\varphi_\coarse} + \delta \, \sprod{\varphi_\coarse}{I_\coarse^{-1} (F - \AA w_\coarse)} \quad \text{ for all } \varphi_\coarse, w_\coarse \in \XX_\coarse.
	\end{align*}
	Given $v_\coarse, w_\coarse \in \XX_\coarse$ with $\max\big\{ \enorm{v_\coarse}, \enorm{v_\coarse - w_\coarse} \big\} \le \vartheta$, we exploit the last equality for $\Phi_\coarse(\delta; v_\coarse)$ by subtraction of $\Phi_\coarse(\delta; w_\coarse)$ and use $\varphi_\coarse = \Phi_\coarse(\delta; v_\coarse) - \Phi_\coarse(\delta; w_\coarse)$ to arrive at
	\begin{align*}
		\enorm{\Phi_\coarse(\delta; v_\coarse) - \Phi_\coarse(\delta; w_\coarse)}^2 & =  \enorm{v_\coarse - w_\coarse}^2 - 2 \delta\, \sprod{v_\coarse - w_\coarse}{I^{-1}_\coarse(\AA v_\coarse - \AA w_\coarse)} \\
		& \quad + \delta^2\, \enorm{I^{-1}_\coarse (\AA v_\coarse - \AA w_\coarse)}^2.
	\end{align*}
	The isometry property of $I_\coarse$ implies that
	\begin{align*}
		\enorm{I^{-1}_\coarse(\AA v_\coarse - \AA w_\coarse)}^2 \eqreff{eq:isometry}{=} \norm{\AA v_\coarse - \AA w_\coarse}{\XX^\prime}^2 \eqreff{eq:locally-lipschitz}{\le} L[\vartheta]^2 \,\enorm{v_\coarse - w_\coarse}^2. 
	\end{align*}
	Moreover, it holds that
	\begin{align*}
		\sprod{v_\coarse - w_\coarse}{I^{-1}_\coarse(\AA v_\coarse - \AA w_\coarse)} \eqreff{eq:isometry}{=} \dual{\AA v_\coarse - \AA w_\coarse}{v_\coarse - w_\coarse} \eqreff{eq:strongly-monotone}{\geq} \alpha  \,\enorm{v_\coarse - w_\coarse}^2. 
	\end{align*}
	Combining these observations, we see that
	\begin{align*}
		0 \le \enorm{\Phi_\coarse(\delta; v_\coarse)                                                                                                                                                                                                                                                                                                                                                                                                                                                                                                                                                                                                                                                                                                                                                                                                                                                                                                                                                                                                                                                                                                                                                                                                                                                                                                                                                                                                                                                                                                                                                                                                                                                                                                                                                                                                                                                                                                                                                                                                                                                                                                                                                                                                                                                                                                                                                                                                                                                                                                                                                                                                                                                                                                                                                                                                                                                                                                                                                                                                                                                                                                                                                                                                                                                                                                                                                                                                                                                                                                                                                                                                                                                                                                                                                                                                                                                                                                                                                                                                                                                                                                                                                                                                                                                                                                                                                                                                                                                                                                                                                                                                                                                                                                                                                                                                                                                                                                                                                                                                                                                                                                                                                                   - \Phi_\coarse (\delta; w_\coarse)}^2 \le [1- 2 \delta \alpha + \delta^2 L[\vartheta]^2] \, \enorm{v_\coarse - w_\coarse}^2.
	\end{align*}
	Rearranging $\qN[\delta]^{2} \coloneqq 1- 2 \delta \alpha + \delta^2 L[\vartheta]^2  = 1- \delta ( 2\alpha - \delta L[\vartheta]^2)$, we conclude the first claim. Finally, it follows from elementary calculus that $\delta = \alpha/L[\vartheta]^2$ is the unique minimizer of the quadratic polynomial $\qN[\delta]$ if $\alpha$ and $L[\vartheta]^2$ are fixed. This concludes the proof.
\end{proof}
\begin{corollary}\label{cor1:zarantonello}
	Suppose that $\AA$ satisfies~\eqref{eq:strongly-monotone} and~\eqref{eq:locally-lipschitz}. Let $u_\coarse^0 \in \XX_\coarse$ with $\enorm{u_\coarse^0} \le 2M$. Let $0 < \delta < 2 \alpha / L[3M]^2$ and let $0 < \qN[\delta]< 1$ be chosen according to Proposition~\ref{prop:zarantonello}, where $\vartheta = 3M$. Define
	\begin{align}\label{eq0:cor:zarantonello}
		u_\coarse^{k+1} \coloneqq \Phi_\coarse(\delta; u_\coarse^k)
		\quad \text{for all } k \in \N_0.
	\end{align}
	Then, it holds that
	\begin{align}\label{eq:contr}
		(1- \qN[\delta]) \,\enorm{u^\exact_\coarse - u_\coarse^{k}}  \le \enorm{ u_\coarse^{k+1} - u_\coarse^{k}}\le (1 +  \qN[\delta] )\, \enorm{u^\exact_\coarse - u_\coarse^{k}}
	\end{align}
	and
	\begin{align}\label{eq1:cor:zarantonello}
		\enorm{u_\coarse^\exact \!-\! u_\coarse^{k+1}}
		\le \qN[\delta] \, \enorm{u_\coarse^\exact \!-\! u_\coarse^k}
		\le \qN[\delta]^{k+1} \, \enorm{u_\coarse^\exact \!-\! u_\coarse^0}
		\le 3M
		\quad \text{ for all } k \in \N_0.
	\end{align}
	In particular, it follows that
	\begin{align}\label{eq2:cor:zarantonello}
		\enorm{u_\coarse^k}
		\le 4M \quad
		\text{ for all } k \in \N_0.
	\end{align}
\end{corollary}

\begin{proof}
	The claim~\eqref{eq1:cor:zarantonello} is proved by induction on $k$. By recalling~\eqref{eq:exact:bounded}, it holds that 
	$\enorm{u_\coarse^\exact} \le M$ as well as 
	$\enorm{u_\coarse^\exact - u_\coarse^0} \le \enorm{u_\coarse^\exact} + \enorm{u_\coarse^0} 
	\le 3M$. Therefore, Proposition~\ref{prop:zarantonello} proves that
	\begin{align*}
		\enorm{u_\coarse^\exact - u_\coarse^1}
		= \enorm{\Phi_\coarse(\delta; u_\coarse^\exact) - \Phi_\coarse(\delta; u_\coarse^{0})}
		\eqreff{eq:prop:zarantonello}\le 
		\qN[\delta] \, \enorm{u_\coarse^\exact - u_\coarse^0}
		\le 3M.
	\end{align*}
	This proves~\eqref{eq1:cor:zarantonello} for $k = 0$.
	In the induction step, we know that $\enorm{u_\coarse^\exact - u_\coarse^k} \le 3M$.
	As before,~\eqref{eq:prop:zarantonello} from Proposition~\ref{prop:zarantonello} and the induction hypothesis prove that
	\begin{align*}
		\enorm{u_\coarse^\exact - u_\coarse^{k+1}}
		= \enorm{\Phi_\coarse(\delta; u_\coarse^\exact) - \Phi_\coarse(\delta; u_\coarse^k)} \,
		& \eqreff*{eq:prop:zarantonello}\le \, 
		\qN[\delta] \, \enorm{u_\coarse^\exact - u_\coarse^k} 
		\\ &
		\le 
			\qN[\delta]^{k+1} \, \enorm{u_\coarse^\exact - u_\coarse^0} 
		\le 3M. 
	\end{align*}
	This proves~\eqref{eq1:cor:zarantonello} for general $k \in \N_0$, and the inequalities~\eqref{eq:contr} follow from~\eqref{eq:prop:zarantonello} and the triangle inequality. Moreover, the triangle inequality yields that
	\begin{align*}
		\enorm{u_\coarse^k}
		\le \enorm{u_\coarse^\exact} + \enorm{u_\coarse^\exact - u_\coarse^k}
		\le 4M.
	\end{align*}
	This concludes the proof.
\end{proof}

\begin{corollary}\label{cor2:zarantonello}
Suppose that $\AA$ satisfies~\eqref{eq:strongly-monotone} and~\eqref{eq:locally-lipschitz}. Let $u_\coarse^0 \in \XX_\coarse$ with $\enorm{u_\coarse^0} \le 2M$. Let $0 < \delta < 2 \alpha / L[6M]^2$ and let $0 < \qN[\delta]< 1$ be chosen according to Proposition~\ref{prop:zarantonello}, where $\vartheta = 6M$. 
	Then, the Zarantonello iterates from~\eqref{eq0:cor:zarantonello} satisfy~\eqref{eq:contr}--\eqref{eq2:cor:zarantonello} as well as
	\begin{align}\label{eq3:cor:zarantonello}
		\enorm{u_\coarse^{k+1}\!-\! u_\coarse^k}
		\le \qN[\delta] \, \enorm{u_\coarse^k \!-\! u_\coarse^{k-1}}
		\le \qN[\delta]^{k} \, \enorm{u_\coarse^1 \!-\! u_\coarse^0}
		\le 6M
		\text{ for all } k \in \N.
	\end{align}  
\end{corollary}

\begin{proof}
	Since $L[3M] \le L[6M]$, it only remains to prove~\eqref{eq3:cor:zarantonello}. We argue by induction and note that
	\begin{align*}
		\enorm{u_{\coarse}^1 - u_{\coarse}^0}
		\le  \enorm{ u_{\coarse}^1} + \enorm{ u_{\coarse}^0}  \eqreff{eq2:cor:zarantonello} \le 6M.
	\end{align*}
	Therefore, Proposition~\ref{prop:zarantonello} proves that
	\begin{align*}
		\enorm{u_\coarse^2 - u_\coarse^1}
		= \enorm{\Phi_\coarse(\delta; u_\coarse^1) - \Phi_\coarse(\delta; u_\coarse^0)}
		\eqreff{eq:prop:zarantonello}\le 
		\qN[\delta] \, \enorm{u_\coarse^1 - u_\coarse^0}
		\le 6M. 
	\end{align*}
	This proves~\eqref{eq3:cor:zarantonello} for $k = 1$. 
	In the induction step, we know that $\enorm{u_\coarse^{k+1} - u_\coarse^k} \le 6M$.
	Therefore, Proposition~\ref{prop:zarantonello} and the induction hypothesis prove that
	\begin{align*}
		\enorm{u_\coarse^{k+2}\!-\! u_\coarse^{k+1}}
		= \enorm{\Phi_\coarse(\delta; u_\coarse^{k+1}) - \Phi_\coarse(\delta; u_\coarse^k)}
		\eqreff{eq:prop:zarantonello}\le 
		\qN[\delta] \, \enorm{u_\coarse^{k+1} - u_\coarse^k}
		\le 6M.
	\end{align*}
	This proves~\eqref{eq3:cor:zarantonello} for general $k \in \N$ and concludes the proof.
\end{proof}

\subsection{Zarantonello iteration and energy contraction}
Let $\XX_\coarse \subseteq \XX$ be a closed subspace. The next result extends the abstract lower bound from~\cite[Proposition~1]{hw2020:convergence} to the Zarantonello iteration in the locally Lipschitz continuous setting.
\begin{lemma}
	\label{lemma:f4}
Suppose that $\AA$ satisfies~\eqref{eq:strongly-monotone},~\eqref{eq:locally-lipschitz}, and~\eqref{eq:potential}. Let $u_\coarse^0 \in \XX_\coarse$ with $\enorm{u_\coarse^0} \le 2M$. Then, for $0 < \delta < 2 \alpha  /L[6M]^2$, the Zarantonello iteration~\eqref{eq:zarantonello} yields that
	\begin{align}\label{eq:f4}
 0 \le \kappa[\delta]\, \enorm{u^{k+1}_{\coarse} - u^{k}_{\coarse}}^2 \le \EE(u^{k}_{\coarse}) - \EE(u^{k+1}_{\coarse}) \le K[\delta] \, \enorm{u^{k+1}_{\coarse} -u^{k}_{\coarse}}^2,
	\end{align} 
where $\kappa[\delta] = (\delta^{-1} - L[6M]/2 ) > 0$ and $K[\delta]=\big( \delta^{-1} - \alpha /2 \big)$.
\end{lemma}
\begin{proof}
	Define $e^{k+1}_{\coarse} \coloneqq u^{k+1}_{\coarse} - u^{k}_{\coarse}$ for all $k \in \N_0$. Then,~\eqref{eq:potential} guarantees that $\EE = \PP - F$ is G\^{a}teaux differentiable. Define $\varphi(t) \coloneqq \EE(u^{k}_{\coarse} + t\, e^{k+1}_{\coarse})$ for $t \in [0,1]$ and observe that
	\begin{align*}
		\varphi^\prime(t)  = \prod{\d{\EE}(u^{k}_{\coarse} + t\, e^{k+1}_{\coarse})}{e^{k+1}_{\coarse}} = \prod{\AA(u^{k}_{\coarse} + t\, e^{k+1}_{\coarse}) - F}{e^{k+1}_{\coarse}}.
	\end{align*}
	For $0 < \delta <2 \alpha / L[6M]^2$, Corollary~\ref{cor2:zarantonello} together with the boundedness $\enorm{u^k_\coarse} \le 4M$ from~\eqref{eq2:cor:zarantonello} and the convexity of the norm show that
	\begin{align}\label{eq:critical}
	 \max \big\{ \enorm{e^{k+1}_{\coarse}}, \enorm{u^k_{\coarse} - t\, e^{k+1}_{\coarse}}\big\} \le 6M \quad \text{ for all } k \in \N_0.
	 \end{align}
 With the fundamental theorem of calculus and the Zarantonello iteration~\eqref{eq:zarantonello}, we see that
	\begin{align*}
		\EE(u^{k}_{\coarse}) - \EE(u^{k+1}_{\coarse}) &= - \int_0^1 \prod{\AA(u^{k}_{\coarse} + t\, e^{k+1}_{\coarse})-F}{e^{k+1}_{\coarse}} \, \d{t} \\
		& = - \int_0^1 \prod{\AA(u^{k}_{\coarse} + t\, e^{k+1}_{\coarse})- \AA u^{k}_{\coarse}}{e^{k+1}_{\coarse}} \, \d{t} - \prod{ \AA u^{k}_{\coarse}-F}{e^{k+1}_{\coarse}}  \\
		& \stackrel{\mathclap{\eqref{eq:zarantonello}}}{=}  - \int_0^1 \prod{\AA(u^{k}_{\coarse} + t \,e^{k+1}_{\coarse})- \AA u^{k}_{\coarse}}{e^{k+1}_{\coarse}} \, \d{t} + \frac{1}{\delta} \sprod{ e^{k+1}_{\coarse}}{e^{k+1}_{\coarse}} \\
		& \stackrel{\mathclap{\eqref{eq:locally-lipschitz}}}{\ge} \, \ \Big( \frac{1}{\delta} - \int_0^1 t L[6M] \, \d{t} \Big) \enorm{u^{k+1}_{\coarse} -u^{k}_{\coarse}}^2 = \Big( \frac{1}{\delta} - \frac{L[6M]}{2}\Big)\enorm{u^{k+1}_{\coarse} -u^{k}_{\coarse}}^2. 
	\end{align*}
	Since $\delta < 2 \alpha / L[6M]^2 \le 2 /L[6M]$, it follows that $\kappa[\delta] = (1/\delta - L[6M]/2 )>0$. This proves the lower bound in~\eqref{eq:f4}. Moreover, the same argument also yields that
		\begin{align*}
			\EE(u^{k}_{\coarse}) - \EE(u^{k+1}_{\coarse}) & \stackrel{\mathclap{\eqref{eq:zarantonello}}}{=}  - \int_0^1 \prod{\AA(u^{k}_{\coarse} + t \,e^{k+1}_{\coarse})- \AA u^{k}_{\coarse}}{e^{k+1}_{\coarse}} \, \d{t} + \frac{1}{\delta}\, \sprod{ e^{k+1}_{\coarse}}{e^{k+1}_{\coarse}} \\
			& \stackrel{\mathclap{\eqref{eq:strongly-monotone}}}{\le} \quad \Big( \frac{1}{\delta} - \int_0^1 \alpha \,t \, \d{t} \Big) \enorm{u^{k+1}_{\coarse} -u^{k}_{\coarse}}^2 = \Big( \frac{1}{\delta} - \frac{\alpha}{2} \Big) \, \enorm{u^{k+1}_{\coarse} -u^{k}_{\coarse}}^2.
		\end{align*}
	This concludes the proof.
\end{proof}

The Zarantonello iterates are also contractive with respect to the energy difference.
\begin{proposition}[energy contraction]\label{proposition:econtraction}
	Suppose that $\AA$ satisfies~\eqref{eq:strongly-monotone}, \eqref{eq:locally-lipschitz}, and \eqref{eq:potential}. Then, for $0 < \delta < 2 \alpha /L[6M]^2$, it holds that
	\begin{subequations}\label{eq:energycontr}
		\begin{equation}\label{eq:energycontr:ineq}
			0 \le \EE(u^{k+1}_{\coarse})-\EE(u^\exact_{\coarse}) \le \qE[\delta]^2 \, [\EE(u^{k}_{\coarse})-\EE(u^\exact_{\coarse})] \quad \text{ for all } k \in \N_0
		\end{equation}
		with contraction constant 
		\begin{equation}\label{eq:energycontr:const} 
			0 \le \qE[\delta]^2 \coloneqq 1 - \Big(1 -  \frac{\delta L[6M]}{2} \Big) \, \frac{2\delta \alpha^2}{L[3M]} < 1.
		\end{equation}
	\end{subequations}
We note that $\qE[\delta] \to 1$ as $\delta \to 0$. Furthermore, for all $k \in \N_0$, it holds that
\begin{align}\label{eq:energycontr:triangle}
	(1- \qE[\delta]^{2})\, \big[\EE(u^{k}_\coarse) - \EE(u_\coarse^\exact)\big] \le \EE(u_\coarse^{k}) - \EE(u_\coarse^{k+1}) \le (1 +  \qE[\delta]^{2} ) \; \big[\EE(u^{k}_\coarse) - \EE(u_\coarse^\exact)\big]. 
\end{align}%
\end{proposition}
\begin{proof}
	First, we observe that
	\begin{align}\label{eq:helper:equiv}
		\begin{split}
		\alpha \, \enorm{u^\exact_\coarse - u^k_\coarse}^2 &\le \prod{\AA u^\exact_\coarse - \AA u^k_\coarse}{u^\exact_\coarse - u^k_\coarse} \stackrel{\eqref{eq:weakform:discrete}}{=}  \prod{F - \AA u^k_\coarse}{u^\exact_\coarse - u^k_\coarse} \\
		& \stackrel{\mathclap{\eqref{eq:zarantonello}}}{=} \,  \frac{1}{\delta} \,\sprod{u^{k+1}_\coarse - u^k_\coarse}{u^\exact_\coarse - u^k_\coarse} \le \frac{1}{\delta} \, \enorm{u^{k+1}_\coarse - u^k_\coarse} \enorm{u^\exact_\coarse - u^k_\coarse}.
\end{split}
\end{align}
	Since $0 < \delta < 2\alpha/L[6M]^2$, it follows that
	\begin{align*}
		0 \,&\stackrel{\mathclap{\eqref{eq:equivalence}}}{\le} \,\EE(u^{k+1}_{\coarse}) - \EE(u^{\exact}_{\coarse}) = \EE(u^{k}_{\coarse}) - \EE(u^{\exact}_{\coarse}) - \big[\EE(u^{k}_{\coarse}) - \EE(u^{k+1}_{\coarse}) \big] \\
		&\stackrel{\mathclap{\eqref{eq:f4}}}{\le} \, \EE(u^{k}_{\coarse}) - \EE(u^{\exact}_{\coarse}) - \Big( \frac{1}{\delta} - \frac{L[6M]}{2} \Big)\,\enorm{u^{k+1}_{\coarse} - u^{k}_{\coarse}}^2\, \\
		&  \stackrel{\mathclap{\eqref{eq:helper:equiv}}}{\le} \, \EE(u^{k}_{\coarse}) - \EE(u^{\exact}_{\coarse}) - \Big( \frac{1}{\delta} - \frac{L[6M]}{2} \Big) \, \delta^2  \alpha^2\,\enorm{u^\exact_{\coarse} - u^{k}_{\coarse}}^2 \\
		&   \stackrel{\mathclap{\eqref{eq:equivalence}}}{\le} \, \Big[ 1 - \Big( 1  -\frac{\delta L[6M]}{2}\Big)\, \frac{2  \delta  \alpha^2}{L[3M]}\Big]\,\big[\EE(u^{k}_{\coarse}) -  \EE(u^{\exact}_{\coarse})\big],
	\end{align*}
where~\eqref{eq:equivalence} holds due to~\eqref{eq1:cor:zarantonello} from Corollary~\ref{cor1:zarantonello}.
This proves~\eqref{eq:energycontr}. The inequalities~\eqref{eq:energycontr:triangle} follow from the triangle inequality. This concludes the proof.
\end{proof}

	\begin{remark}\label{remark:delta}
		For a globally Lipschitz continuous $\AA$ with Lipschitz constant $L$, we observe that the energy contraction factor is minimal for $\delta = 1/L$, where $\qE[\delta]^2 = 1- \tfrac{\alpha^2}{L^2}$. In contrast, the optimal norm contraction factor $\qN[\delta]^2 = 1 - \tfrac{\alpha^2}{L^2}$ is obtained for $\delta = \tfrac{\alpha}{L^{2}}$; cf. Proposition~\ref{prop:zarantonello}. To allow a larger damping parameter $\delta > 0$, energy contraction is prefered.
	\end{remark}

\subsection{Mesh refinement}
\label{subsection:mesh-refinement}
From now on, let $\TT_0$ be a given conforming triangulation of the polyhedral Lipschitz domain $\Omega \subset \R^d$ with $d \ge 1$. For mesh refinement, we employ newest vertex bisection (NVB) for $d \ge 2$ (see, e.g.,~\cite{stevenson2008}), or the 1D bisection from~\cite{affkp2013} for $d=1$. 
For each triangulation $\TT_\coarse$ and a set of marked elements $\MM_\coarse \subseteq \TT_\coarse$, let $\TT_\fine \coloneqq \refine(\TT_\coarse,\MM_\coarse)$ be the coarsest triangulation such that all $T \in \MM_\coarse$ have been refined, i.e., $\MM_\coarse \subseteq \TT_\coarse \backslash \TT_\fine$. 
We write $\TT_\fine \in \T(\TT_\coarse)$, if $\TT_\fine$ results from $\TT_\coarse$ by finitely many steps of refinement.
To abbreviate notation, let $\T\coloneqq\T(\TT_0)$.

Throughout, each triangulation $\TT_\coarse \in \T$ is associated with a conforming finite-di\-men\-sion\-al space $\XX_\coarse \subset \XX$, and we suppose that mesh refinement $\TT_\fine \in \T(\TT_\coarse)$ implies nestedness $\XX_\coarse \subseteq \XX_\fine \subset \XX$.
\subsection{Axioms of adaptivity and \textsl{a~posteriori} error estimator}\label{subsection:axioms}
For $\TT_\coarse \in \T$ and $v_\coarse \in \XX_\coarse$, let
\begin{align}
	\begin{split}\label{eq:estimator:generic}
		\eta_\coarse(T, \cdot)\colon \XX_\coarse \to \R_{\ge 0}\quad \text{ for all } T \in \TT_\coarse 
	\end{split}
\end{align}
be the local contributions of an \textsl{a~posteriori} error estimator 
\begin{equation*}
	\eta_\coarse(v_\coarse) \coloneqq \eta_\coarse(\TT_\coarse, v_\coarse), \text{ where } 
	\eta_\coarse(\UU_\coarse, v_\coarse)
	\coloneqq
	\Big( \sum_{T \in \,\UU_\coarse} \eta_\coarse(T, v_\coarse)^2 \Big)^{1/2}  \text{ for all } \UU_\coarse \subseteq \TT_\coarse.
\end{equation*}
We suppose that the error estimator $\eta_\coarse$ satisfies the following axioms of adaptivity from~\cite{axioms} with a slightly relaxed variant of stability~\eqref{assumption:stab} from~\cite{bbimp2022}.

\renewcommand{\theenumi}{{A\arabic{enumi}}}
\begin{enumerate}
	\bf
	\item\label{assumption:stab} stability: \rm 
	For all $\vartheta > 0$ and all\footnote{While~\cite[Proposition~15]{bbimp2022} states stability only for $\TT_\fine \cap \TT_\coarse$, the inspection of the proof reveals that indeed arbitrary subsets $\UU_\coarse \subseteq \TT_\fine \cap \TT_\coarse$ are admissible.} $\UU_\coarse \subseteq \TT_\fine \cap \TT_\coarse$, there exists $\Cstab[\vartheta] >0$ such that for all $v_\fine \in \XX_\fine$ and $v_\coarse \in \XX_\coarse$ with $\max\big\{\enorm{v_\fine}, \enorm{v_\fine-v_\coarse}\big\}  \le \vartheta$, it holds that 
	\begin{align*}
		\big| \eta_\fine(\UU_\coarse, v_\fine) - \eta_\coarse(\UU_\coarse, v_\coarse) \big|
		&\le \Cstab[\vartheta]\, \enorm{v_\fine - v_\coarse}.
	\end{align*}
	\bf
	\item\label{assumption:red} reduction: \rm
	With $0 < \qred <1$, it holds that 
	\begin{align*}
		\eta_\fine(\TT_\fine \backslash \TT_\coarse, v_\coarse) 
		&\le \qred \, \eta_\coarse(\TT_\coarse \backslash \TT_\fine, v_\coarse) \quad \text{ for all } v_\coarse \in \XX_\coarse.
	\end{align*}
	\bf
	\item\label{assumption:rel} reliability: \rm
	There exists $\Crel >0$ such that
	\begin{align*}
		\enorm{u^\exact - u_\coarse^\exact} 
		&\le \Crel \, \eta_\coarse(u_\coarse^\exact).
	\end{align*} 
	\bf
	\item\label{assumption:drel} discrete reliability: \rm
	There exists $\Cdrel >0$ such that
	\begin{align*}
		\enorm{u_\fine^{\exact} - u_\coarse^{\exact}} 
		&\le \Cdrel \, \eta_\coarse(\TT_\coarse \backslash \TT_\fine, u_\coarse^{\exact})	. 
	\end{align*} 
\end{enumerate}

\subsection{Idealized adaptive algorithm}
\label{subsection:idealized}

In the following, we formulate and analyze an AILFEM algorithm in the spirit of~\cite{ghps2021}, but with an extended stopping criterion in Algorithm~\ref{algorithm:idealized}{\eqref{eq:stoppingcrit}}, i.e.,
	\begin{align}\label{eq:stoppingcrit}
		|\EE(u^{k-1}_\ell) - \EE(u^{k}_\ell)| \le  \lambda^2 \, \eta_\ell(u^k_\ell)^{2} \quad \land \quad \enorm{u^k_\ell} \le 2M. \tag{\rm i.b}
	\end{align}
Clearly, if the stopping criterion from Algorithm~\ref{algorithm:idealized}{\rm \eqref{eq:stoppingcrit}} holds, then also the simpler stopping criterion $|\EE(u^{k-1}_\ell) - \EE(u^{k}_\ell)| \le  \lambda^2\, \eta_\ell(u^k_\ell) $ from~\cite[Algorithm~2]{ghps2021} holds.

The proposed algorithm is idealized in the sense that an appropriate parameter $\delta >0$ is chosen \textsl{a priori}; see Theorem~\ref{theorem:rates} below.
\begin{algorithm}[idealized AILFEM with energy contraction]\label{algorithm:idealized}
	
	\textbf{Input:} initial triangulation $\TT_0$, initial guess $u_0^0 \coloneqq 0$ with $M =\frac{1}{\alpha} \norm{F - \AA 0}{\XX^\prime} < \infty$ according to~\eqref{eq:exact:bounded}, marking parameters $0 < \theta \le 1$ and $1 \le \Cmark< \infty$, damping parameter $\delta>0$, solver parameter $\lambda > 0$.
	
	\textbf{Loop:} For $\ell = 0, 1, 2, \dots$, repeat the following steps {\rm(i)--(iv)}:
	\begin{itemize}
		\item[\rm(i)] For all $k = 1, 2, 3, \dots$, repeat the following steps {\rm(a)--(b)}:
		\begin{itemize}
			\item[\rm(a)] Compute $u_\ell^k \coloneqq \Phi_\ell(\delta; u_\ell^{k-1})$ and $\eta_\ell(T, u_\ell^k)$ for all $T \in \TT_\ell$.
			\item[\rm(b)] Terminate $k$-loop if \quad $\big( |\EE(u^{k-1}_\ell) - \EE(u^{k}_\ell)| \le  \lambda^2\, \eta_\ell(u^k_\ell)^{2} \quad \land \quad \enorm{u^k_\ell} \le 2M \big)$.
		\end{itemize}
		\item[\rm(ii)] Upon termination of the $k$-loop, define $\kMAXL \coloneqq k$.
		\item[\rm(iii)] Determine a set $\MM_\ell \subseteq \TT_\ell$ with up to the multiplicative factor $\Cmark$ minimal cardinality such that $\theta \, \eta_\ell(u_\ell^{\kMAXL})^2 \le \sum_{T \in \MM_\ell} \eta_\ell(T, u_\ell^{\kMAXL})^2$.
		\item[\rm(iv)] Generate $\TT_{\ell+1} \coloneqq \refine(\TT_\ell, \MM_\ell)$ and define $u_{\ell+1}^0 \coloneqq u_\ell^{\kMAXL}$.
	\end{itemize}
\end{algorithm}
Following~\cite{ghps2021}, the analysis of Algorithm~\ref{algorithm:idealized} requires the ordered index set 
\begin{align}
\QQ \coloneqq \set{(\ell, k) \in \N_0^2}{\text{ index pair }(\ell, k) \text{ occurs in Algorithm~\ref{algorithm:idealized} and } k < \kmax(\ell)},
\end{align}
where $\kmax(\ell) \ge 1$ counts the number of solver steps for each $\ell$. The pair $(\ell, \kmax(\ell))$ is excluded from $\QQ$, since either $(\ell + 1, 0) \in \QQ$ and $u^0_{\ell + 1} = u^{\kmax(\ell)}_\ell$ or even $\kmax(\ell) \coloneqq \infty$ if the $k$-loop does not terminate after finitely many steps. Since Algorithm~\ref{algorithm:idealized} is sequential, the index set $\QQ$ is lexicographically ordered: For $(\ell, k)$ and $(\ell^\prime, k^\prime) \in \QQ$, we write $(\ell^\prime, k^\prime) < (\ell, k)$ if and only if $(\ell^\prime, k^\prime)$ appears earlier in Algorithm~\ref{algorithm:idealized} than $(\ell, k)$. Given this ordering, we define the \emph{total step counter}
\begin{align*}
|(\ell, k)| \coloneqq \#\set{(\ell^\prime, k^\prime) \in \QQ}{(\ell^\prime, k^\prime) < (\ell, k)} = k + \sum_{\ell' = 0}^{\ell - 1} \kmax(\ell^{\prime}),
\end{align*}
which provides the total number of solver steps up to the computation of $u^k_\ell$. 

Moreover, we define $\overline{\QQ} \coloneqq \QQ\,  \cup\, \set{(\ell, \kmaxl)}{ \ell \in \N_0 \text{ with } (\ell+1, 0) \in \QQ}$. Note that $\overline{\QQ} \subset \N_0 \times \N_0$ is a countably infinite index set such that, for all $(\ell,k) \in \N_0 \times \N_0$,
\begin{align*}
 (\ell+1,0) \in \overline{\QQ }
 &\quad \Longrightarrow \quad 
 (\ell, \kmax(\ell)) \in \overline{\QQ} \text{ and }\kmax(\ell) = \max \set{ k \in \N_0}{(\ell,k) \in \overline{\QQ}},
 \\
 (\ell, k+1) \in \overline{\QQ}
 &\quad \Longrightarrow \quad (\ell, k) \in \QQ.
\end{align*}
With $\lmax \coloneqq \sup \set{ \ell \in \N_0}{(\ell,0) \in \QQ}$, it then follows that either $\lmax = \infty$ or $\kmax(\lmax) = \infty$. From now on and throughout the paper, we employ the abbreviations $(\ell, \kmax) \coloneqq (\ell, \kmaxl)$ and $u^\kmax_{\ell} \coloneqq u^\kmaxl_{\ell}$.

\begin{corollary}\label{cor:crucial}
	Suppose that $\AA$ satisfies~\eqref{eq:strongly-monotone},~\eqref{eq:locally-lipschitz}, and~\eqref{eq:potential}.  Suppose the axioms of adaptivity~\eqref{assumption:stab}--\eqref{assumption:rel}. Let $\lambda > 0$ and $0 < \theta \le 1$ be arbitrary. Then, there exists a choice of the parameter $\delta  >0$ in Algorithm~\ref{algorithm:idealized} such that there exist $0 < \qN < 1$ and $0 < \qE < 1$ such that the following properties hold: 
	\begingroup\abovedisplayskip=0pt
	\begin{itemize}[leftmargin=0.75cm]
	\rm {\item[$\triangleright$]{\makebox[4cm]{\bf nested iteration:\hfill}}}	\vspace{-\baselineskip}
		\begin{flalign}\label{crucial:nestediteration} \vspace*{-\abovedisplayskip}
		&& \mathmakebox[6.5cm]{\enorm{u^0_\ell}  \le 2M}  \quad \text{ for all } (\ell, 0) \in \QQ;
		\end{flalign}
	\end{itemize}
	\begin{itemize}[leftmargin=0.75cm]
		\rm {\item[$\triangleright$]{\makebox[4cm]{\bf boundedness:\hfill}}} \vspace{-\baselineskip}
		\begin{flalign}\label{crucial:mbounded} 
		&&   \mathmakebox[6.5cm]{\enorm{u_\ell^k} \le 4M} \quad\text{ for all } (\ell, k) \in \QQ;
		\end{flalign}
		\vspace{-\baselineskip}
	\end{itemize}
\begin{itemize}[leftmargin=0.75cm]
	\rm {\item[$\triangleright$]{\makebox[4cm]{\bf norm contraction:\hfill}}}	\vspace{-\baselineskip}
	\begin{flalign}\label{crucial:contractivity} 
		&&\mathmakebox[6.5cm]{\mkern15mu \enorm{u^\exact_\ell - u_\ell^{k+1}} \le \qN \, \enorm{u^\exact_\ell - u_\ell^{k}}} \quad \text{ for all } (\ell, k) \in \QQ;
		\end{flalign}
\end{itemize}
\begin{itemize}[leftmargin=0.75cm]
	\rm {\item[$\triangleright$]{\makebox[4cm]{\bf energy contraction:\hfill}}}	\vspace{-\baselineskip}
\begin{flalign}\label{crucial:energycontractivity} 
	&&\mathmakebox[6.5cm]{\mkern27mu \EE(u^{k+1}_\ell) - \EE(u_\ell^{\exact}) \le \qE^{2} \big[ \EE(u^{k}_\ell) - \EE(u_\ell^{\exact})\big] }\quad \text{ for all } (\ell, k) \in \QQ.
	\end{flalign}
\end{itemize}
\endgroup
\noindent	Moreover, this guarantees~\eqref{eq:contr}--\eqref{eq1:cor:zarantonello} for all $(\ell, k) \in \QQ$ with $\qN[\delta]$ replaced by $\qN$. Furthermore, there exists $k_0 \in \N_0$ such that $\enorm{u^k_\ell} \le 2M$ for all $(\ell, k) \in \QQ$ with $k \ge k_0$.
\end{corollary}
\begin{proof}
	Let $0 < \delta < 2\alpha/L[6M]^2$ be arbitrary but fixed. From Algorithm~\ref{algorithm:idealized} and $u^0_0 \coloneqq 0$, we have that $\enorm{u^0_{\ell}} \le 2M$. Then, $\enorm{u^\exact_{\ell} - u^0_{\ell}} \le 3M$. Choose $0 < \qN \coloneqq \qN[\delta] < 1$ according to Proposition~\ref{prop:zarantonello}, where $\vartheta = 3M$ as well as $0 < \qE \coloneqq \qE[\delta] <1$ according to Proposition~\ref{proposition:econtraction}. This proves norm contraction~\eqref{crucial:contractivity} as well as energy contraction~\eqref{crucial:energycontractivity} for all $({\ell}, k) \in \QQ$. Furthermore, for all $({\ell}, k) \in \QQ$, it follows that 
		\begin{align}\label{eq:normcrit}
			\enorm{u^k_\ell}  &\le  \enorm{u^\exact_\ell} + \enorm{u^\exact_\ell - u^k_\ell} \stackrel{\mathclap{\eqref{eq1:cor:zarantonello}}}{\le} M + \qN^{k} \,\enorm{u^\exact_\ell - u^0_\ell}\le  M + \qN^{k} \, 3M \le 4M,
		\end{align}
	which proves boundedness~\eqref{crucial:mbounded}. Moreover,~\eqref{eq:normcrit} together with $0 < \qN < 1$ from~\eqref{crucial:contractivity} proves that there exists $k_0 \in \N_0$, which is independent of $\ell$, such that, for all $k \ge k_0$, it holds that
	\begin{align*}
		\enorm{u^k_{\ell}}  & \stackrel{\mathclap{\eqref{eq:normcrit}}}{\le}  M + \qN^{k} \, 3M \stackrel{!}{\le} 2M.
	\end{align*}
This shows for $(\ell,0) \in \QQ$ that the stopping criterion~$\enorm{u^k_\ell} \le 2M$ is met for all $(\ell,k) \in \QQ$ with $k \ge k_0$. This concludes the proof.
\end{proof}

\subsection{\texorpdfstring{AILFEM under the assumption of energy contraction~\eqref{crucial:energycontractivity}}{AILFEM under the assumption of energy contraction~(EC)}}\label{subsection:zarantonello:energy}
Norm contraction~\eqref{crucial:contractivity} is the critical ingredient in the proof of Corollary~\ref{cor:crucial} --- leading to boundedness (Corollary~\ref{cor1:zarantonello}), which is key to the proof of energy contraction~\eqref{crucial:energycontractivity} (cf.~\eqref{eq:critical}). Thus, norm contraction~\eqref{crucial:contractivity} is sufficient for obtaining nested iteration~\eqref{crucial:nestediteration}, boundedness~\eqref{crucial:mbounded}, and energy contraction~\eqref{crucial:energycontractivity}. However, supposing~\eqref{crucial:energycontractivity} already suffices to obtain uniform constants in the energy norm as the next result shows. Thus, throughout the rest of this paper, we suppose that energy contraction~\eqref{crucial:energycontractivity} holds for all $(\ell, k) \in \QQ$. 
\begin{lemma}
		\label{cor:zarantonello:termination}
	Suppose that $\AA$ satisfies~\eqref{eq:strongly-monotone},~\eqref{eq:locally-lipschitz}, and~\eqref{eq:potential}. Suppose that the choice of $\delta  >0$ guarantees that Algorithm~\ref{algorithm:idealized} satisfies~energy contraction~\eqref{crucial:energycontractivity}. Then, it holds that
	\begin{subequations}\label{eq:tau}
	\begin{equation}\label{eq:tau:def}
		\enorm{u^k_{\ell}} \le M + 3M \, \frac{L[3M]}{\alpha} \eqqcolon \frac{\tau}{2} \quad \text{ for all } (\ell, k) \in \QQ.
	\end{equation}
Moreover, it holds that
\begin{equation}\label{eq:uniform:const}
	\enorm{u^{k}_{\ell} - u^{k'}_{\ell}}  \le \tau \quad \text{ for all } ({\ell},k), ({\ell}, k') \in \QQ.
\end{equation}
\end{subequations}
Furthermore, there exists $k_0 \in \N_0$, which is independent of $\ell$, such that
\begin{align}
	\enorm{u^{k}_{\ell}} \le 2M \quad \text{ for all } (\ell, k) \in \QQ \text{ with } k \ge k_0.
\end{align} %
\end{lemma}
\begin{proof}
From Algorithm~\ref{algorithm:idealized} and $u^0_0 \coloneqq 0$, we have that $\enorm{u^0_{\ell}} \le 2M$. With $\enorm{u^\exact_\ell} \le M$ from~\eqref{eq:exact:bounded}, it holds that $\enorm{u^\exact_{\ell} - u^0_{\ell}} \le 3M$. For all $({\ell}, k) \in \QQ$, it follows that 
\begin{align}
	\enorm{u^k_{\ell}} & 
	\le \enorm{u^\exact_{\ell}} + \enorm{u^\exact_{\ell} - u^{k}_{\ell}} \stackrel{\eqref{eq:equivalence}}{\le} M + \Big( \frac{2}{\alpha}\Big)^{1/2} \, \big( \EE(u^k_{\ell})- \EE(u^\exact_{\ell})\big)^{1/2}  \notag
	\\& \stackrel{\mathclap{\eqref{crucial:energycontractivity}}}{\le}\, \, M + \qE^{k}\, \Big( \frac{2}{\alpha}\Big)^{1/2}  \, \big( \EE(u^0_{\ell})- \EE(u^\exact_{\ell})\big)^{1/2}  \stackrel{\eqref{eq:equivalence}}{\le}  M + \qE^{k}\, 3M \, \Big( \frac{L[3M]}{\alpha}\Big)^{1/2}  
	\label{eq:inductivearg:nested}\\ & 
	\stackrel{\mathclap{\eqref{crucial:energycontractivity}}}{\le}\, \, M + 3M \Big(\frac{L[3M]}{\alpha} \Big)^{1/2} \eqqcolon \frac{\tau}{2}.
	\label{eq:inductivearg:uniform} 
\end{align}
This and the triangle inequality prove~\eqref{eq:uniform:const}. Moreover, inequality~\eqref{eq:inductivearg:nested} together with $0 < \qE < 1$ from energy contraction~\eqref{crucial:energycontractivity} proves that there exists $k_0 \in \N_0$, which is independent of $\ell$, such that
\begin{align}\label{eq:normcrit:ec}
	\enorm{u^{k}_{\ell}} \stackrel{\eqref{eq:inductivearg:nested}}{\le} M +  \qE^{k} \,3M \, \Big( \frac{L[3M]}{\alpha}\Big)^{1/2} \stackrel{!}{\le} 2M \quad \text{ for all } (\ell, k) \in \QQ \text{ with } k \ge k_0.
\end{align}
This concludes the proof.
\end{proof}
	\begin{remark}\label{remark:stopping}
{\rm{(i)}}~From Lemma~\ref{cor:zarantonello:termination}, we infer that the stopping criterion can fail only finitely many times due to the energy norm criterion $\enorm{u^k_\ell} \le 2M$. 

{\rm{(ii)}}~Under the assumption of energy contraction~\eqref{crucial:energycontractivity}, we note that~\eqref{eq:uniform:const} shows that $\tau$ provides a uniform upper bound for the involved stability and Lipschitz constants $\Cstab[\tau]$ and $L[\tau]$, respectively. Indeed, it will become apparent later that stability and local Lipschitz continuity will only be exploited for the differences $\enorm{u^k_\coarse- u^{k-1}_\coarse}$, $\enorm{u^\exact_\coarse-u^{k+1}_\coarse}$, or $\enorm{u^\exact- u^\exact_\coarse}$ in~\eqref{assumption:stab},~\eqref{eq:equivalence}, and~\eqref{eq:f4}.
\end{remark}

\subsection{Main results}\label{subsection:mainres}
Given the Pythagoras identity~\eqref{eq:qo:abstract} and energy contraction~\eqref{crucial:energycontractivity}, the first main theorem states full linear convergence of the quasi-error
\begin{align}\label{eq:quasi-error}
	\Delta_\ell^k	\coloneqq \enorm{u^\exact - u_\ell^{k}} + \eta_\ell(u_\ell^{k}).
\end{align}

\begin{theorem}[full linear convergence]\label{theorem:fulllinear}
	Suppose that $\AA$ satisfies~\eqref{eq:strongly-monotone},~\eqref{eq:locally-lipschitz}, and \eqref{eq:potential}. Suppose the axioms of adaptivity~\eqref{assumption:stab}--\eqref{assumption:rel} and orthogonality~\eqref{eq:qo:abstract}, where $\XX_\coarse$ is understood as $\XX_\ell$ for $(\ell,k) \in \QQ$. Let $0 < \theta \le 1$, $1 \le \Cmark \le \infty$, and $\lambda > 0$. Suppose that the choice of $\delta  >0$ guarantees that Algorithm~\ref{algorithm:idealized} satisfies energy contraction~\eqref{crucial:energycontractivity}.
	Then, there exist $\Clin > 0$ and $0 < \qlin < 1$ such that Algorithm~\ref{algorithm:idealized} leads to
	\begin{align}\label{eq:linear}
		\Delta_\ell^k \le \Clin \qlin^{|(\ell, k)|- |(\ell^\prime, k^\prime)|} \, \Delta^{k^\prime}_{\ell^\prime} \quad\text{ for all } (\ell, k), (\ell^\prime, k^\prime)  \in \QQ \text{ with } (\ell^\prime, k^\prime) < (\ell, k).
	\end{align}
	The constants $\Clin$ and $\qlin$ depend only on $M$, $L[\tau/2]$, $\alpha$, $\Cstab[\tau]$, $\qred$, $\Crel$, and $\qE$ as well as on the adaptivity parameters $0 < \theta \le 1$ and $\lambda > 0$. \hfill \qed
\end{theorem}
The proof of Theorem~\ref{theorem:fulllinear} extends that of~\cite[Theorem~4]{ghps2021}, since the stopping criterion from Algorithm~\ref{algorithm:idealized}{\eqref{eq:stoppingcrit}} requires further analysis to cover all cases. 
	To ease notation, we introduce the shorthand
	\begin{align*}
		\dist(v, w)^2 = | \EE(v) - \EE(w)| \quad \text{ for all} v, w \in \XX.
	\end{align*}	
	The following lemma provides the essential step in the proof of Theorem~\ref{theorem:fulllinear}. 
\begin{lemma}\label{prop:lambda}
	Under the assumptions of Theorem~\ref{theorem:fulllinear}, there exist constants $\mu > 0$ and $0 < \qlin < 1$ such that 
	\begin{align}\label{eq:lambda}
		\Lambda_\ell^k \coloneqq \dist(u^\star, u_\ell^k)^{2} + \mu \, \eta_\ell(u_\ell^k)^2
		\quad \text{for all } (\ell, k) \in \QQ
	\end{align}
	satisfies the following statements~{\rm(i)}--{\rm(ii)}:
	\begin{itemize}
		\item[\rm(i)] $\Lambda_\ell^{k+1} \le \qlin^2 \, \Lambda_\ell^k$ \quad  \kern10pt for all $(\ell,k+1) \in \QQ$.
		\item[\rm(ii)] $\Lambda_{\ell+1}^0 \mkern1mu \le \qlin^2 \, \Lambda_\ell^{\k-1}$ \quad  for all $(\ell+1,0) \in \QQ$.
	\end{itemize}
	The constants $\mu$ and $\qlin$ depend only on $M$, $L[2M]$, $\alpha$, $\Cstab[\tau]$, $\qred$, $\Crel$, and $\qE$ as well as on the adaptivity parameters $0 < \theta \le 1$ and $\lambda > 0$.
\end{lemma}
\begin{proof}
For $k\in \N$ such that $1 \le k \le \kmaxl$, the stopping criterion of Algorithm~\ref{algorithm:idealized}{\eqref{eq:stoppingcrit}}, i.e.,
		\begin{align}\tag{i.b} \label{eq:fulllinear:stoppingcrit}
		\dist(u^{k-1}_\ell, u_\ell^k)^{2} = |\EE(u^{k-1}_\ell) - \EE(u^{k}_\ell)| \le  \lambda^2 \, \eta_\ell(u^k_\ell)^{2} \quad \land \quad \enorm{u^k_\ell} \le 2M,
		\end{align}
		comprises four cases. Statement~{\rm(i)} contains the cases $\mathtt{true \land false}$, $\mathtt{false \land false}$, and $\mathtt{false \land true}$. Statement~{\rm(ii)} consists of the remaining case $\mathtt{true \land true}$.
		
 {\bfseries Case 1: Evaluation of~(\ref{eq:fulllinear:stoppingcrit}) returns $\mathtt{true \land false}$.}  This case investigates~\eqref{eq:fulllinear:stoppingcrit} for $k +1 < \kmaxl$. First, we note that
\begin{align*}
\enorm{u^\star - u_\ell^\star}^2
\eqreff{assumption:rel}\le \Crel^2 \, \eta_\ell(u_\ell^\star)^2 \, \quad
\stackrel{\mathclap{\eqref{assumption:stab}, \eqref{eq:tau:def}}}\le \,  \quad 2 \, \Crel^2 \, \eta_\ell(u_\ell^{k+1})^2
+ 2 \, \Crel^2 \, \Cstab^2[\tau] \, \enorm{u_\ell^\star - u_\ell^{k+1}}^2.
\end{align*}
Together with~\eqref{eq:equivalence}, this leads us to
\begin{align*}
\nonumber
\dist(u^\star, u_\ell^\star)^{2}
\eqreff{eq:equivalence}\le \frac{L[2M]}{2} 
\enorm{u^\star - u_\ell^\star}^2
&\eqreff{eq:equivalence}\le C_1\,\eta_\ell(u_\ell^{k+1})^2 
+ C_2 \,  \dist(u_\ell^\star, u_\ell^{k+1})^{2},
\end{align*}
where we define $C_1 \coloneqq L[2M] \Crel^2$ and $C_2 \coloneqq 2\,  \alpha^{-1} L[2M]\Crel^2 \Cstab^2[\tau]$. For $0 < \varepsilon < 1$, we obtain that
\begin{align*}
\dist(u^\star, u_\ell^{k+1})^{2}
&\stackrel{\mathclap{\eqref{eq:qo:abstract}}}= 
(1-\eps) \, \dist(u^\star, u_\ell^\star)^{2} + \eps \, \dist(u^\star, u_\ell^\star)^{2}
+ \dist(u_\ell^\star, u_\ell^{k+1})^{2}
\\& 
\le\, 
(1-\eps) \, \dist(u^\star, u_\ell^\star)^{2}
+ \eps \, C_1 \, \eta_\ell(u_\ell^{k+1})^2 + (1 + \eps \, C_2) \, \dist(u_\ell^\star, u_\ell^{k+1})^{2}
\\&
\stackrel{\mathclap{\eqref{crucial:energycontractivity}}}\le 
(1-\eps) \, \dist(u^\star, u_\ell^\star)^{2}
+ \eps \, C_1 \, \eta_\ell(u_\ell^{k+1})^2 + (1 + \eps \, C_2) \, \qE^2 \, \dist(u_\ell^\star, u_\ell^k)^{2}.
\end{align*}
We use the last inequality for the quasi-error $\Lambda_{\ell}^{k+1}$ to obtain that
\begin{align}\label{eq:full-linear:main}
\Lambda_\ell^{k+1} 
&= \dist(u^\star, u_\ell^{k+1})^{2} + \mu \, \eta_\ell(u_\ell^{k+1})^2 
\notag \\& 
\le (1-\eps) \, \dist(u^\star, u_\ell^\star)^{2}
+ (\mu + \eps \, C_1) \, \eta_\ell(u_\ell^{k+1})^2 + (1 + \eps \, C_2) \, \qE^2 \, \dist(u_\ell^\star, u_\ell^k)^{2}.
\end{align}
We need four auxiliary estimates: 

First, 
	since $\enorm{u^\exact_\ell} \le M$ and $\enorm{u^\exact_\ell - u^\exact_0} \le 2M$ hold independently of $\ell$, the axioms~\eqref{assumption:stab}--\eqref{assumption:rel} and Proposition~\ref{prop:existence} imply quasi-monotonicity of the estimators, i.e., 
\begin{align}\label{eq:estimator:mon}
	\eta_\ell(u^\exact_\ell) \le \Cmon \eta_0(u^\exact_0) \quad \text{ with } \quad \Cmon = \big(2 + 8 \Cstab[2M]^{2}(1+\Ccea^2) \Crel^2\big)^{1/2};
\end{align}
cf.~\cite[Lemma~3.6]{axioms}. With ${C_0} \coloneqq \Cmon \max\{1, \Cstab[M] M\}$, we infer that
\begin{align}\label{eq:estimator:mon:aux}
	\eta_\ell(u^\exact_\ell) \stackrel{\eqref{eq:estimator:mon}}{\le} \Cmon \eta_0(u^\exact_0) \stackrel{\eqref{assumption:stab}}{\le} \Cmon \eta_0(0) + \Cmon \Cstab[M] \enorm{u^\exact_0} \le C_0(\eta_0(0) + 1). 
\end{align}

Second, with $C_3 \coloneqq2 \,C_0 (\eta_0(0) + 1)$ and $C_4 \coloneqq 4 \, \alpha^{-1} \,  \Cstab[\tau]^2 \, \qE^2$, it holds that
\begin{align}\label{eq:full-linear:aux1}
\eta_\ell(u_\ell^{k+1})^2
& \,	\stackrel{\mathclap{\eqref{assumption:stab}}}\le 2 \,  \eta_\ell(u_\ell^{\exact})^2
+ 2\,  \Cstab[\tau]^2 \, \enorm{u_\ell^{\exact}-u_\ell^{k+1}}^2 \, \stackrel{\mathclap{\eqref{eq:equivalence}}}\le 2 \,  \eta_\ell(u_\ell^{\exact})^2
+ \frac{4}{\alpha} \,  \Cstab[\tau]^2 \,  \dist(u_\ell^\star, u_\ell^{k+1})^{2} \notag \\
&\, \stackrel{\mathclap{\eqref{crucial:energycontractivity}}}\le \,2\,  \eta_\ell(u_\ell^{\exact})^2 + \frac{4}{\alpha} \,  \Cstab[\tau]^2 \, \qE^2\, \dist(u_\ell^\star, u_\ell^{k})^{2} \, \stackrel{\mathclap{\eqref{eq:estimator:mon:aux}}}\le C_3 + C_4 \,\dist(u_\ell^\star, u_\ell^{k})^{2}.
\end{align}

Third, the error estimator allows for the following estimate with an arbitrary but fixed Young parameter $0 < \gamma < 1$:
\begin{align}\label{eq:full-linear:aux2}
\eta_\ell(u_\ell^{k+1})^2
\,& 	\stackrel{\mathclap{\eqref{assumption:stab}}}\le \, (1+ \gamma) \,  \eta_\ell(u_\ell^{k})^2
+ (1+ \gamma^{-1}) \,  \Cstab[\tau]^2 \, \enorm{u_\ell^{k+1} - u_\ell^{k}}^2 \notag \\
&\le (1+ \gamma) \,  \eta_\ell(u_\ell^{k})^2
+ 2 \,(1+ \gamma^{-1}) \,  \Cstab[\tau]^2 \, \big[ \enorm{ u_\ell^{\exact}-u_\ell^{k+1}}^2 + \enorm{u_\ell^{\exact}-u_\ell^{k}}^2\big] \notag  \\
\,& 	\stackrel{\mathclap{\eqref{eq:equivalence}}}\le \, (1+ \gamma) \,  \eta_\ell(u_\ell^{k})^2
+ \frac{4}{\alpha} \,(1+ \gamma^{-1}) \,  \Cstab[\tau]^2 \, \big[ \dist(u_\ell^\star, u_\ell^{k+1})^{2} + \dist(u_\ell^\star, u_\ell^{k})^{2} \big] \notag  \\
\,& 	\stackrel{\mathclap{\eqref{crucial:energycontractivity}}}\le \, (1+ \gamma) \,  \eta_\ell(u_\ell^{k})^2
+C_5\,\dist(u_\ell^\star, u_\ell^{k})^{2},
\end{align}
where $C_5 \coloneqq 4\, \alpha^{-1} (1+ \gamma^{-1}) \,  \Cstab[\tau]^2 \, ( 1 + \qE^2 )$.

Fourth, we observe that the case $\mathtt{true \land false}$ yields that
\begin{align*}
2M < \enorm{u^{k+1}_\ell} \le  \enorm{u^\exact_\ell} +\enorm{u^\exact_\ell-u^{k+1}_\ell} \le M + \enorm{u^\exact_\ell- u^{k+1}_\ell}
\end{align*}
and hence $M < \enorm{u^\exact_\ell-u^{k+1}_\ell}$. With $C_6 \coloneqq 2 \, \alpha^{-1} M^{-2} \, \qE^2$, this observation leads us to
\begin{align}\label{eq:full-linear:aux3}
1 < \frac{\enorm{u^\exact_\ell-u^{k+1}_\ell}^{2}}{M^2} \, \stackrel{\eqref{eq:equivalence}}{\le} \,2 \, \alpha^{-1} M^{-2}\, \dist(u_\ell^\star, u_\ell^{k+1})^{2} \eqreff{crucial:energycontractivity}{\le} C_6 \, \dist(u_\ell^\star, u_\ell^{k})^{2}. 
\end{align}

Recall that $0< \varepsilon < 1$ and define $0 <\sigma \coloneqq \frac{\varepsilon+\gamma}{1+\gamma}<1$. This choice of $\sigma$ ensures that
\begin{align}\label{eq:full-linear:aux4}
(1-\sigma) \,(1+ \gamma) = 1-\varepsilon .
\end{align}
We apply these observations to the term $(\mu + \eps \, C_1) \, \eta_\ell(u_\ell^{k+1})^2$ of~\eqref{eq:full-linear:main} to arrive at
\begin{align}\label{eq:full-linear:mainhelp}
&(\mu + \varepsilon \, C_1) \, \eta_\ell(u_\ell^{k+1})^2 =  (1-\sigma) \,\mu \, \eta_\ell(u_\ell^{k+1})^2 + (\sigma\, \mu + \varepsilon \, C_1)\, \eta_\ell(u_\ell^{k+1})^2 \notag\\
& \mkern9mu \stackrel{\mathclap{\eqref{eq:full-linear:aux2}}}{\le} \, (1-\sigma) (1+ \gamma) \,\mu \,  \eta_\ell(u_\ell^{k})^2 
+ (1-\sigma) \,\mu \,  C_5 \, \dist(u_\ell^{\exact},u_\ell^{k})^2
+ (\sigma \, \mu+ \varepsilon \, C_1)\, \eta_\ell(u_\ell^{k+1})^2  \notag \\
&\mkern9mu  \stackrel{\mathclap{\eqref{eq:full-linear:aux1}}}{\le} \, (1-\sigma) (1+ \gamma) \,\mu \,  \eta_\ell(u_\ell^{k})^2 
+ (1-\sigma)\, \mu \,  C_5 \, \dist(u_\ell^{\exact},u_\ell^{k})^2
+ (\sigma \, \mu+ \varepsilon \, C_1) \big[C_3 + C_4 \,\dist(u_\ell^{\exact},u_\ell^{k})^2 \big] \notag \\
&\mkern9mu \stackrel{\mathclap{\eqref{eq:full-linear:aux3}}}{<} \, (1-\sigma) (1+ \gamma) \,\mu \,  \eta_\ell(u_\ell^{k})^2 
\! +\!(1-\sigma) \,\mu\, C_5 \, \dist(u_\ell^{\exact},u_\ell^{k})^2
\! +\! (\sigma \, \mu+ \varepsilon \, C_1) (C_3 C_6 + C_4) \,\dist(u_\ell^{\exact},u_\ell^{k})^2 \notag\\
&\mkern9mu \stackrel{\mathclap{\eqref{eq:full-linear:aux4}}}{=} \, (1-\varepsilon)\,\mu \,  \eta_\ell(u_\ell^{k})^2 
+ \big[ (1-\sigma) \, C_5 
+ \sigma \,C_7 \big] \mu \,\dist(u_\ell^{\exact},u_\ell^{k})^2 + \varepsilon\, C_1\,  C_7 \,\dist(u_\ell^{\exact},u_\ell^{k})^2,
\end{align}
where $C_7 \coloneqq C_3 C_6 + C_4$. 
Together with~\eqref{eq:full-linear:main}, we obtain that
\begin{align*}
\Lambda_\ell^{k+1} 
&\,\stackrel{\mathclap{\eqref{eq:full-linear:main}}}{\le} \, (1-\eps) \, \dist(u^\star, u_\ell^\star)^{2}
+ (\mu + \eps \, C_1) \, \eta_\ell(u_\ell^{k+1})^2 + (1 + \eps \, C_2) \, \qE^2 \, \dist(u_\ell^\star, u_\ell^k)^{2} \\
&\,\stackrel{\mathclap{\eqref{eq:full-linear:mainhelp}}}{\le} \, (1-\eps) \, \dist(u^\star, u_\ell^\star)^{2}
+ \, (1-\varepsilon)\,\mu \,  \eta_\ell(u_\ell^{k})^2 \\
& \quad \quad
+  \big\{ \big[(1-\sigma) \, C_5 
+ \sigma \,C_7 \big] \mu \,+ \varepsilon\, C_1\,  C_7 + (1 + \eps \, C_2) \, \qE^2  \big\} \, \dist(u_\ell^\star, u_\ell^k)^{2} \\
&\,\le \, (1-\eps) \, \dist(u^\star, u_\ell^\star)^{2}
+ \, (1-\varepsilon)\,\mu \,  \eta_\ell(u_\ell^{k})^2 \\
& \quad \quad
+  \big\{\mu \, \max{\{C_5, C_7\}} + \varepsilon\, C_1\,  C_7 + (1 + \eps \, C_2) \, \qE^2  \big\} \, \dist(u_\ell^\star, u_\ell^k)^{2}.
\end{align*}
Note that $C_1, \ldots, C_7$ depend only on the problem setting. Provided that
\begin{align}\label{eq:full-linear:condition:tf}
\mu \, \max{\{C_5, C_7\}} + \varepsilon\, C_1\,  C_7 + (1 + \eps \, C_2) \, \qE^2 \le 1 - \varepsilon,
\end{align}
we conclude that
\begin{align*}
\Lambda_\ell^{k+1} &\le (1-\eps) \, \big[ \dist(u^\star, u_\ell^\star)^{2} + \dist(u_\ell^\star, u_\ell^k)^{2}  + \mu \,  \eta_\ell(u_\ell^{k})^2\big] \\
&\stackrel{\mathclap{\eqref{eq:qo:abstract}}}{=} (1-\eps) \, \big[ \dist(u^\star, u_\ell^k)^{2} + \mu \,  \eta_\ell(u_\ell^{k})^2 \big]
= (1-\eps) \, \Lambda_\ell^k.
\end{align*}

{\bf Case 2: Evaluation of~{(\ref{eq:fulllinear:stoppingcrit})} returns $\mathtt{false \land false}$ or $\mathtt{false \land true}$.} These cases follow from the arguments found in~\cite[Lemma~10{\rm(i)}]{ghps2021}. There, the proof is based in essence on the estimate  
\begin{align*}
\eta_\ell(u_\ell^{k+1})^2 
< \lambda^{-2} \, \dist(u_\ell^{k+1}, u_\ell^{k})^2
\stackrel{\eqref{crucial:energycontractivity}}\le \lambda^{-2} \, (1+\qE^2) \, \dist(u_\ell^\star, u_\ell^{k})^2,
\end{align*}
to obtain an upper bound of the quasi-error $\Lambda_\ell^{k+1}$ in terms of the linearization error $\dist(u^\exact_\ell, u_\ell^k)^2$. With $C_8 \coloneqq \lambda^{-2}\, (1+\qE^2)$ and provided that
\begin{align}\label{eq:full-linear:condition:ff}
(\mu + \eps \, C_1) \, C_8 + (1 + \eps \, C_2) \, \qE^2 \le 1-\eps,
\end{align}
\cite[Lemma~10{\rm(i)}]{ghps2021} then proves that
\begin{align*}
	\Lambda_\ell^{k+1} \le (1-\eps) \Lambda_\ell^k.
\end{align*}
Up to the final choice of $\mu, \varepsilon > 0$, this concludes the proof of~these cases and~statement~{\rm (i)}.

{\bf Case 3: Evaluation of~(\ref{eq:fulllinear:stoppingcrit}) returns $\mathtt{true \land true}$.}  
The case $\mathtt{true \land true}$ is analyzed in~\cite[Lemma~10{\rm(ii)}]{ghps2021} and is based on the contractivity of the error estimator given that the Dörfler marking is employed.

Define $q_\theta \coloneqq \big( 1 - (1-\qred^2) \,\theta)$ and $C_9 \coloneqq 4 \,\alpha^{-1}(1 + \qE^2)\, \Cstab[\tau]^2$. Let $0< \omega<1$ be arbitrary. Note that $C_1, C_2, C_{9}>0$ and $0 < q_\theta < 1$ depend only on the problem setting. Provided that 
\begin{align}\label{eq:full-linear:condition:tt}
\eps \, C_1 \, \mu^{-1} + q_\theta \, (1 + \delta) \le 1-\eps
\quad \text{and} \quad
\eps \, C_2 + \qE^2 + \mu \, q_\theta \, (1 + \omega^{-1}) \,C_9 \le 1-\eps,
\end{align}
we obtain from~\cite[Lemma~10{\rm(ii)}]{ghps2021} that 
\begin{align*}
&\Lambda_{\ell+1}^0 
\le  
(1-\eps) \, \Lambda_\ell^{\k-1}.
\end{align*}
Up to the final choice of $\omega, \mu, \eps > 0$, this concludes the proof of Lemma~\ref{prop:lambda}(ii).

{\bf Choice of parameters.}
We proceed as follows:
\begin{enumerate}[label={\arabic*.}]
\item Choose $\omega > 0$ such that \ $(1+\omega) \, q_\theta < 1$.
\item Choose $\mu > 0$ such that $\qE^2 +  \mu\, \max\{C_5, C_7\} < 1$, $\qE^2+ \mu C_8 < 1$, \linebreak\ and \ $\qE^2 + \mu \, q_\theta(1+\omega)^{-1}  C_9 < 1$.
\item Finally, choose $\eps > 0$ sufficiently small such that~\eqref{eq:full-linear:condition:tf}--\eqref{eq:full-linear:condition:tt} are satisfied.
\end{enumerate}
This concludes the proof of Lemma~\ref{prop:lambda} with $\qlin^2 \coloneqq (1-\eps)$.
\end{proof}
\begin{proof}[Proof of Theorem~\ref{theorem:fulllinear}]
According to~\eqref{eq:equivalence}, it holds that 
$\Delta_\ell^k \simeq (\Lambda_\ell^k)^{1/2}$,
where the hidden constants depend only on $\mu$, $\alpha$, and $L[\tau/2]$. We use~\eqref{eq:equivalence} for the term $\enorm{u^\exact_\ell - u^k_\ell}$, and hence the dependency $L[\tau/2]$ is justified by~\eqref{eq:inductivearg:uniform}. Then, linear convergence~\eqref{eq:linear} follows from Lemma~\ref{prop:lambda} and induction, since the set $\QQ$ is linearly ordered with respect to the total step counter $|(\cdot,\cdot)|$.
\end{proof}
\begin{remark}
	
	{\rm{(i)}}~Provided that energy contraction~\eqref{crucial:energycontractivity} holds and that the adaptivity parameter $\lambda >0$ is sufficiently small, the stopping criterion 
	\begin{align}\label{eq:stopping:alternative}
		|\EE(u^{k-1}_\ell) - \EE(u^{k}_\ell)| \le  \lambda^2\, \eta_\ell(u^k_\ell)^{2} \tag{\rm i.b$'$}
	\end{align} 
	from~\cite{ghps2021} is a viable alternative to the stopping criterion of Algorithm~\ref{algorithm:idealized}{\eqref{eq:stoppingcrit}}. The main difficulty is to ensure nested iteration~\eqref{crucial:nestediteration}. This relies, in essence, on the estimate
	\begin{align*}	
		\frac{\alpha}{2}\,\enorm{u_\ell^\exact - u_\ell^\kk}^{2}
		\,\, & \stackrel{\mathclap{\eqref{eq:equivalence}}}{\le} \, \, \EE(u^\kk_\ell) - \EE(u^\exact_\ell)\quad \stackrel{\mathclap{\eqref{crucial:energycontractivity}, \,\eqref{eq:energycontr:triangle}}}{\le} \quad  \,   \frac{\qE[\delta]^2}{1-\qE[\delta]^2} \, \big[ \EE(u^{\kk-1}_\ell) - \EE(u^{\kk}_\ell) \big] \\
		&\mkern-54mu \stackrel{\eqref{eq:stopping:alternative}}{\le}  
		\frac{\qE[\delta]^2}{1-\qE[\delta]^2} \, \lambda^2 \, \eta_\ell(u_\ell^\kk)^2 \, \, \stackrel{\mathclap{\eqref{assumption:stab}}}{\le} 
		\, 2 \,   \frac{\qE[\delta]^2}{1-\qE[\delta]^2} \,  \lambda^2 \, \big[ \eta_\ell(u_\ell^\exact)^2\!+\! \Cstab[\tau/2]^2  \enorm{u_\ell^\exact - u_\ell^\kk}^2  \big],
	\end{align*}
where $\enorm{u^\exact_\ell - u^k_\ell} \le \tau/2$ stems from~\eqref{eq:inductivearg:uniform}.
Using a uniform estimate for the error estimator as in~\eqref{eq:estimator:mon:aux}, the last estimate, and the observation that $\enorm{u^\kmax_\ell} \le M + \enorm{u^\exact_\ell - u^\kmax_\ell}$ lead us to
\begin{align*}
	\enorm{u^0_{\ell+1}} = \enorm{u^\kmax_{\ell}} \le M + \lambda \, \frac{ r[\delta] \, C_0 \, (\eta_0(0) + 1)}{[1 -  \lambda^{2} \, r[\delta]^2\, \Cstab[\tau/2]^{2}]^{1/2}} \stackrel{!}{\le} 2M \quad \text{ with } r[\delta]^2 \coloneqq \frac{4}{\alpha} \,\frac{\qE[\delta]^2}{1-\qE[\delta]^2},
\end{align*}
where a sufficiently small $\lambda$ such that $\lambda^2 \,r[\delta]^2\, \Cstab[\tau/2]^{2} < 1$ is required and where $C_0 \coloneqq \Cmon \,\max\{ 1, \Cstab[M]\,M\}$. We see that a sufficently small $\lambda >0$ ensures nested iteration~\eqref{crucial:nestediteration}. In contrast,~\eqref{eq:stoppingcrit} leads to full linear convergence for arbitrary $\lambda > 0$.

	{\rm (ii)}
	Theorem~\ref{theorem:fulllinear} proves linear convergence, and hence in particular plain convergence $\Delta_\ell^k \to 0$ as $|(\ell, k)| \to \infty$. In Appendix~\ref{section:appendix}, it is shown that plain convergence also holds for Algorithm~\ref{algorithm:idealized} with the modified stopping criterion
		\begin{align}\tag{\rm i.b$''$}
	\enorm{u^{k}_\ell - u^{k+1}_\ell} \le \lambda \, \eta_\ell(u^k_\ell) \quad \land \quad \enorm{u^k_\ell} < 2M
		\end{align} 
	(instead of Algorithm~\ref{algorithm:idealized}{\eqref{eq:stoppingcrit}}) in the strongly monotone and locally Lipschitz continuous setting without~\eqref{eq:potential}. Due to the lack of an energy $\EE$, the result relies on norm contraction~\eqref{crucial:contractivity} instead of energy contraction~\eqref{crucial:energycontractivity}.
\end{remark}

To formulate our main result on optimal convergence rates, we need some additional notation.
For $N \in \N_0$, let $\T_N\coloneqq\set{\TT\in\T}{\#\TT-\#\TT_0\le N}$ denote the (finite) set of all refinements of $\TT_0$ which have at most $N$ elements more
than $\TT_0$. 
For $s>0$, we define
\begin{align}\label{eq:approxclass}
\norm{u^\exact}{\bbA_s} 
&\coloneqq \sup_{N\in\N_0} \Big((N+1)^s \min_{\TT_\opt\in\T_N} \big[ \enorm{u^\exact - u^\exact_\opt} + \eta_{\opt}(u^\exact_\opt) \big] \Big) \in \R_{\geq 0} \cup \{ \infty \}.
\end{align}
Here, $u^\exact_\opt \in \XX_\opt$ denotes the exact Galerkin solution~\eqref{eq:weakform:discrete} with respect to the optimal mesh $\TT_\opt$, where optimality is understood with respect to the quasi-error $\Delta_\opt^\exact$ from~\eqref{eq:quasi-error} (consisting of the energy norm error plus error estimator).
In explicit terms, $\norm{u^\exact}{\bbA_s} < \infty$ means that an algebraic convergence rate $\OO(N^{-s})$ for the quasi-error $\Delta_\opt^\exact$ is possible, if the optimal triangulations are chosen.

The second main theorem states optimal convergence rates of the quasi-error~\eqref{eq:quasi-error} with respect to the number of degrees of freedom. As usual in this context (see, e.g.,~\cite{axioms}), the result requires that the adaptivity parameters $0 < \theta \le 1$ and $\lambda > 0$ are sufficiently small. 
The proof is found in, e.g.,~\cite[Theorem~8]{ghps2021}. A careful inspection of the proof reveals that it requires only estimates of the form
\begin{align*}
\dist(u^{\kmax}_\ell, u^{\kmax-1}_\ell) \le \lambda \, \eta_\ell(u^\kmax_\ell),
\end{align*}
as well as linear convergence~\eqref{eq:linear}, which are satisfied for Algorithm~\ref{algorithm:idealized}. The results from~\cite{ghps2021} are proven for a uniform Lipschitz and stability constant; in the present setting, this follows from~Remark~\ref{remark:stopping}{\rm (ii)}.
\begin{theorem}[rate-optimality w.r.t.~degrees of freedom]\label{theorem:rates}
Suppose that $\AA$ satisfies~\eqref{eq:strongly-monotone},~\eqref{eq:locally-lipschitz}, and~\eqref{eq:potential} as well as the axioms of adaptivity~\eqref{assumption:stab}--\eqref{assumption:drel}. Suppose that the choice of $\delta  >0$ guarantees that Algorithm~\ref{algorithm:idealized} satisfies energy contraction~\eqref{crucial:energycontractivity}. Define
\begin{align}\label{def:lambda_opt}
\lambdaopt \coloneqq 
\frac{1 - \qE}{\qE \, \Cstab[\tau]}  \Big( \frac{\alpha}{2} \Big)^{1/2},
\end{align}
with $\tau$ from~\eqref{eq:tau}. Let $0 < \theta \leq 1$ and $0 < \lambda < \lambdaopt \theta$ such that
\begin{align}\label{eq:opt:theta'}
0 < \theta' \coloneqq \frac{\theta + \lambda/\lambdaopt}{1 - \lambda/\lambdaopt} < (1 + \Cstab[\tau]^2 \Crel^2)^{-1/2}.
\end{align}
Let $s > 0$.
Then, there exist $\copt, \Copt > 0$ such that
\begin{align}\label{eq:optimal:rates}
\begin{split}
\copt^{-1} \, \norm{u^\exact}{\bbA_s} 
&\le  \sup_{(\ell,k) \in \QQ} (\#\TT_{\ell} - \#\TT_0 + 1)^s \, \Delta_{\ell}^{k}
\le \Copt \, \max\{\norm{u^\exact}{\bbA_s},\Delta_0^0\},
\end{split}
\end{align}
where $\norm{u^\exact}{\mathbb{A}_s}$ is defined in~\eqref{eq:approxclass}.
The constant $\copt > 0$ depends only on $\Ccea=L[2M]/\alpha$, fine properties of NVB refinement, $\Cstab[\tau]$, $\Crel$, $\#\TT_0$, and $s$, and additionally on $\lmax$ or $\ell_0$, if $\lmax<\infty$ or $\eta_{\ell_0}(u_{\ell_0}^{\kmax})=0$ for some $(\ell_0,0)\in \QQ$, respectively.
The constant $\Copt > 0$ depends only on fine properties of NVB refinement, $\alpha$, $\Cstab[\tau]$, $\qred$, $\Crel$, $\Cdrel$, $1 - \lambda/\lambdaopt$ (and hence on energy contraction $\qE$), $\Cmark$, $\Clin$, $\qlin$, and~$s$. 
\hfill \qed
\end{theorem}

To estimate the work necessary to compute $u_{\ell}^{k}\in \XX_\ell$, we make the following assumptions which are usually satisfied in practice:
\begin{itemize}
	\item[$\triangleright$] The computation of all indicators $\eta_\ell(T,u_{\ell}^{k})$ for $T \in \TT_\ell$ requires $\OO(\#\TT_\ell)$ operations;
	
	\item[$\triangleright$] The marking in Algorithm~\ref{algorithm:idealized}{\rm(iii)} can be performed at linear cost $\OO(\#\TT_\ell)$ (cf. \cite{stevenson2007}, or the algorithm from~\cite{pp2020} providing $\MM_\ell$ with minimal cardinality);
	
	\item[$\triangleright$] We have linear cost $\OO(\#\TT_\ell)$ to generate the new mesh $\TT_{\ell+1}$ in Algorithm~\ref{algorithm:idealized}{\rm(iv)}.
\end{itemize}
In addition, we make the following \enquote{idealized} assumption, but refer to Remark~\ref{remark:linearcost:solver}{\rm (ii)}:
\begin{itemize}
\item[$\triangleright$] The solutions $u^k_\ell \in \XX_\ell$ of the linearized problems in Algorithm~\ref{algorithm:idealized}{\rm (i.a)} can be computed in linear complexity $\OO(\#\TT_\ell)$.
	\end{itemize}
Since a step $(\ell,k) \in \QQ$ of Algorithm~\ref{algorithm:idealized} depends on the full history of preceding steps, the total work spent to compute $u_{\ell}^{k} \in \XX_\ell$ is then of order
\begin{equation}
	\label{eq:work}
	\mathtt{work}(\ell,k) \coloneqq \sum_{\substack{ (\ell',k') \in \QQ \\ (\ell',k') \le (\ell,k) }} \#\TT_{\ell'}
	\quad \text{for all }
	(\ell,k) \in \QQ.
\end{equation}
\begin{remark}\label{remark:linearcost:solver}
	{\rm (i)} In order to avoid the computation of $\eta_{\ell+1}(u^k_{\ell+1})$ in each step of the inner loop, i.e., for all $k$ such that $(\ell + 1, k)\in \QQ$, one may use $\eta_\ell(u^\kmax_\ell)$ instead. While the proof of linear convergence with the adapted stopping criterion is possible, the proof of optimality remains an open question that goes beyond this work.
	
{\rm (ii)} The idealized assumption that the cost of solving the linearized discrete system in Algorithm~\ref{algorithm:idealized}{\rm (i.a)} is linear, can be avoided with an extended algorithm (and refined analysis) in the spirit of~\cite{hpsv2021}. There, an algebraic solve procedure is built into the presented adaptive algorithm as an additional inner loop, taking into account not only discretization and linearization errors but also algebraic errors. In this setting, the ``idealized'' assumption on the solver would be reduced to the assumption that one solver step has linear cost, which is feasible in the context of FEM.
	To keep the length of the present manuscript reasonable, we have decided to focus only on the linearization. The details follow along the lines of~\cite{hpsv2021} and are omitted.
\end{remark}

	The next corollary states the equivalence of rate-optimality with respect to the number of degrees of freedom and rate-optimality with respect to the total work, i.e., the overall computational cost.

\begin{corollary}[rate-optimality w.r.t.~computational cost]\label{cor:cost}
	Let $(\TT_\ell)_{\ell \in \N_0}$ be the sequence generated by Algorithm~\ref{algorithm:idealized}. Suppopse full linear convergence~\eqref{eq:linear} with respect to the quasi-error $\Delta_\ell^k$ from~\eqref{eq:quasi-error}. Then, for all $s > 0$, it holds that
	\begin{align} \label{eq:cost}
		\begin{split}
			C_{\rm rate} \coloneqq \sup_{(\ell,k) \in \QQ} (\#\TT_{\ell} - \#\TT_0 + 1)^s \, \Delta_{\ell}^{k} & \le  \sup_{(\ell,k) \in \QQ} \mathtt{work}(\ell,k)^s \, \Delta_{\ell}^{k}  
			\le  \frac{(\# \TT_0)^s \, \Clin}{\big(1-\qlin^{1/s}\big)^s} \,C_{\rm rate}.
		\end{split}
	\end{align}
Consequently, rate-optimality with respect to the number of elements~\eqref{eq:optimal:rates} yields that
	\begin{align}\label{eq:optimal:cost}
		\begin{split}
			\copt^{-1} \, \norm{u^\exact}{\bbA_s} 
			&\le \sup_{(\ell,k) \in \QQ} \mathtt{work}(\ell,k)^s \, \Delta_{\ell}^{k}
			\le \Copt \, \frac{(\# \TT_0)^s \, \Clin}{\big(1-\qlin^{1/s}\big)^s} \, \max\{\norm{u^\exact}{\bbA_s},\Delta_0^0\}.
		\end{split}
	\end{align}
\end{corollary}
\begin{proof}
	The first inequality in~\eqref{eq:cost} is obvious. To obtain the upper bound, let $(\ell,k) \in \QQ$. Elementary calculus (see~\cite[Lemma~22]{bhp2017}) proves that
	\begin{align*}
		\# \TT_{\coarse} \le \# \TT_0 \, \big( \# \TT_{\coarse} - \# \TT_0 + 1 \big) \text{ for all } \TT_{\coarse} \in \T.
\end{align*}
Moreover, linear convergence~\eqref{eq:linear} and the geometric series lead us to
\begin{align*}
	\sum_{\substack{(\ell',k') \in \QQ \\ (\ell',k') \le (\ell,k)}}(\Delta_{\ell'}^{k'})^{-1/s}  \stackrel{\eqref{eq:linear}}{\le} \Clin^{1/s} \,  (\Delta_{\ell}^{k})^{-1/s} \sum_{\substack{(\ell',k') \in \QQ \\ (\ell',k') \le (\ell,k)}} (\qlin^{1/s})^{|(\ell, k)| - |(\ell', k')|} \le \frac{\Clin^{1/s} \,  (\Delta_{\ell}^{k})^{-1/s}}{1-\qlin^{1/s}}.
\end{align*}
Combining the last two inequalities, we obtain that
	\begin{align*}
		\sum_{\substack{(\ell',k') \in \QQ \\ (\ell',k') \le (\ell,k)}} \#\TT_{\ell'} 
		 \le (\# \TT_0) \,\sum_{\substack{(\ell',k') \in \QQ \\ (\ell',k') \le (\ell,k)}}(\#\TT_{\ell'} - \#\TT_0 + 1) 
		&\le  (\# \TT_0) \,C_{\rm rate}^{1/s}\, \sum_{\substack{(\ell',k') \in \QQ \\ (\ell',k') \le (\ell,k)}}(\Delta_{\ell'}^{k'})^{-1/s} \\
		& \le (\# \TT_0) \,\frac{\Clin^{1/s}}{1-\qlin^{1/s}}\, (\Delta_{\ell}^{k})^{-1/s} \, C_{\rm rate}^{1/s}.
	\end{align*}
	Rearraging this estimate, we obtain the upper bound in~\eqref{eq:cost}.
\end{proof}


\section{Semilinear model problem}\label{section:modelproblem}

\subsection{Model problem}
For $d \in \{1, 2, 3\}$, let $\Omega \subset \R^d$ be a bounded Lipschitz domain. 
Given $f \in L^2(\Omega)$ and $\f \in [L^2(\Omega)]^d$, we aim to approximate the weak solution $u^\exact \in \XX \coloneqq H^1_0(\Omega)$ of the semilinear elliptic PDE
\begin{align}\label{eq:strongform:primal}
 -\div(\A \nabla u^\exact) + b(u^\exact) = f - \div \f
 \text{ \ in } \Omega
 \quad \text{subject to} \quad 
 u^\exact = 0 \text{ \ on }  \partial \Omega.
\end{align}
While the precise assumptions on the coefficients $\A \colon \Omega \to \R_{\rm sym}^{d \times d}$ and $b \colon \Omega \times \R \to \R$ are given in Section~\ref{subsection:assump:diff}--\ref{subsection:assump:nonlinear},
we note that, here and below, we abbreviate $\A \nabla u^\exact \equiv \A(\cdot) \nabla u^\exact(\cdot) \colon \Omega \to \R^d$ and $b(u^\exact) \equiv b(\cdot, u^\exact(\cdot)) \colon \Omega \to \R$.

Let $\prod{\, \cdot}{\cdot \,}_\Omega$ denote the $L^2(\Omega)$-scalar product $\prod{v}{w}_\Omega \coloneqq \int_\Omega vw \d{x}$ and let $\sprod{v}{w} \coloneqq \prod{\A\nabla v}{\nabla w}_\Omega$ be the $\A$-induced energy scalar product on $H^1_0(\Omega)$. Then, the weak formulation of~\eqref{eq:strongform:primal} reads as follows: Find $u^\exact \in H^1_0(\Omega)$ such that
\begin{align}\label{eq:weakform:primal}
\mkern-10mu \prod{\AA u^\exact}{v}\!\coloneqq\!  \sprod{u^\exact}{v} \!+\! \prod{b(u^\exact)}{v}_{\Omega}\! =\! \prod{f}{v}_{\Omega}\! +\! \prod{\f}{\nabla v}_{\Omega}\! \eqqcolon\! \prod{F}{v} \text{ for all } v \in H^1_0(\Omega).
\end{align}
Existence and uniqueness of the solution $u^\exact \in H^1_0(\Omega)$ of~\eqref{eq:weakform:primal} follow from the Browder--Minty theorem on monotone operators (see Section~\ref{subsection:browder-minty} for details).

Based on conforming triangulations $\TT_\coarse$ of $\Omega$ and fixed polynomial degree $m \in \N$, let $\XX_\coarse \coloneqq \set{v_\coarse \in H^1_0(\Omega)}{\forall \,T \in \TT_\coarse\colon \, v_\coarse|_T  \text{ is a polynomial of degree} \le m}$. Then, the FEM discretization of~\eqref{eq:weakform:primal} reads: Find $u_\coarse^\exact \in \XX_\coarse$ such that
\begin{align}\label{eq:weakform:primal:discrete}
 \sprod{u_\coarse^\exact}{ v_\coarse} + \prod{b(u_\coarse^\exact)}{v_\coarse}_{\Omega} = \dual{F}{v_\coarse} 
 \quad \text{ for all } v_\coarse \in \XX_\coarse.
\end{align}
The FEM solution $u^\exact_\coarse$ approximates the sought exact solution $u^\exact$.

\subsection{General notation}
For $1 \le p \le \infty$, let $1 \le p' \le \infty$ be the conjugate H\"{o}lder index which ensures that $\norm{\phi \,\psi}{L^1(\Omega)} \le \norm{\phi}{L^p(\Omega)}\norm{\psi}{L^{p'}(\Omega)}$ for $\phi \in L^p(\Omega)$ and $\psi \in L^{p'}(\Omega)$, i.e., $1/p + 1/p' = 1$ with the convention that $p' = 1$ for $p=\infty$ and vice versa. Moreover, for $1 \le p < d$, let $1 \le p^\ast \coloneqq dp/(d-p) < \infty$ denote the critical Sobolev exponent of $p$ in dimension $d \in \N$.
We recall the Gagliardo--Nirenberg--Sobolev inequality (see, e.g.,~\cite[Theorem~16.6]{fk1980})
\begin{align}\label{eq:intro:gns}
	\norm{v}{L^r(\Omega)} \le \CGNS \, \norm{\nabla v}{L^p(\Omega)} \quad \text{ for all } v \in W^{1,p}_0(\Omega)
\end{align}
with a constant $\CGNS = \CGNS(|\Omega|, d, p, r)$. With $\XX = H^1_0(\Omega)$, we restrict to $p=2$.
If $d \in \{1,2\}$,~\eqref{eq:intro:gns} holds for any $1 \le r < \infty$.  
If $d =3$,~\eqref{eq:intro:gns} holds for all $1 \le r \le p^\ast =6$, where $r = p^\ast$ is the largest possible exponent such that the embedding $W^{1,p}(\Omega) \hook L^r(\Omega)$ is continuous.

\subsection{Assumptions on diffusion coefficient}\label{subsection:assump:diff}

The diffusion coefficient $\A \colon \Omega \to \R_{\rm sym}^{d \times d}$ satisfies the following standard assumptions:

\renewcommand{\theenumi}{\textrm{ELL}}
\begin{enumerate}
\item $\A \in L^\infty( \Omega; \R_{\rm sym}^{d\times d})$, where $\A(x) \in \R_{\rm sym}^{d\times d}$ is a symmetric and uniformly positive definite matrix, i.e., the minimal and maximal eigenvalues satisfy
\begin{align*}
 0 < \mu_0 \coloneqq \inf_{x \in \Omega} \lambda_{\rm min} (\A(x))
 \le \sup_{x \in \Omega} \lambda_{\rm max} (\A(x)) \eqqcolon \mu_1 < \infty.
\end{align*}
\label{assump:ell}
\end{enumerate}\vspace{-\baselineskip}
In particular, the $\A$-induced energy scalar product $\sprod{v}{w} \coloneqq \prod{\A\nabla v}{\nabla w}_{\Omega}$ induces an equivalent norm $\enorm{v} \coloneqq \sprod{v}{v}^{1/2}$ on $H^1_0(\Omega)$. 

To guarantee later that the residual \textsl{a~posteriori} error estimators are well-defined, we additionally require that $A|_T \in [W^{1, \infty}(T)]^{d \times d}$ for all $T \in \TT_0$, where $\TT_0$ is the initial triangulation of the adaptive algorithm.
\subsection{Assumptions on the nonlinear reaction coefficient}\label{subsection:assump:nonlinear}

The nonlinearity $b \colon \Omega \times \R \to \R$ satisfies the following assumptions, which follow~\cite[(A1)--(A3)]{bhsz2011}:
\renewcommand{\theenumi}{\textrm{CAR}}
\begin{enumerate}
\item $b\colon \Omega \times \R \to \R$ is a \textit{Carathéodory} function, i.e., for all $n \in \N_0$, the $n$-th derivative  $b^{(n)} \coloneqq\partial_{\xi}^n b$ of $b$ with respect to the second argument $\xi$ satisfies that
\begin{itemize}
\item[$\triangleright$] for any $\xi \in \R$, the function $x \mapsto b^{(n)}(x,\xi)$ is measurable on $\Omega$,
\item[$\triangleright$] for any $x \in \Omega$, the function $\xi \mapsto b^{(n)}(x,\xi)$ exists and is continuous in $\xi$.
\end{itemize} \label{assump:car}
\end{enumerate}
\renewcommand{\theenumi}{\textrm{MON}}
\begin{enumerate}
\item We assume monotonicity in the second argument, i.e., $b'(x, \xi)\coloneqq b^{(1)}(x, \xi) \ge 0$ for all $x \in \Omega$ and $\xi \in \R$. Without loss of generality\footnote{Otherwise, consider $\tilde{b}(v)\coloneqq b(v) - b(0)$ and $\tilde{f} \coloneqq f - b(0)$ instead.}, we assume that $b(x,0)=0$. \label{assump:mon}
\end{enumerate}
To establish continuity of $ v \mapsto \prod{b(v)}{w}_{\Omega}$, we impose the following growth condition on $b(v)$; see, e.g., \cite[Chapter III,~(12)]{fk1980} or~\cite[(A4)]{bhsz2011}:
\renewcommand{\theenumi}{\textrm{GC}}
\begin{enumerate}
\item If $d \in \{1,2\}$, there exists $N \in \N$ such that $1 \le N < \infty$. For $d=3$, there exists $N \in \N$ such that $1 \le N \le 5$. Suppose that, for $d \in \{1,2, 3\}$, there exists $R > 0$ such that

\begin{align}\label{eq:estimate:gc}
|b^{(N)}(x,\xi)| \le R \quad \text{ for a.e. } x \in \Omega \text{ and all } \xi \in \R.
\end{align} \label{assump:poly}
\end{enumerate}\vspace{-\baselineskip}
While~\eqref{assump:poly} turns out to be sufficient for plain convergence of the later AILFEM algorithm, we require the following stronger assumption for linear convergence and optimal convergence rates.
\renewcommand{\theenumi}{\textrm{CGC}}
\begin{enumerate}
\item There holds~\eqref{assump:poly}, if $d \in \{1,2\}$. If $d=3$, there holds~\eqref{assump:poly} with the stronger assumption $ N \in \{ 2, 3\}$. \label{assump:compact}
\end{enumerate}
\begin{remark}\label{remark:reg12}
\rm (i) Let $v, w \in H^1_0(\Omega)$. To establish continuity of $(v, w) \mapsto \prod{b(v)}{w}_{\Omega}$, we apply the H\"{o}lder inequality with H\"{o}lder conjugates $1 \le s, s' \le \infty$ to obtain that
\begin{align}
\label{rem:gc:1}
|\prod{b(v)}{w}_{\Omega}| \le \norm{b(v)}{L^{s'}(\Omega)} \norm{w}{L^{s}(\Omega)}. 
\end{align}
The smoothness assumption~\eqref{assump:car} admits a Taylor expansion for $b$. Together with $b(0)=0$ from~\eqref{assump:mon}, this yields that
	\begin{align}\label{eq:nonlinearity:taylor}
		b(v)\, \, \,\,  & \eqreff*{assump:mon}= \quad  \sum_{n=1}^{N-1} \frac{b^{(n)}(0)}{n!} \,  v^n +\bigg( \int_0^1 \frac{(1-\xi)^{N-1}}{(N-1)!}\, b^{(N)}(\xi v) \, \d{\xi}\bigg) \,v^{N}.
	\end{align}
	With $\norm{v^n}{L^{s'}(\Omega)} = \norm{v}{L^{ns'}(\Omega)}^n$, it follows that
	\begin{align*}
		\norm{b(v)}{L^{s'}(\Omega)} &\eqreff*{assump:poly}\lesssim \mkern9mu \sum_{n=1}^{N-1}  \norm{v^n}{L^{s'}(\Omega)} +\norm{v^{N}}{L^{s'}(\Omega)} = \sum_{n=1}^{N-1}  \norm{v}{L^{ns'}(\Omega)}^n +\norm{v}{L^{Ns'}(\Omega)}^{N} \\
		&\lesssim \sum_{n=1}^{N}  \norm{v}{L^{Ns'}(\Omega)}^n \le N \max\{ 1, \norm{v}{L^{Ns'}(\Omega)}^{N-1}\}\, \norm{v}{L^{Ns'}(\Omega)},
	\end{align*}
	where the second to last estimate exploits the $L^p$-space inclusions for bounded $\Omega$. To guarantee that $|\prod{b(v)}{w}_{\Omega}| < \infty$, condition~\eqref{assump:poly} should ensure that the embedding 
\begin{align}\label{eq:embedding}
H^1_0(\Omega) \hook L^{r}(\Omega) \quad \text{ is continuous } \quad \text{ for } \quad r = s \quad  \text{ and } \quad r = Ns'.
\end{align}
 If $d \in \{1,2\}$,~\eqref{eq:embedding} follows if $1 \le r < \infty$ and hence arbitrary $1 < s < \infty$ and $N \in \N$. If $d=3$, $r = s = 2^\ast=6$ is the maximal index in~\eqref{eq:embedding}. Hence, it follows that $ N \le 2^\ast/s' = 2^\ast/{2^\ast}' = 2^\ast-1 = 5$. 
 Altogether, we conclude continuity of $(v, w) \mapsto \prod{b(v)}{w}_{\Omega}$ for all $N \in \N$ if $d \in \{1,2\}$, and $N \le 5$ if $d =3$.

\rm{(ii)} The definition of~\cite[(GC)]{bbimp2022} uses 
\begin{align*}
	|b^{(n)}(x,\xi)| \le R (1+|\xi|^{N-n}) \quad \text{ for all } x \in \Omega, \text{ all } \xi \in \R, \text{ and all } 0 \le n \le N
\end{align*}
instead of~\eqref{eq:estimate:gc}. However, the following observation replaces the estimates for all $b^{(n)}$ with $0 \le n <N$. Due to the smoothness assumption~\eqref{assump:car}, we may apply a Taylor expansion for an admissible $\sigma$ such that $(N-n)\,\sigma < \infty$ if $d=1, 2$ and $(N-n)\,\sigma \le 6$ if $d=3$. Together with $\norm{v^n}{L^\sigma(\Omega)}=\norm{v}{L^{n\sigma}(\Omega)}^n$, this leads us to
\begin{align}\label{rem:estimate:derivatives}
\norm{b^{(n)}(v)}{L^\sigma(\Omega)}\,   & \le \sum_{j=n}^{N-1} \frac{b^{(j)}(0)}{(j-n)!} \,  \norm{v^{j-n}}{L^\sigma(\Omega)} +\bigg( \int_0^1 \frac{(1-\xi)^{N-1-n}}{(N-1-n)!}\, b^{(N)}(\xi v) \, \d{\xi}\bigg) \,\norm{v^{N-n}}{L^\sigma(\Omega)} \nonumber \\
& \eqreff*{assump:poly}\lesssim \, \, \sum_{j=n}^{N-1} \norm{v}{L^{(j-n)\sigma}(\Omega)}^{j-n} + \norm{v}{L^{(N-n)\sigma}(\Omega)}^{N-n} \lesssim \sum_{j=n}^{N} \norm{v}{L^{(N-n)\sigma}(\Omega)}^{j-n} \nonumber \\
&\le (N-n)\,\big( 1 + \norm{v}{L^{(N-n)\sigma}(\Omega)}^{N-n} \big) \lesssim (N-n) \, ( 1 + \enorm{v}^{N-n}),
	\end{align}
where the additive constant stems from the fact that $b^{(n)}(0) \neq 0$ in general (in contrast to the reasoning in~{\rm (i)}). This estimate plays a central role in proving the local Lipschitz continuity of $b$ and thus of the overall semilinear model problem; see Lemma~\ref{prop:energybound} below and the discussion thereafter.
\hfill \qed
\end{remark}
\subsection{Assumptions on the right-hand sides}\label{subsection:rhs}
For $d=1$, the exact solution $u^\exact$ from~\eqref{eq:weakform:primal} below satisfies an $L^\infty$-bound, since $H^1$-functions are absolutely continuous. For $d \in \{2, 3\}$, we need the following assumption:

\renewcommand{\theenumi}{\textrm{RHS}}
\begin{enumerate}
\item We suppose that the right-hand side fulfils that
\begin{align*}
\f \in [L^{p}(\Omega)]^{d} \text{ for some } p  > d \ge 2 \quad \text { and } \quad f \in L^{q}(\Omega) \text{ where } 1/q \coloneqq 1/p + 1/d.
\end{align*} \label{assump:rhs}
\end{enumerate} \vspace{-\baselineskip}
To guarantee later that the residual \textsl{a~posteriori} error estimator from~\eqref{eq:estimator:primal} is well-defined, we additionally require that $\f|_T \in H(\div,T)$ and $\f|_T \cdot \n \in L^2(\partial T)$ for all $T \in \TT_0$, where $\TT_0$ is the initial triangulation of the adaptive algorithm.

\subsection{Well-posedness and applicability of abstract framework}\label{subsection:browder-minty}
Let $v, w \in H^1_0(\Omega)$. We consider the operator $\AA$, where $H^{-1}(\Omega) \coloneqq H^1_0(\Omega){'}$ is used to denote the dual space of $H^1_0(\Omega)$,
\begin{align}\label{eq:op:bm}
 \AA \colon H^1_0(\Omega) \to H^{-1}(\Omega),
 \quad 
 \AA w \coloneqq \sprod{w}{\cdot \,} + \prod{b(w)}{\cdot \,}_{\Omega}.
\end{align}
Since $b'(x,\zeta) \ge 0$ according to~\eqref{assump:mon}, this implies that
\begin{align*}
	\big( b(x, \xi_2) - b(x, \xi_1) \big) (\xi_2-\xi_1) \ge 0
	\quad \text{for all } x \in \Omega \text{ and } \xi_1, \xi_2 \in \R.
\end{align*}
Together with~\eqref{assump:ell} and for $v, w \in H^1_0(\Omega)$, we thus see that
\begin{align}
	\begin{split}\label{eq:bm:stronglymonotone}
		\prod{ \AA w - \AA v}{w - v}
		&= \sprod{w - v}{ w - v} + \prod{b(w)-b(v)}{w-v}_{\Omega}	\ge \enorm{w-v}^2. 
	\end{split}
\end{align}
This proves that $\AA$ is strongly monotone with $\alpha = 1$ with respect to the energy norm $\enorm{\,\cdot\,}$. The following lemma is crucial to prove local Lipschitz continuity.
\begin{lemma}\label{prop:energybound}
	Suppose~\eqref{assump:rhs},~\eqref{assump:ell},~\eqref{assump:car},~\eqref{assump:mon}, and~\eqref{assump:poly}. Let $\vartheta>0$ and let $v, w \in H^1_0(\Omega)$ with $\max \big\{\enorm{w}, \enorm{w-v}\big\} \le \vartheta < \infty$. Then, it holds that
	\begin{align}\label{eq:critical:lipschitz}
		\prod{b(w) - b(v)}{z}_{\Omega}
			\le \widetilde{L}[\vartheta] \, \enorm{w-v} \enorm{z}
			\quad \text{for all } z \in H^1_0(\Omega)
	\end{align}
	with $\widetilde{L}[\vartheta] = \widetilde{L}(|\Omega|, d, \vartheta, N, R, \mu_0)$. 
\end{lemma}
\begin{proof} Due to the smoothness assumption~\eqref{assump:car}, we may consider the Taylor expansion
	\begin{align}
		\begin{split}\label{eq:primal:taylorexp}
			b(v) &= \sum_{n=0}^{N-1} b^{(n)}(w) \frac{( v - w )^{n}}{n!} + \frac{  ( v- w )^{N}}{(N-1)!} \int_0^1 ( 1 - \xi)^{N-1} \, b^{(N)}\big( w + (v - w) \, \xi \big)  \d{\xi}.
		\end{split}
	\end{align}
	In order to apply the generalized H\"{o}lder inequality for three terms $\phi, \varphi, \psi \in H^1_0(\Omega)$
		\begin{align*}
			\prod{\phi \,\varphi}{\psi}_{\Omega} \le \norm{\phi}{L^{t''}(\Omega)}\, \norm{\varphi}{L^{t}(\Omega)}\,\norm{\psi}{L^{t}(\Omega)},
		\end{align*}
		where $1 = 1/t + 1/t + 1/t''$, we choose $t > 2$ arbitrarily for $d \in \{1, 2\}$ and $t= 6$ and hence $t'' = 3/2$ for $d=3$. In both cases, we see that
	\begin{align*}
		\prod{b(w) - b(v)}{z}_{\Omega} &\le \sum_{n=1}^{N-1} \frac{1}{n!} \,\norm{b^{(n)}(w)( w-v )^{n-1}}{L^{t''}(\Omega)} \norm{w-v}{L^t(\Omega)}\norm{z}{L^{t}(\Omega)} \\
		& \hspace{-15mm} + \norm[\Big]{\frac{  (w-v )^{N-1}}{(N-1)!} \int_0^1 ( 1 - \xi )^{N-1} \, b^{(N)}\big( w + (v - w) \, \xi \big)  \d{\xi}}{L^{t''}(\Omega)} \norm{w-v}{L^t(\Omega)} \norm{z}{L^{t}(\Omega)}  \\
		&\hspace{-20mm} \stackrel{\eqref{assump:poly}}{\lesssim} \Big(\sum_{n=1}^{N-1} \norm{b^{(n)}(w)( w - v )^{n-1}}{L^{t''}(\Omega)} + \norm{w-v}{L^{(N-1)t''}(\Omega)}^{N-1} \Big) \enorm{w-v}\, \enorm{z},
	\end{align*}
where the hidden constant depends on $R$ from~\eqref{assump:poly}. Since $H^1_0(\Omega) \hook L^{(N-1)t''}(\Omega)$ for $d \in \{1, 2, 3\}$, it remains to prove that
	\begin{align}\label{eq:tbp}
		\norm{b^{(n)}(w)( w - v )^{n-1}}{L^{t''}(\Omega)}  \le C[\vartheta]  \quad \text{ for all } n=1, \ldots, N-1.
	\end{align}
	To this end, choose $t_1 = (N-1) t''/(N-n)$ and $t_2 = (N-1)t''/(n-1)$ and note that
	\begin{align*}
		\frac{1}{t''} = \frac{1}{t''} \Big( \frac{N-n}{N-1} + \frac{n-1}{N-1} \Big) = \frac{1}{t_1} + \frac{1}{t_2}.
	\end{align*}
	Using the H\"{o}lder inequality, we arrive at
	\begin{align*}
		\norm{b^{(n)}(w)( w - v )^{n-1}}{L^{t''}(\Omega)} \le \norm{b^{(n)}(w)}{L^{t_1}(\Omega)}\norm{( w - v )^{n-1}}{L^{t_2}(\Omega)}.
	\end{align*}
	Since $\norm{\varphi^j}{L^\sigma(\Omega)} =\norm{\varphi}{L^{j\sigma}(\Omega)}^j$ and $(N-1)t'' < \infty$ if $d \in \{1, 2\}$ and $(N-1)t'' \le 6$ if $d=3$ guarantee admissibility as in Remark~\ref{remark:reg12}{\rm (ii)}, we apply the Sobolev embedding to obtain that
	\begin{align*}
		\norm{b^{(n)}(w)}{L^{t_1}(\Omega)} \eqreff*{rem:estimate:derivatives}{\lesssim} 1 + \norm{w}{L^{(N-n)t_1}(\Omega)}^{N-n} = 1 + \norm{w}{L^{(N-1)t''}(\Omega)}^{N-n} \lesssim 1 + \enorm{w}^{N-n}
	\end{align*}
	and
	\begin{align*}
		\norm{(w-v)^{n-1}}{L^{t_2}(\Omega)} =  \norm{w-v}{L^{(n-1)t_2}(\Omega)}^{n-1} =\norm{w-v}{L^{(N-1)t''}(\Omega)}^{n-1} \lesssim \enorm{w-v}^{n-1}. 
	\end{align*}
	The last estimates together with the assumptions $\enorm{w-v} \le \vartheta$ and $\enorm{w}\le \vartheta$ conclude the proof with hidden constant $\widetilde{L}[\vartheta] = \widetilde{L}(|\Omega|, d, \vartheta, N, R, \mu_0)> 0$.
\end{proof}

To see the local Lipschitz continuity of $\AA$, let $v, w, \psi \in H^1_0(\Omega)$ and observe that
\begin{align*}
	\prod{ \AA w - \AA v}{\psi} & =  \sprod{w - v}{\psi} + \prod{b(w)- b(v)}{\psi}_{\Omega}  \stackrel{\mathclap{\eqref{eq:critical:lipschitz}}}{\le}  (1 + \widetilde{L}[\vartheta]) \,\enorm{w- v}\, \enorm{\psi},
	\end{align*}
provided that $\enorm{w} \le \vartheta$ and $\enorm{w-v} \le \vartheta$. This shows that $\AA$ is locally Lipschitz continuous with Lipschitz constant $L[\vartheta] \coloneqq 1 + \widetilde{L}[\vartheta]$. Hence, $\AA$ fits into the abstract setting of Section~\ref{section:monlip}. 

Furthermore, following~\cite{ahw2022}, we note that the energy for the semilinear model problem~\eqref{eq:strongform:primal} of Section~\ref{section:modelproblem} for $v \in H^1_0(\Omega)$ is given by
		\begin{equation}\label{eq:semilinear:energy}
			\EE(v) = \frac{1}{2} \int_\Omega |\A^{1/2}\nabla v|^2 \, \d{x} + \int_\Omega \int_0^{v(x)} b(s) \, \d{s} \d{x} - \int_\Omega  f  v\, \d{x}-  \int_\Omega  \f \cdot \nabla v\, \d{x}. 
		\end{equation}
	To see that the second integral is well-defined, note that the integration of the Taylor expansion~\eqref{eq:nonlinearity:taylor} gives rise to a term $s^{N+1}$ evaluated at $s = v(x)$ and $s= 0$. Its integrability $\norm{v^{N+1}}{L^1(\Omega)}=\norm{v}{L^{(N+1)}(\Omega)}^{N+1} < \infty$ is ensured by~\eqref{assump:compact}.


%
\subsection{Residual error estimators}

For $\TT_\coarse \in \T$ and $v_\coarse \in \XX_\coarse$, the local contributions of the standard residual error estimator for the semilinear model problem~\eqref{eq:weakform:primal} read
\begin{align}
\begin{split}\label{eq:estimator:primal}
\eta_\coarse(T, v_\coarse)^2 &\coloneqq h_T^2 \,\norm{f + \div(\A \, \nabla v_\coarse - \f) - b(v_\coarse)}{L^2(T)}^2  \\
& \quad + h_T \, \norm{\jump{(\A \, \nabla v_\coarse - \f ) \, \cdot \, \n}}{L^2(\partial T \cap \Omega)}^2,
\end{split}
\end{align}
where $\jump{\,\cdot\,}$ denotes the jump across edges (for $d=2$) resp.\ faces (for $d=3$) and $\n$ denotes the outer unit normal vector. For $d=1$, these jumps vanish, i.e., $\jump{\,\cdot\,} = 0$.~\cite[Proposition~15]{bbimp2022} proves the axioms of adaptivity~\eqref{assumption:stab}--\eqref{assumption:drel} for the present setting. 
\begin{proposition}[\phantom{}{\cite[Proposition~15]{bbimp2022}}]\label{proposition:axioms} Suppose~\eqref{assump:rhs},~\eqref{assump:ell},~\eqref{assump:car},\allowbreak~\eqref{assump:mon}, and~\eqref{assump:compact}. Then, the residual error estimator from~\eqref{eq:estimator:primal} satiesfies~\eqref{assumption:stab}--\eqref{assumption:drel} from Section~\ref{subsection:axioms}. The constant $\Crel$ depends only on $d$,  $\mu_0$, and uniform shape regularity of the meshes $\TT_\coarse \in \T$. The constant $\Cdrel$ depends, in addition, on the polynomial degree $m$, and $\Cstab[\vartheta]$ depends furthermore on $|\Omega|$, $\vartheta$, $n$, $R$, and $\A$. \hfill \qed
\end{proposition}


\section{Practical algorithm}
\label{section:practical_algorithm}
For the semilinear problem~\eqref{eq:strongform:primal} of Section~\ref{section:modelproblem}, it holds that $\alpha = 1$ according to~\eqref{eq:bm:stronglymonotone}. The optimal damping parameter $\delta>0$ as well as $L[6M]$ are unknown in practice. In this section, we present a practical algorithm which is formulated with computable quantities only.

\subsection{AILFEM and contraction of damped Zarantonello iteration}

Instead of adaptively choosing $\delta>0$, we adapt the local Lipschitz constant $L$. Since $\alpha = 1$, this already determines the optimal choice $\delta=1/L$ and $q[\delta]^2= 1- \delta^2$; see Remark~\ref{remark:delta}. 

\begin{algorithm}[practical AILFEM]\label{algorithm:practical} 
	\textbf{Input:} initial triangulation $\TT_0$, initial guess $u_0^0 \coloneqq 0$ and $M = \norm{F - \AA 0}{\XX^\prime} < \infty$ according to~\eqref{eq:exact:bounded}, marking parameters $0 < \theta \le 1$ and $\Cmark \ge 1$, solver termination parameter $\lambda >0$, and solver parameters $L_0 \coloneqq 1$ and $\beta \coloneqq \sqrt{2}$.
	
 \noindent \textbf{Loop:} For $\ell = 0, 1, 2, \dots$, repeat the following steps {\rm(i)--(v)}:
	\begin{itemize}
		\item[\rm(i)] Calculate $\delta_\ell \mapsfrom 1/L_\ell$ and $q_\ell^2 \mapsfrom 1 - \delta_\ell^2$. 
		\item[\rm(ii)] For all $k = 1, 2, \dots$, repeat the following steps {\rm(a)--(c)}:
		\begin{itemize}
			\item[\rm(a)] Compute $u_\ell^k \coloneqq \Phi_{\ell}(\delta_\ell; u_\ell^{k-1})$ and $\eta_\ell(T, u_\ell^k)$ for all $T \in \TT_\ell$.
			\item[\rm(b)] Terminate $k$-loop if \quad $\big(|\EE(u^{k-1}_\ell) - \EE(u^{k}_\ell)| \le  \lambda^2\, \eta_\ell(u^k_\ell)^{2} \quad \land \quad \enorm{u^k_\ell} \le 2M$\big).
			\item[\rm(c)] 
 	If $\big( \EE(u^{k}_\ell) > q_\ell^2\, \EE(u^{k-1}_\ell) \big)$, then 
		 	\begin{itemize}
		 		\item[\rm{(c1)}] Discard the computed $u_\ell^k$ and set $k \mapsfrom k-1$.
		 		\item[\rm{(c2)}] Increase $L_\ell \mapsfrom \beta \, L_\ell$.
		 		\item[\rm{(c3)}] Update $\delta_\ell \mapsfrom 1/L_\ell$ and $q_\ell^2 \mapsfrom 1 - \delta_\ell^2$. 
		 	\end{itemize}
		\end{itemize}
		\item[\rm(iii)] Upon termination of the $k$-loop, define $\kk(\ell) \coloneqq k$.
		\item[\rm(iv)] Determine $\MM_\ell \subseteq \TT_\ell$ with $\theta \,\eta_\ell(u_\ell^{\kmaxl})^2 \le \sum_{T \in \MM_\ell} \eta_\ell(T, u_\ell^{\kmaxl})^2$.
		\item[\rm(v)] Generate $\TT_{\ell+1} \coloneqq \refine(\TT_\ell, \MM_\ell)$ and define $u_{\ell+1}^0 \coloneqq u_\ell^{\kmaxl}$. 
	\end{itemize}
\end{algorithm}

\begin{remark}\label{remark:energy:necessary}
	The motivation of the criterion in Algorithm~\ref{algorithm:practical}{\rm (ii.c)} is based on the equivalence
		\begin{align}\label{eq:criterion:equivalence}
		\EE(u^{k}_\ell)\! -\! \EE(u^\exact_\ell) \le q_\ell^2 \big[	\EE(u^{k-1}_\ell)\!  -\! \EE(u^\exact_\ell) \big] \quad  \Longleftrightarrow \quad  
		\EE(u^{k}_\ell) \!- \! q_\ell^2  \EE(u^{k-1}_\ell) \le (1\! -\! q_\ell^2)\, \EE(u^\exact_\ell).
	\end{align}
The energy minimization property from Lemma~\ref{lemma:equivalence} and $b(0)= 0$ from~\eqref{assump:mon} show that $\EE(u^\exact_\ell) \le \EE(0) = 0$; cf.~\eqref{eq:semilinear:energy}.
 As a necessary criterion for energy contraction~\eqref{crucial:energycontractivity}, we thus obtain $\EE(u^{k}_\ell) \le q_\ell^2 \EE(u^{k-1}_{\ell})$, which is enforced by Algorithm~\ref{algorithm:practical}{\rm (ii.c)}.
 \end{remark}
\begin{remark}
	Note that $\lambda >0$ is arbitrary but fixed and remains unchanged throughout the algorithm. In the numerical experiments below, the particular choice $\lambda = 0.1$ is motivated by the following heuristic argument: the estimator $\eta_\ell(u^\exact_\ell)$ and hence approximately $\eta_\ell(u^\kmax_\ell)$ controls the discretization error, while $\enorm{u^\exact_\ell - u^\kmax_\ell}^2 \eqreff{eq:equivalence}{\simeq} \EE(u^\kmax_\ell) - \EE(u^\exact_\ell) \eqreff{eq:energycontr:triangle}\lesssim \EE(u^{\kmax-1}_\ell) - \EE(u^{\kmax}_\ell) \eqreff{eq:f4}{\simeq} \enorm{u^\kmax_\ell - u^{\kmax-1}_\ell}^2$ controls the linearization error --- at least if $\delta_\coarse$ is sufficiently small. Hence, $\EE(u^{\kmax-1}_\ell) - \EE(u^{\kmax}_\ell) \le 0.1^{2} \, \eta_\ell(u^\kmax_\ell)^{2}$ heuristically aims at limiting the linearization error to be at most $10\%$ of the current discretization error.
\end{remark}
The next result states that Algorithm~\ref{algorithm:practical}{\rm{(ii.c)}} will not lead to an infinite loop.
\begin{proposition}\label{prop:damped-zarantonello}
	Suppose that $\AA$ satisfies~\eqref{eq:strongly-monotone},~\eqref{eq:locally-lipschitz}, and~\eqref{eq:potential}. Let $u_\coarse^0 \in \XX_\coarse$ with $\enorm{u_\coarse^0} \le 2M$. Set $L_{0}, L_{\coarse} \mapsfrom 1$ and define $\beta \coloneqq \sqrt{2}$. Compute $\delta_\coarse = 1/L_{\coarse}$ and $q_\coarse^2 = 1 - \delta_\coarse^2$. Starting with $k \mapsfrom 1$ and $u_\coarse^1 \coloneqq \Phi_\coarse(\delta_\coarse; u_\coarse^0) \in \XX_\coarse$, we proceed as follows:
	\begin{itemize}
		\item Given $u_\coarse^k \in \XX_\coarse$ for $k \ge 1$, compute $u_\coarse^{k+1} \coloneqq  \Phi_\coarse(\delta_\coarse; u_\coarse^k) \in \XX_\coarse$
		and check if
		\begin{align}\label{eq:prop:damped-zarantonello}
			 \EE(u^{k+1}_{\coarse}) \le q_{\coarse}^2\, \EE(u^{k}_{\coarse}).
		\end{align}
		\item If~\eqref{eq:prop:damped-zarantonello} holds, then increase $k \mapsfrom k + 1$.
		\item If~\eqref{eq:prop:damped-zarantonello} fails, then 
increase $L_{\coarse}  \mapsfrom \beta \, L_{\coarse} $ and update $\delta_{\coarse} \mapsfrom 1/L_{\coarse} $ and $q_\coarse^2 \mapsfrom 1 - \delta_\coarse^2$. Discard the computed $u_\coarse^{k+1}$.
	\end{itemize}
	Then, the condition~\eqref{eq:prop:damped-zarantonello} fails only finitely often so that this simple algorithm defines the sequence of iterates $(u_\coarse^k)_{k \in \N_0}$. 
\end{proposition}

\begin{proof}
{\bf Step 1.} Given the initial $L_0 = 1$, there exists a minimal number $j \in \N_0$ such that
	\begin{align*}
		\frac{L[6M]^2}{2\alpha} < \beta^j L_0 = L_{\coarse}(j)  \quad \text{ and thus } \quad \delta_\coarse \coloneqq\delta_\coarse(j) = \frac{1}{\beta^j L_0} < \frac{2 \alpha}{L[6M]^2}. 
	\end{align*}
Define $q_\coarse[\delta_\coarse(k)]^2 \coloneqq 1 - \delta_\coarse(k)^2$. Recall $\qE[\delta_\coarse]$ from~\eqref{eq:energycontr:const} and observe that 
	\begin{align*}
		\qE[\delta_\coarse(k)]^2 = 1 - \big(1 -  \frac{\delta_\coarse(k) L[6M]}{2} \big) \, \frac{2\delta_\coarse(k) \alpha^2}{L[3M]} \simeq 1 - \delta_\coarse(k) + \delta_\coarse(k)^2  \quad \text{ for } \delta_\coarse(k) \to 0.
		\end{align*} 
	Since $\delta_\coarse(k) \to 0$ for $k \to \infty$, there exists a minimal number $k_0 \in \N$ with $k_0 \ge j$ such that
	\begin{align*}
\qE[\delta_\coarse(k_0)]^2 < q_\coarse^2[\delta_\coarse(k_0)]  = 1 - \frac{1}{\beta^{2k_0} L_0^2}  <1  \quad \text{ as well as } \quad \delta_\coarse(k_0) = \frac{1}{\beta^{k_0} L_0} < \frac{2 \alpha}{L[6M]^2}.
	\end{align*}
 
This implies that Proposition~\ref{proposition:econtraction} holds for the theoretical sequence $\widetilde u_\coarse^{0} \coloneqq u_\coarse^{k_0}$ and $\widetilde u_\coarse^{k+1} \coloneqq \Phi_\coarse(\delta_{\coarse}; \widetilde u_\coarse^k)$. In particular, we conclude that energy contraction~\eqref{crucial:energycontractivity} holds with $q^{2}_{\coarse} =1 - \delta^2_\coarse$. Moreover, Remark~\ref{remark:energy:necessary} shows that the necessary criterion~\eqref{eq:prop:damped-zarantonello} is guaranteed to hold for the iterates $(\widetilde{u}_\coarse^k)_{k\in \N_0}$ as soon as~\eqref{crucial:energycontractivity} holds.

{\bf Step 2.}
	Since the failure of~\eqref{eq:prop:damped-zarantonello} increases the current value of $L$ to $\beta L$, it follows from Step~1 that~\eqref{eq:prop:damped-zarantonello} can fail only finitely often, until the recomputed sequence $(u_\coarse^k)_{k \in \N_0}$ satisfies~\eqref{eq:prop:damped-zarantonello} for all $k \in \N_0$ with $k \ge k_0$. 
\end{proof}

\begin{remark}\label{rem:comparison}
The optimality results for Algorithm~\ref{algorithm:idealized} are expected to carry over --- at least asymptotically --- to Algorithm~\ref{algorithm:practical}; see Proposition~\ref{prop:damped-zarantonello}. The major difficulty lies in algorithmically determining whether the correct estimate of the Lipschitz constant (and thus $\delta_\coarse$) is preasymptotic or not, i.e., determining $k$ in Step~2 from the last proposition by means of computable quantities only. However, it is ensured that $\delta_\coarse$ remains uniformly bounded from below.
	\end{remark}


\section{Numerical experiments}\label{section:numerical}
\renewcommand{\thesubfigure}{\Alph{subfigure}}

In this section, we test and illustrate Algorithm~\ref{algorithm:practical} with numerical experiments. All experiments were implemented using the Matlab code \emph{MooAFEM}~\cite{ip2022}. Throughout, $\Omega \subset \R^2$ and we use $x = (x_1, x_2) \in \Omega$ to denote the Cartesian coordinates. In all experiments, we consider equation~\eqref{eq:strongform:primal} with isotropic diffusion $\A =\big( \begin{smallmatrix}
		\varepsilon & 0 \\
		0 & \varepsilon 
\end{smallmatrix} \big)$ with $0 < \varepsilon \le 1$. The adaptivity parameter is set to $\theta = 0.5$ and $\Cmark =1$. Moreover, recall the definition of the overall computational cost from~\eqref{eq:work}, which reads
	\begin{align*}
		\mathtt{work}(\ell,k) = \sum_{\substack{ (\ell',k') \in \QQ \\ (\ell',k') \le (\ell,k)} } \# \TT_{\ell'} = k \, \# \TT_\ell + \sum_{\ell'=0}^{\ell-1} \kmax(\ell') \, \# \TT_{\ell'}.
	\end{align*} 
\begin{experiment}[nonlinear variant of the sine-Gordon equation~{\cite[Experiment~5.1]{ahw2022}}]\label{example:gordon1}
		\begin{figure}
		\centering
			{\includegraphics[width=0.49\textwidth]{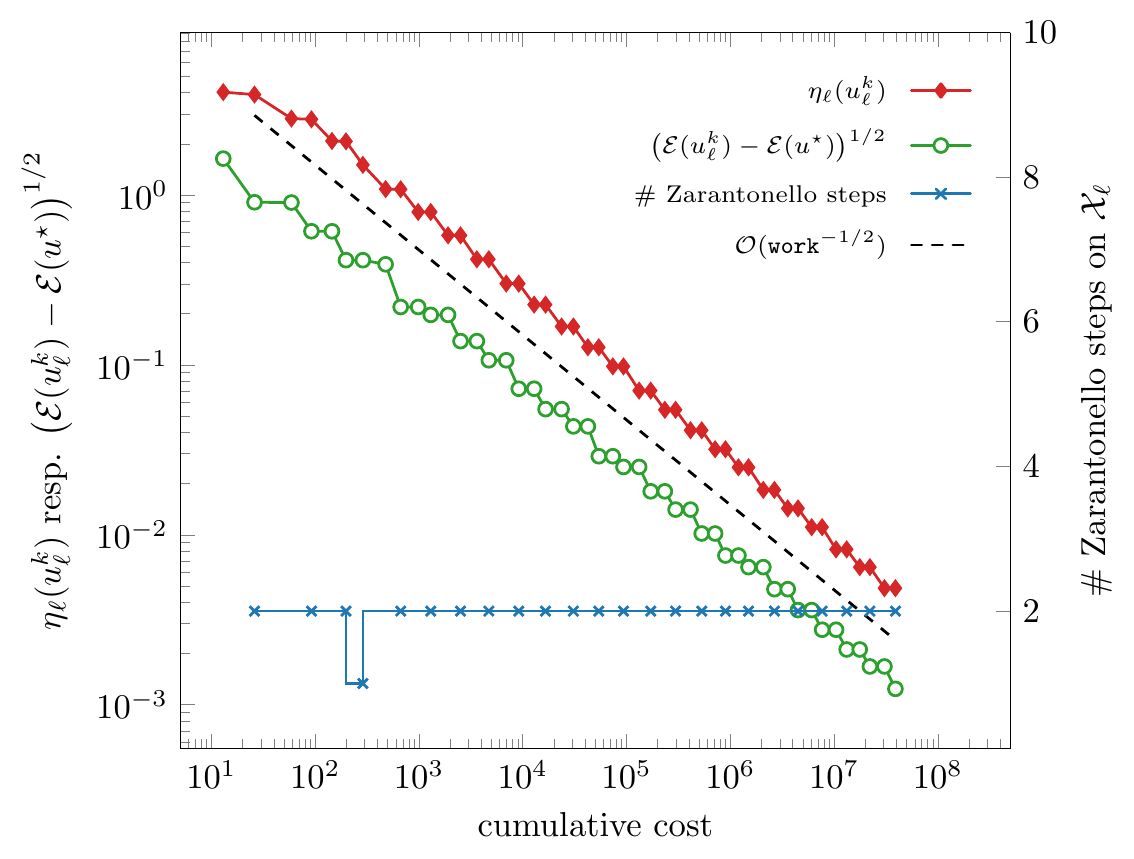}}
		\quad
			{\includegraphics[width=0.45\textwidth]{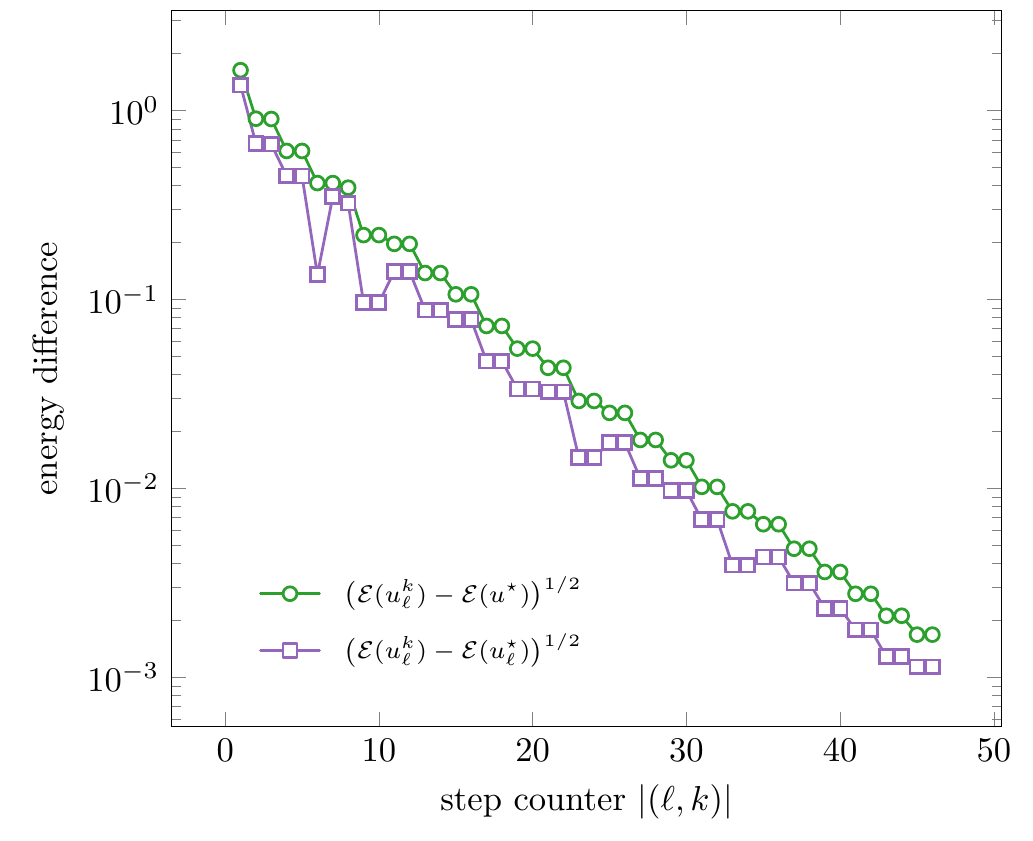}}
			\label{fig:GordonA}
		\caption{Results of Experiment~\ref{example:gordon1} with polynomial degree $m=1$. Left: Error estimator $\eta_\ell(u^k_\ell)$ (diamond, left ordinate) and energy difference of iterative solutions $\big( \EE(u^k_\ell) - \EE(u^\exact) \big)^{1/2}$ (circle, left ordinate) against $\mathtt{work}(\ell,k)$ and the number of Zarantonello steps on $\XX_\ell$ (cross, right ordinate). Right: Energy difference of $\EE(u_\ell^k)$ to $\EE(u^\exact)$ (circle) and to $\EE(u^\exact_\ell)$ (square) over the total step counter $|(\ell,k)|$. Throughout, $\EE(u^\exact)$ is obtained by Aitken extrapolation and $\EE(u^\exact_\ell)$ by sufficient Zarantonello steps on each level $\ell$.} \label{fig:Gordon1}
	\end{figure}
	For $\Omega = (0,1)^2$, let $\XX= H^1_0(\Omega)$ with $\enorm{\, \cdot \, }^{2} = \prod{\nabla \cdot}{\nabla \cdot}$ (i.e., $\varepsilon=1$) and consider
	\begin{align}\label{eq:gordon1}
		-\Delta u^\exact + (u^\exact)^3+ \sin(u^\exact) = f \quad \text{ in } \Omega \quad \text{ subject to } \quad u^\exact= 0 \text{\ on } \partial \Omega, 
	\end{align}
	with the monotone semilinearity $b(v) = v^{3}+ \sin(v)$, which satisfies~\eqref{assump:ell},~\eqref{assump:car}, \eqref{assump:mon}, and~\eqref{assump:poly}. We set $\f = 0$ and choose $f$ in such a way that
	\begin{align*}
		u^\exact(x) = \sin(\pi x_1)\sin(\pi x_2),
	\end{align*}
which satisfies~\eqref{assump:rhs}.
In Figure~\ref{fig:Gordon1}, we plot the \textsl{a~posteriori} estimator $\eta_\ell(u^k_\ell)$ and the energy difference of the iterative solutions $\big( \EE(u^k_\ell) - \EE(u^\exact) \big)^{1/2}$ against the $\mathtt{work}(\ell,k)$ for lowest order FEM $m=1$, where we approximate $\EE(u^\exact)$ by means of Aitken convergence acceleration on uniform meshes with up to $\#\TT_{\mathrm{final}} = 67108864$ degrees of freedom on the finest mesh. The decay rate is of (expected) optimal order $\mathcal{O}(\mathtt{work}(\ell,k)^{-1/2})$ as $|(\ell, k)| \to \infty$. Moreover, the experimentally observed number of sufficient linearization steps $\kmax(\ell)$ is two. Furthermore, in Figure~\ref{fig:Gordon1}, we plot the difference of $\EE(u^k_\ell)$ to the approximated reference energy $\EE(u^\exact)$ using Aitken's acceleration and to the energy $\EE(u^\exact_\ell)$ on $\XX_\ell$ over the step counter $|(\ell,k)|$. The reference energy $\EE(u^\exact_\ell)$ is calculated by a sufficient number of Zarantonello iterations on each level $\ell$ until the energy difference of successive iterates is below the tolerance $\mathtt{tol} < 10^{-15}$.
\end{experiment} 
\begin{experiment}[singularly perturbed sine-Gordon equation]\label{example:gordon2}
	\begin{figure}\label{fig:Gordon2}
		\begin{subfigure}[b]{\textwidth}%
			\includegraphics[width=0.49\textwidth]{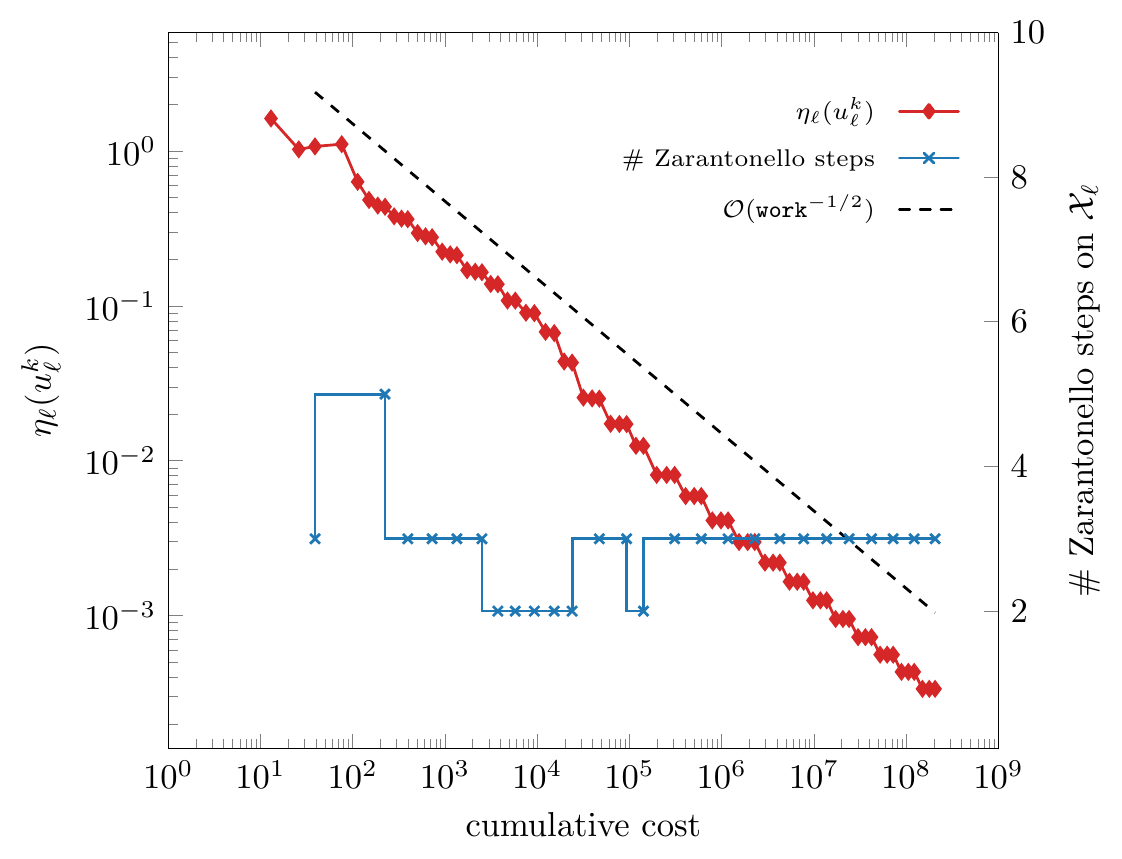} \quad
			\includegraphics[width=0.49\textwidth]{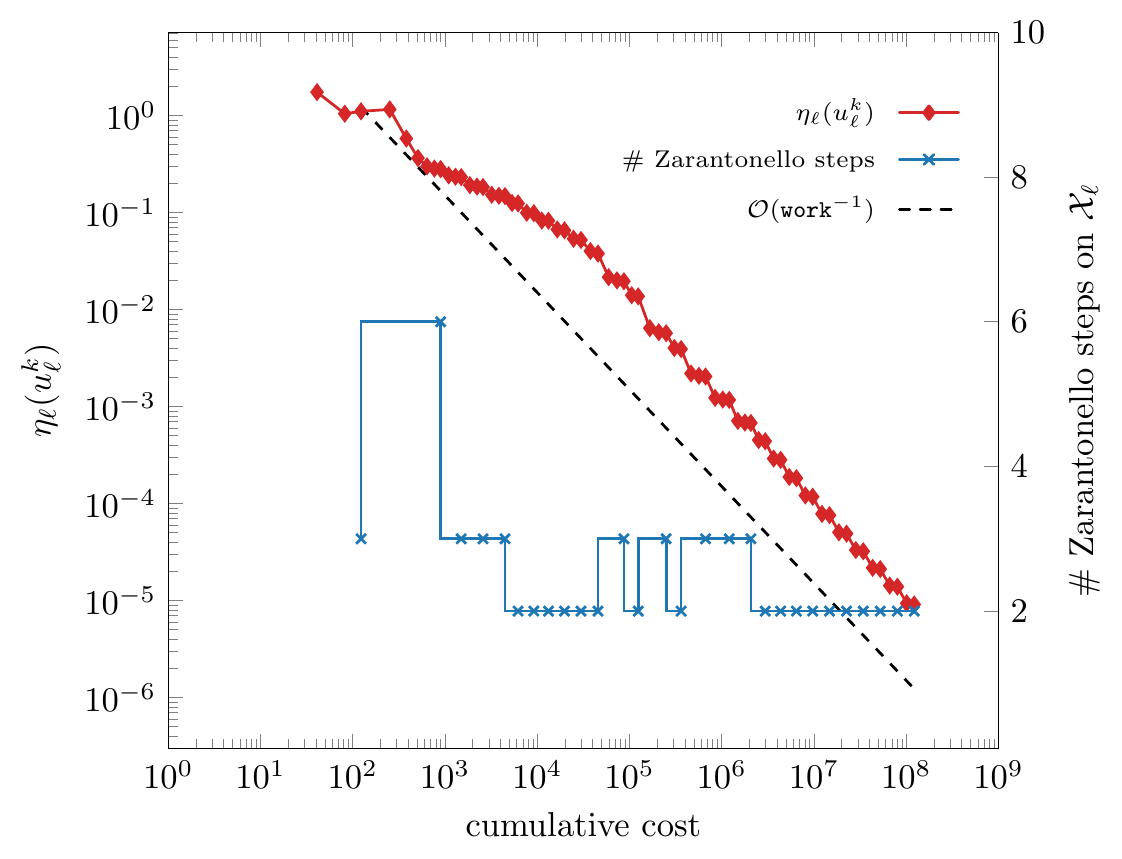}
			\caption{Error estimator $\eta_\ell(u^k_\ell)$ over $\mathtt{work}$ (diamond, left ordinate) and number of Zarantonello iteration steps on $\XX_\ell$ over $\mathtt{work}$ (cross, right ordinate) for $m=1$ (left) and $m=2$ (right).}
			\label{fig:Gordon2A}
		\end{subfigure}
		\bigskip \\
		\begin{subfigure}[b]{0.4\textwidth}
			{\includegraphics[width=\textwidth]{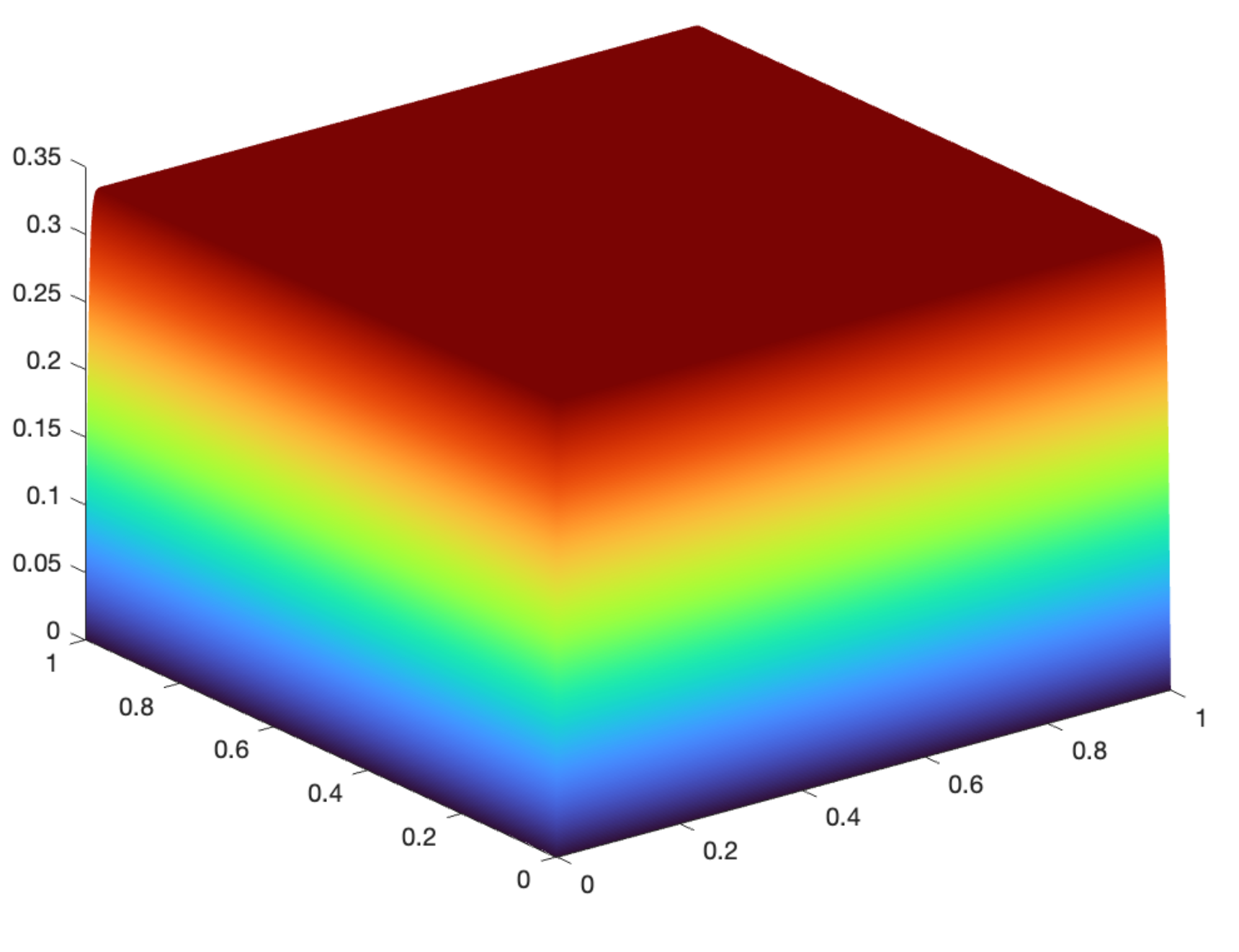}}
			\caption{Approximate solutions $u_\lmax^\kmax$, where $\lmax = 28$, $\kmax(28) = 2$, and $m=2$.}
			\label{fig:Gordon2B}
		\end{subfigure} \quad \quad
		\begin{subfigure}[b]{0.33\textwidth}%
			\includegraphics[width=\textwidth]{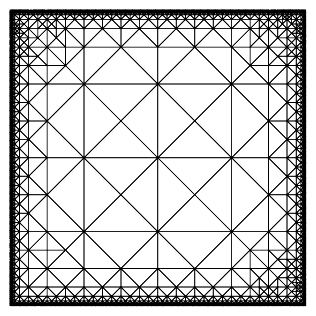}
			\caption{Mesh with $\# \TT_\ell = 4295$, where $\ell = 11$ and $m=1$.}
			\label{fig:Gordon2C}
		\end{subfigure}
		\caption{Using the norm $\enorm{\, \cdot \,}^{2} = \varepsilon\, \prod{\nabla \cdot }{\nabla \cdot } + \prod{\cdot }{\cdot }$ in Experiment~\ref{example:gordon2}. Top: Convergence plot of the error estimator $\eta_\ell(u^k_\ell)$ over $\mathtt{work}(\ell,k)$ and number of Zarantonello iterations on $\XX_\ell$ over $\mathtt{work}$ for $m=1$ (top, left) and $m=2$ (top, right). Bottom: Plot of the approximate solution $u_\ell^{\kmax}$ (bottom, left) and plot of a sample mesh (bottom, right).}
	\end{figure} This example is a variant of~{\cite[Experiment~5.2]{ahw2022}}. For $d=2$ and $\Omega = (0,1)^2$, let $\varepsilon = 10^{-5}$ and consider
	\begin{align*}
		-\varepsilon \Delta u^\exact + 2u^\exact+ \sin(u^\exact) = 1 \quad \text{ in } \Omega \quad \text{ subject to } \quad u^\exact= 0\text{\ on } \partial \Omega, 
	\end{align*}
	with the monotone semilinearity $b(v) = v+ \sin(v)$. In this case, the exact solution $u^\exact$ is unknown.  The used $\XX$-norm is given by $\enorm{\, \cdot\, }^{2} = \varepsilon\, \prod{\nabla \cdot }{\nabla \cdot } + \prod{\cdot }{\cdot }$. The particular choice of the $\XX$-norm allows for $\alpha = 1$ due to the monotonicity of $b(v)$. The problem clearly satisfies~\eqref{assump:ell},~\eqref{assump:car}, \eqref{assump:mon}, and~\eqref{assump:poly}. Moreover, $f = 1$ and $\f = 0$ satisfy~\eqref{assump:rhs}. In this experiment, we employ a slight modification of the error estimator~\eqref{eq:estimator:primal} following~\cite[Remark~4.14]{v2013}
	\begin{align*}
		\begin{split}
			\eta_\coarse(T, v_\coarse)^2 &\coloneqq \hslash_T^2 \,\norm{f + \varepsilon \Delta v_\coarse  - b(v_\coarse)}{L^2(T)}^2  + \hslash_T \, \norm{\jump{\varepsilon \, \nabla v_\coarse  \, \cdot \, \n }}{L^2(\partial T \cap \Omega)}^2,
		\end{split}
	\end{align*}
	where the scaling factors $\hslash_T = \min\{ \varepsilon^{-1/2} \, h_T, 1 \}$ ensure $\varepsilon$-robustness of the estimator. 
	
	In Figure~\ref{fig:Gordon2A}, we plot the error estimator $\eta_\ell(u^k_\ell)$ for  all $(\ell, k) \in \QQ$ against the $\mathtt{work}(\ell,k)$  for polynomial degrees $m \in \{1,2\}$. The decay rate is of (expected) optimal order \allowbreak $\mathcal{O}(\mathtt{work}(\ell,k)^{-m/2})$ as $|(\ell, k)| \to \infty$. The number of Zarantonello steps on each mesh refinement level $\ell$ stabilizes for $m \in \{1,2\}$ at three ($m=1$) and two ($m=2$) after an initial phase. For $m=2$, Figure~\ref{fig:Gordon2B} shows the approximate solution $u_\ell^\kmax$, where $\ell = 28$ and $\kmax(28) = 2$. Figure~\ref{fig:Gordon2C} depicts a mesh plot for $\# \TT_\ell = 4295$ for $\ell = 11$ and $m=1$. In particular, this experiment shows that Algorithm~\ref{algorithm:practical} is suitable for a setting with dominating nonlinear reaction given that a suitable norm on $\XX$ is chosen. Furthermore, we remark that the nonlinearity $b(v) = v + \sin(v)$ is globally Lipschitz continuous with Lipschitz constant $L=2$. In our experiments, $\delta_\ell$ is decreased twice, i.e., $\delta_\ell$ decreases from $1$ to $0.5 = 1/L$, which is optimal according to Remark~\ref{remark:delta} and remains uniformly bounded from below; cf.~Remark~\ref{rem:comparison}.
\end{experiment} 
\begin{experiment}[Goal-oriented AILFEM (GAILFEM)]\label{example:ms} 
	\begin{figure}
		\captionsetup[subfigure]{justification=justified}
	\begin{subfigure}[b]{0.48\textwidth}
		\centering
		\includegraphics[width=\textwidth, valign=t]{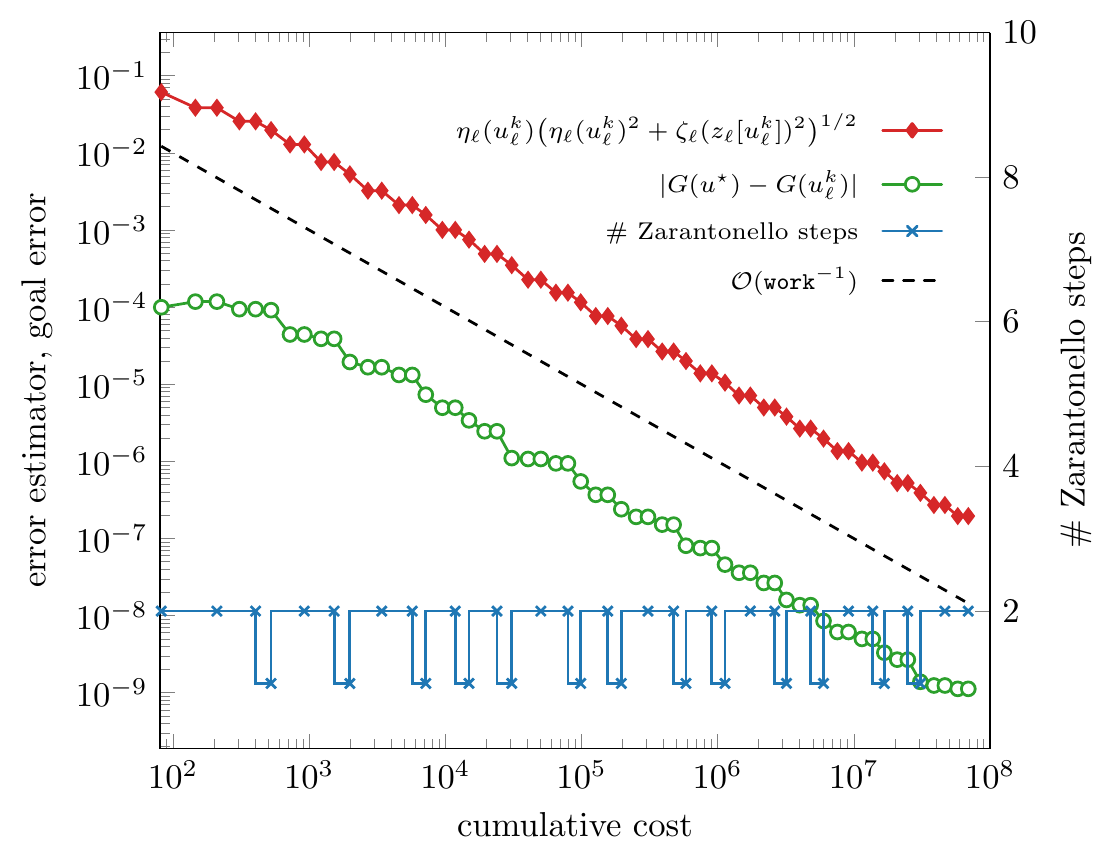}
		\caption{Results for $m=1$.}
		\label{fig:GAIL1}
	\end{subfigure} \quad
\begin{subfigure}[b]{0.48\textwidth}
	\centering
\includegraphics[width=\textwidth, valign=t]{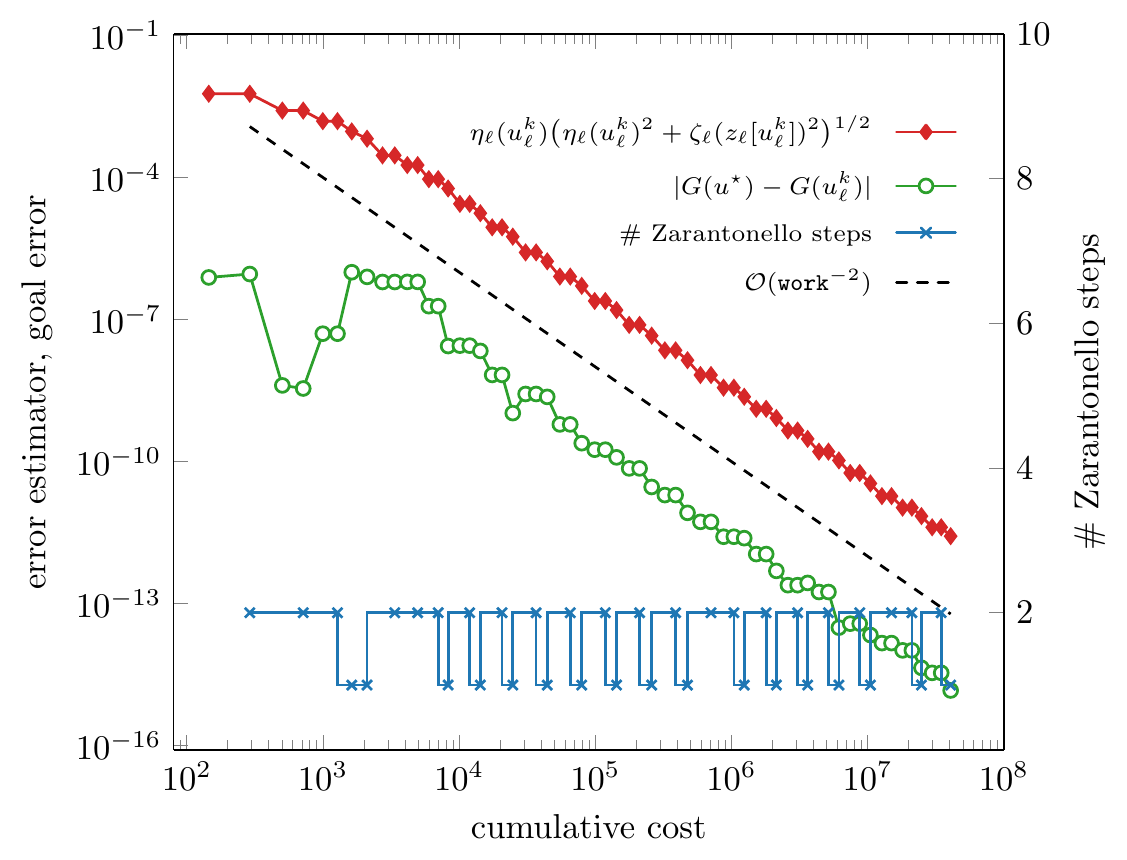}
\caption{Results for $m=2$.}
\label{fig:GAIL2}
	\end{subfigure} 
\begin{subfigure}[b]{0.48\textwidth}
	\centering
\includegraphics[width=\textwidth]{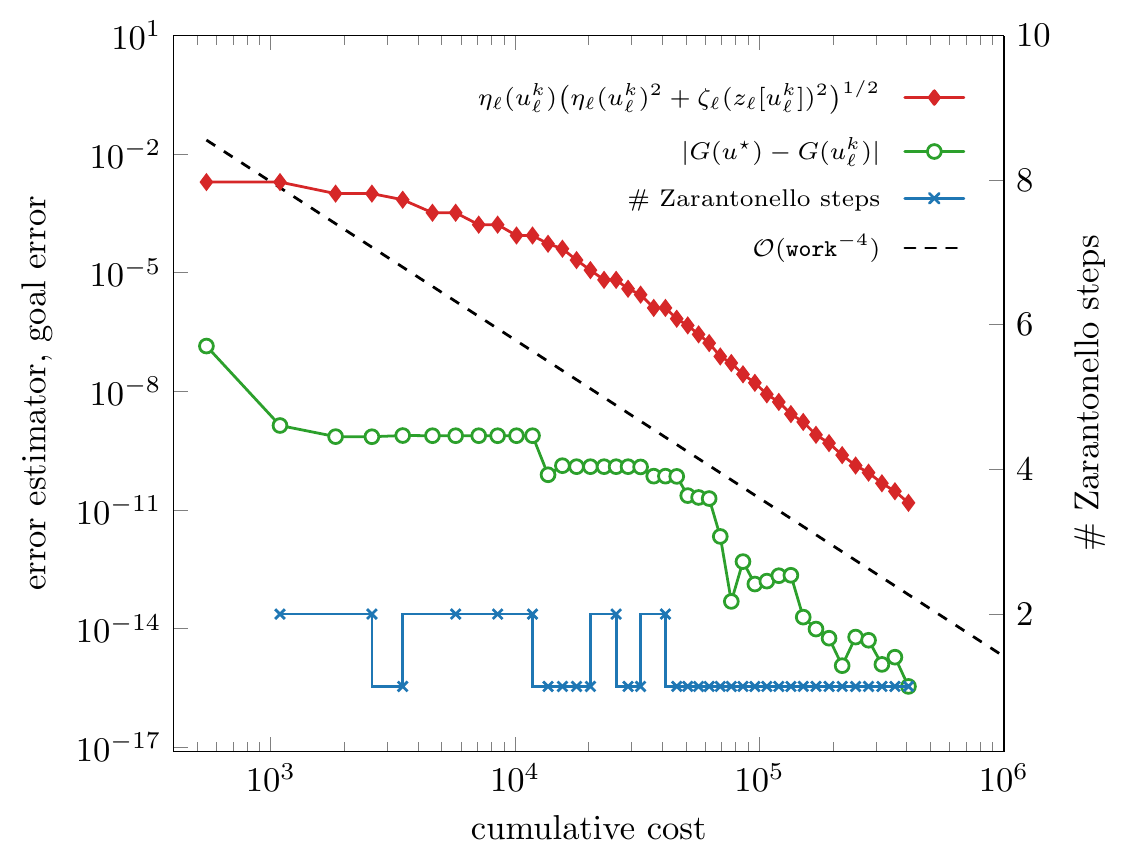}
\caption{Results for $m=4$.}
\label{fig:GAIL4}
\end{subfigure}\quad 
\begin{subfigure}[b]{0.48\textwidth}
	\centering
	\includegraphics[width=\textwidth]{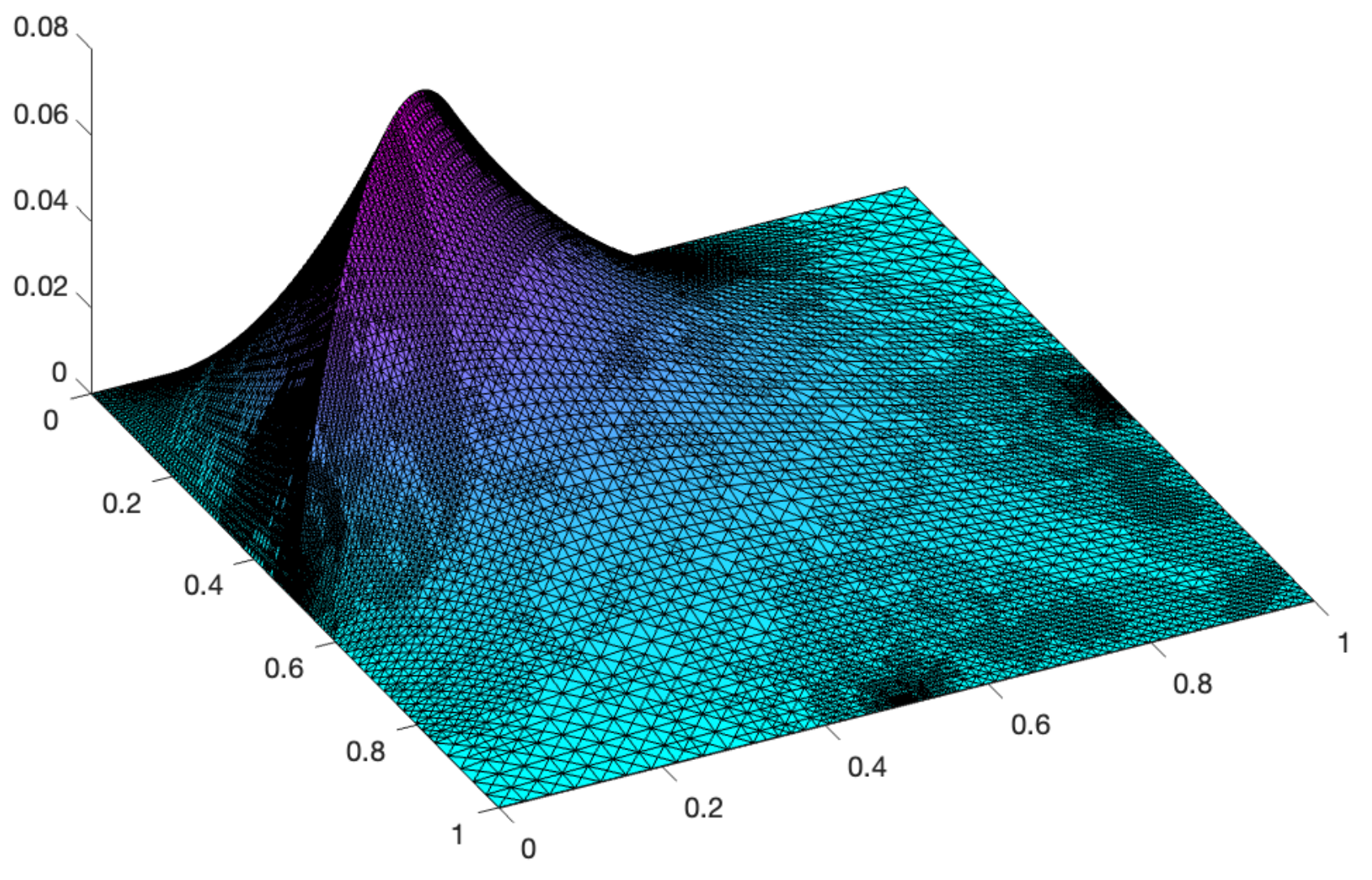}
	\caption{Plot of iterative solution $u^{\kmax}_\ell$, where $\ell = 16$, $\kmax(16) = 2$, $\dim(\XX_{\ell}) = 14599$, and $m=1$.}
	\label{fig:GAILplot}
\end{subfigure}
	\caption{\ref{fig:GAIL1}--\ref{fig:GAIL4}: Product error estimator $\eta_\ell(u_\ell^{k}) \big[\eta_\ell (u_\ell^{k})^2 + \zeta_\ell(z_\ell[u_\ell^{k}])^2\big]^{1/2}$ (diamond, left ordinate), absolute goal error $|G(u^\exact) - G(u_\ell^{k})|$ (circle, left ordinate), and number of Zarantonello steps on $\XX_\ell$ over $\mathtt{work}$ (cross, right ordinate) for $m=1$ (top, left), $m=2$ (top, right), and $m=4$ (bottom, left). \ref{fig:GAILplot}: Plot of an iterative solution $u^{\kmax}_\ell$. (bottom, right).}\label{fig:2d}
\end{figure}

	We also test a canonical extension of Algorithm~\ref{algorithm:practical} in a goal-oriented setting similar to that of~\cite[Example~7.3]{ms2009}. A thorough treatment of this problem (and the assumptions thereof) is found in~\cite[Example~35]{bbimp2022}. We use the proposed practical Algorithm~\ref{algorithm:practical} as the solve module for the semilinear primal problem in the GOAFEM algorithm \cite[Algorithm~17]{bbimp2022}. Let $\Omega= (0,1)^2$ and $\varepsilon = 1$. The weak formulation of the primal problem reads: Find $u^\exact \in H^1_0(\Omega)$ such that
	\begin{align}
		\prod{\nabla u^\exact}{\nabla v} + \prod{b(u^\exact)}{v} = \int_{\Omega} \f \cdot \nabla v \d{x}, \quad \text{ for all } v \in H^1_0(\Omega),
	\end{align}
	where $b(v) = v^3$ and $\f =  \chi_{\Omega_{\f}}\, (-1, 0)$ with the characteristic function $\chi_{\Omega_{\f}}$ of $\Omega_{\f} = \set{x \in \Omega}{x_1 + x_2 \le \tfrac{1}{2}}$.
	The weak formulation of the practical dual problem for the linearization point $w \in H^1_0(\Omega)$ reads: Find $z^\exact[w] \in H^1_0(\Omega)$ such that
	\begin{align*}
		\prod{\nabla z^\exact[w]}{\nabla v} + \prod{b'(w)z^\exact[w]}{v} = \int_{\Omega} \g \cdot \nabla v \d{x}, \quad \text{ for all } v \in H^1_0(\Omega),
	\end{align*}
	where $b'(v) = 3v^2$ and $\g = \chi_{\Omega_{\g}}\,(-1, 0)$ with $\Omega_{\g}= \set{x \in \Omega}{x_1 + x_2 \ge \tfrac{3}{2}}$. The goal functional thus reads
	\begin{align*}
		G(v) \coloneqq - \int_{\Omega_{\g}} \frac{\partial v}{\partial x_1} \d{x}\quad \text{ for all } v\in H^1_0(\Omega). 
	\end{align*}.
	Since $\div(\g) = 0$ on every element $T \in \TT_0$, the associated error estimator for the dual problem reads 
		\begin{align}
			\begin{split}\label{eq:estimator:dual}
				\zeta_\coarse(w; T, v_\coarse)^2 &\coloneqq h_T^2 \,\norm{\Delta v_\coarse - b'(w)(v_\coarse)}{L^2(T)}^2 + h_T \, \norm{\jump{(\nabla v_\coarse - \g ) \, \cdot \, \n }}{L^2(\partial T \cap \Omega)}^2.
			\end{split}
		\end{align}
 We used $\enorm{\, \cdot\, }^2 = \sprod{\cdot \, }{\, \cdot}$ as the $\XX$-norm. For various polynomial degrees $m \in \{1, 2, 4\}$, Figure~\ref{fig:GAIL1}--\ref{fig:GAIL4} shows the results of the proposed GAILFEM algorithm driven by the product estimator $\eta_\ell(u_\ell^{k}) \allowbreak \big[\eta_\ell (u_\ell^{k})^2 + \zeta_\ell(z_\ell[u_\ell^{k}])^2\big]^{1/2}$, which is an upper bound to the goal error difference $G(u^\exact)- G(u^\exact_\ell)$ and a viable way to recover optimal convergence rates; cf.~\cite{bbimp2022}. We plot the estimator product $\eta_\ell(u_\ell^{k}) \big[\eta_\ell (u_\ell^{k})^2 + \allowbreak \zeta_\ell(z_\ell[u_\ell^{k}])^2\big]^{1/2}$, the number of Zarantonello steps, and the absolute goal error difference $|G(u^\exact) - G(u_\ell^{k})|$ over the $\mathtt{work}(\ell,k)$, where $G(u^\exact) = -0.0015849518088245$ serves as a reference value; see \cite[Example~35]{bbimp2022}. In Figure~\ref{fig:GAILplot}, we plot the sample solution $u^{\kmax}_\ell$, where $\ell = 13$, $\kmax(13) = 2$, and $m=1$.
	
	The decay rate is of (expected) optimal order $\mathcal{O}(\mathtt{work}(\ell,k)^{- \,m})$ for $|(\ell, k)| \to \infty$, where $m\in \{1,2,4\}$ is the polynomial degree of the FEM space $\XX_\ell$. The number of Zarantonello steps does not exceed two for $m= \{1, 2, 4\}$ and stabilizes after an initial phase at one for $m=4$, respectively. Figure~\ref{fig:gailfem:meshes} depicts two meshes for $m=1$ and $m=4$.

\begin{figure}
	\begin{subfigure}[t]{0.38\textwidth}
		\centering
		{\includegraphics[width=\textwidth, valign=t]{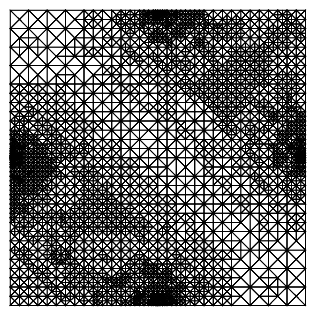}}
		\caption{Mesh generated for $m=1$, where $\dim \XX_\ell = 3092$ and $\ell = 12$.}
		\label{fig:GAIL:mesh:1}
	\end{subfigure} \quad \quad 
\begin{subfigure}[t]{0.38\textwidth}
	{\includegraphics[width=\textwidth, valign=t]{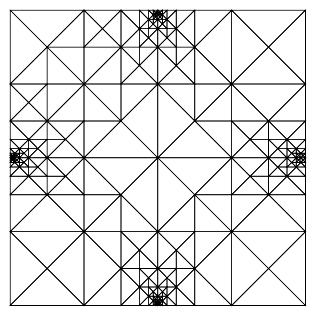}}
	\caption{Mesh generated for $m=4$, where $\dim \XX_\ell = 3081$ and $\ell = 12$.}
		\label{fig:GAIL:mesh:4}
	\end{subfigure}
	\caption{Generated GAILFEM meshes for $m=1$~{(Figure \ref{fig:GAIL:mesh:1})} and $m=4$~(Figure~\ref{fig:GAIL:mesh:4}).}
		\label{fig:gailfem:meshes}
\end{figure}
\end{experiment}

\appendix
\section{Convergence for vector-valued semilinear pdes}\label{section:appendix}
This appendix aims to extend the analysis from Section~\ref{section:monlip} to problems where the monotone operator does not have a potential, e.g., vector-valued semilinear PDEs.
We prove plain convergence of Algorithm~\ref{algorithm:idealized} without the assumption~\eqref{eq:potential} and with the modified stopping criterion
		\begin{align}\label{eq:stopp:alternative}
			\enorm{u^{k}_\ell - u^{k-1}_\ell} \le \lambda \, \eta_\ell(u^k_\ell) \quad \land \quad \enorm{u^k_\ell} \le 2M \tag{{\rm i.b$''$}}
			\end{align}
		 replacing Algorithm~\ref{algorithm:idealized}{\eqref{eq:stoppingcrit}}. The proof requires some preliminary observations:  First, the convergence of the exact discrete solutions $u^\exact_\ell$ towards the exact solution $u^\exact_\infty$ in the so-called discrete limit space, which dates back to the seminal work~\cite{bv1984}. Second, we need to show that the approximate discrete solutions $u^k_\ell$ converge to the same limit.
\begin{lemma}\label{lemma:inexact}
	Suppose that $\AA$ satisfies~\eqref{eq:strongly-monotone} and~\eqref{eq:locally-lipschitz}. With the discrete subspaces $\XX_\ell \subset \XX$ from Algorithm~\ref{algorithm:idealized} (with or without the modified stopping criterion~\eqref{eq:stopp:alternative}), define the discrete limit space $\XX_\infty := \overline{\bigcup_{\ell = 0}^{\lmax} \XX_\ell}$, where we recall that $\lmax = \sup \set{ \ell \in \N_0}{(\ell,0) \in \QQ}$. Then, there exists a unique $u^\exact_\infty \in \XX_\infty$ which solves
	\begin{align}\label{eq:discretelimitspace}
		\prod{\AA u^\exact_\infty}{v_\infty} = \prod{F}{v_\infty} \quad \text{ for all } v_\infty \in \XX_\infty.
	\end{align}
	 Moreover, given the exact discrete solutions $u^\exact_\ell \in \XX_\ell$, it holds that
	\begin{align}\label{eq:plain:exact}
		\enorm{u^\exact_\infty - u_\ell^\exact} \to 0 \quad \text{ as } \quad \ell \to \lmax.
	\end{align}
Additionally, suppose~\eqref{assumption:stab}--\eqref{assumption:rel} and suppose that the choice of $\delta >0$ in Algorithm~\ref{algorithm:idealized} ensures norm contraction~\eqref{crucial:contractivity}. Then, the approximations  $u_\ell^{k}$ computed in Algorithm~\ref{algorithm:idealized} fulfil that
\begin{align}~\label{eq:lemma:inexact}
	\enorm{u^\exact_\infty - u_\ell^{k}} \to 0 \quad \text{ as } \quad (\ell, k) \in \QQ \quad \text{ with } \quad  |(\ell,k)| \to \infty.
\end{align} \vskip-\lastskip
\end{lemma}
\begin{proof}
	The proof consists of three steps.

{\bf Step 1 (exact solutions).} Since $\XX_\ell \subseteq \XX_{\ell+1} \subset \XX$, the discrete limit space $\XX_\infty := \overline{\bigcup_{\ell = 0}^{\lmax} \XX_\ell}$ is a closed subspace of $\XX$. Proposition~\ref{prop:existence} proves the existence of a unique $u_\infty^\exact \in \XX_\infty$ satisfying~\eqref{eq:discretelimitspace}. The Galerkin solutions $u^\exact_\ell$ from~\eqref{eq:weakform:discrete} are also Galerkin approximations of $u^\exact_\infty$. Hence, there holds the C\'{e}a-type estimate
\begin{align}
	\enorm{u^\exact_\infty - u^\exact_\ell} \eqreff{eq:lemma:cea}{\le} \Ccea \min_{v_\ell \in \XX_\ell} \enorm{u^\exact_\infty - v_\ell} \xrightarrow[]{\ell \to \lmax} 0,
\end{align} 
where convergence follows by definition of $\XX_\infty$. 

{\bf Step 2 (approximate solutions for $\lmax = \infty$).} The norm contraction~\eqref{crucial:contractivity} and $u^0_{\ell+1} = u^\kmax_\ell$ reveal that
\begin{align*}
	0 \le \enorm{u_{\ell+1}^\exact - u_{\ell+1}^{\kmax(\ell+ 1)}} \eqreff{crucial:contractivity}{\le} \qN^{{\kmax(\ell + 1)}}  \enorm{u_{\ell+1}^\exact - u_{\ell+1}^0} \le \qN \big[ \enorm{u_{\ell}^\exact - u_{\ell}^{{\kmaxl}}} + \enorm{u_{\ell+1}^\exact - u_{\ell}^\exact} \big].
\end{align*}
From Step~1, we infer that $(u^\exact_\ell)_{\ell \in \N_0}$ is a Cauchy sequence. Defining $a_\ell := \enorm{u^\exact_\ell - u^{\kmax}_\ell}$ and $b_\ell := \qN\, \enorm{u^\exact_{\ell+1} - u^\exact_\ell}$, the last estimate can be rewritten as
\begin{align*}
	0 \le a_{\ell + 1} \le \qN\, a_\ell + b_\ell, \quad \text{ where } \quad \lim_{\ell \to \infty} b_\ell = 0.
\end{align*}
It follows from elementary calculus (cf.~\cite[Corollary~4.8]{axioms}) that
\begin{align*}
0= \lim_{\ell \to \infty} a_\ell = \lim_{\ell \to \infty} \enorm{u^\exact_\ell - u^{\kmax}_\ell}.
\end{align*}
Altogether, we obtain that
\begin{align*}
	\enorm{u^\exact_\infty - u_\ell^{{k}}} &\le \enorm{u^\exact_\infty - u_\ell^{\exact}} + \enorm{u^\exact_\ell - u_\ell^{{k}}} \eqreff{crucial:contractivity}{\le} \enorm{u^\exact_\infty - u_\ell^{\exact}} + \enorm{u^\exact_\ell - u_\ell^{0}} \\
& \le  \enorm{u^\exact_\infty - u_\ell^{\exact}} +  \enorm{u^\exact_\ell - u_{\ell-1}^{\exact}} + \enorm{u^\exact_{\ell-1} - u_{\ell-1}^{\kmax}} \to 0 \quad \text{ as } \quad \ell \to \infty.
\end{align*}

{\bf Step 3 (approximate solutions for $\lmax < \infty$ and $\kmaxl = \infty$).} It holds that $u^\exact_\infty = u^\exact_\lmax$ and hence, due to~\eqref{crucial:contractivity}, 
	\begin{align*}
		\enorm{u^\exact_\infty \!- u^k_\ell} = \enorm{u^\exact_\lmax \!- u^k_\lmax} \to 0 \quad \text{ as } \quad |(\ell, k)| \to \infty.
	\end{align*}
This concludes the proof.
\end{proof}

The following theorem states plain convergence in the abstract setting of the proposed AILFEM algorithm.
\begin{theorem}[Plain convergence]
	\label{prop:plain-convergence}
	Suppose that $\AA$ satisfies~\eqref{eq:strongly-monotone} and~\eqref{eq:locally-lipschitz}. Suppose the axioms of adaptivity~\eqref{assumption:stab}--\eqref{assumption:rel}. Suppose that the choice of $\delta >0$ in Algorithm~\ref{algorithm:idealized} ensures~\eqref{crucial:contractivity}. Then, for any choice of the marking parameters $0 < \theta \le 1$, $\lambda >0$, and $1 \le \Cmark \le \infty$, Algorithm~\ref{algorithm:idealized} with modified stopping criterion~\eqref{eq:stopp:alternative} guarantees convergence of the quasi-error from~\eqref{eq:quasi-error}, i.e.,
	\begin{align}\label{appendix:eq:quasi-error}
	\Delta_\ell^k	= \enorm{u^\exact - u_\ell^{k}} + \eta_\ell(u_\ell^{k}) \to 0 \quad \text{ as } (\ell, k) \in \QQ \text{ with }  |(\ell, k)| \to \infty.
	\end{align} \vskip-\lastskip
\end{theorem}

\begin{proof} The assertion $|(\ell, k)| \to \infty$ consists of two cases:
	
	{\bf{Case~1 ($\lmax = \infty$).}} 
Recall the generalized estimator reduction~\cite[Lemma~4.7]{axioms}: Let $\omega >0$. Given the D\"{o}rfler marking in Algorithm~\ref{algorithm:idealized}{\rm{(iii)}}, it follows that
 \begin{align}\label{eq:generalized_estimator_reduction}
 	\eta_{\ell+1}({u^{{\kmax}}_{\ell+1})}^2 \le \qest \, \eta_{\ell}(u_\ell^{\kmax})^2 + \Cest \, \enorm{u^{{\kmax}}_{\ell+1} - u_\ell^{\kmax}}^2,
\end{align}
where  $0 < \qest := (1 + \omega) \, \big[ \, 1 - (1-\qred^2) \, \theta \, \big] < 1$ and $\Cest := (1 + \omega^{-1}) \, \Cstab[4M]^2$ with $\omega > 0$ being sufficiently small and where $4M$ stems from nested iteration~\eqref{crucial:nestediteration}. From Lemma~\ref{lemma:inexact}, we infer that $\enorm{u^{{\kmax}}_{\ell+1} - u_\ell^\kmax} \to 0$ as $\ell \to \infty$. Hence, it follows from elementary calculus (cf.~\cite[Corollary~4.8]{axioms}) that $\eta_\ell(u^{\kmax}_\ell) \to 0$ as $\ell \to \infty$. Moreover, this and Lemma~\ref{lemma:inexact} prove that
	  \begin{align*}
		\enorm{u^\exact - u_\ell^\kmax}  & \eqreff{assumption:rel}{\le} \Crel \eta_\ell(u^\exact_\ell) +  \enorm{u^\exact_\ell - u^\kmax_\ell} \eqreff{assumption:stab}{\le} 
		\Crel \eta_\ell(u^\kmax_\ell) + (1 + \Crel \Cstab[3M])\,\enorm{u^\exact_\ell - u^\kmax_\ell} \\
		&\hspace{-10mm}\le 
		\Crel \eta_\ell(u^\kmax_\ell) + (1 + \Crel \Cstab[3M])\big[\enorm{u^\exact_\ell - u^\exact_\infty} + \enorm{u^\exact_\infty - u^\kmax_\ell} \big] \xrightarrow[]{\ell \to \infty} 0.
	\end{align*}
We conclude that $\enorm{u^\exact-u^\kmax_\ell} + \eta_\ell(u^\kmax_\ell) + \eta_\ell(u^\exact_\ell) \to 0$ as $\ell \to \infty$. Due to~\eqref{eq1:cor:zarantonello} together with Lemma~\ref{lemma:inexact} and for $\Crel' \coloneqq 1 + \Crel$, this yields for all $(\ell, k) \in \QQ$ that
\begin{align*}
\Delta_\ell^k  &\le \Crel' \eta_\ell(u^\exact_\ell)+ \big[1+\Cstab[3M]\big]\, \enorm{u^\exact_\ell - u^k_\ell} \eqreff{crucial:contractivity}{\le}  \Crel' \eta_\ell(u^\exact_\ell) + \big[1 + \Cstab[3M]\big]\, \enorm{u^\exact_\ell - u^0_\ell} \\
&\le  \Crel' \eta_\ell(u^\exact_\ell) + \big[ 1 + \Cstab[3M] \big]\, \big[\enorm{u^\exact_\ell - u^\exact_{\ell-1}}  + \enorm{u^\exact_{\ell-1} - u^\kmax_{\ell-1}}\big] \xrightarrow[]{\ell \to \infty} 0.
\end{align*}
 This concludes the proof of the first case.

	{\bf{Case~2 ($\lmax < \infty$ and $\kmax(\lmax) =\infty$).}}
	Since $\kmax(\lmax) = \infty$, at least one of the cases is met:
	\begin{align*}
		\#\set{k\in \N_0}{\enorm{u^k_{\lmax}} > 2M }   = \infty \quad \text{ or } \quad \#\set{k\in \N_0}{\lambda \eta_{\lmax}(u^k_{\lmax}) < \enorm{u^k_{\lmax} - u^{k-1}_{\lmax}}} = \infty.
	\end{align*}
Since norm contraction~\eqref{crucial:contractivity} holds, the arguments to obtain~\eqref{eq:normcrit} prove the existence of $k_0 \in \N$ such that, for all $ k \ge k_0$, it holds that
	\begin{align*}
		\enorm{u^k_{\lmax}} \le 2M.
	\end{align*}
We deduce from the (not met) stopping criterion in Algorithm~\ref{algorithm:idealized}{\eqref{eq:stopp:alternative} and~\eqref{crucial:contractivity}} that
	\begin{align*}
	\lambda \eta_{\lmax}(u^k_{\lmax}) \stackrel{{\eqref{eq:stopp:alternative}}}{<} \enorm{u^k_{\lmax} - u^{k-1}_{\lmax}} \xrightarrow[]{k \to \infty} 0.
	\end{align*}
With contraction~\eqref{crucial:contractivity}, we see that
\begin{align*}
\enorm{u^\exact - u_{\lmax}^k} \! \eqreff{assumption:rel}{\le}\! \Crel \eta_{\lmax}(u^\exact_{\lmax}) +  \enorm{u^\exact_{\lmax}-  u^k_{\lmax}}\! \eqreff{assumption:stab}{\le} \! \Crel \eta_{\lmax}(u^k_{\lmax})+ (1 + \Cstab[3M])\,\enorm{u^\exact_{\lmax} - u^k_{\lmax}} \xrightarrow[]{k \to \infty} 0.
\end{align*}
This concludes the proof of the second case and the proof is complete.
\end{proof}

The next corollary states that the exact solution $u^\exact = u^\exact_{\lmax}$ is discrete if $\lmax<\infty$. Moreover, if there exists $\ell$ with $\eta_{\ell}(u^\kmax_{\ell})=0$, then the exact solution $u^\exact$ coincides with $u^\kmax_{\ell}$.

\begin{corollary}\label{cor:earlystop}
	Under the assumptions of Theorem~\ref{prop:plain-convergence}, there hold the following implications:
	\begin{itemize} 
		\item[\rm (i)] If $\lmax = \sup \set{ \ell \in \N_0}{(\ell,0) \in \QQ} < \infty$, then  $u^\exact = u_{\lmax}^\exact$ and $\eta_\lmax(u^\exact_\lmax) = 0$.
		\item[\rm (ii)] If $\ell \in \N_0$ with $\kmax < \infty$ and $\eta_{\ell}(u^\kmax_{\ell}) = 0$, then $u^\kmax_{\ell} = u^\exact = u^\exact_{\ell}$.
		\end{itemize}
\end{corollary}

\begin{proof} 
	{\bf (i). }According to Theorem~\ref{prop:plain-convergence}, it holds that 
	\begin{align*} 
	\Delta_{\lmax}^k = \enorm{u^\exact - u_{\lmax}^k} + \eta_\lmax(u^k_\lmax)   \to 0
	\quad \text{ as }\quad  k \to \infty.
	\end{align*}
	Norm contraction~\eqref{crucial:contractivity} proves that 
	\begin{align*}
	\enorm{u_{\underline\ell}^\exact - u_{\lmax}^k}
	\le \qN^k \, \enorm{u_{\lmax}^\exact - u_{\lmax}^0}
	\to 0
	\quad \text{ as }\quad k \to \infty.
	\end{align*}
	Uniqueness of the limit yields that $u^\exact=u_{\lmax}^\exact$. With stability~\eqref{assumption:stab}, we obtain that
	\begin{align*}
 	0 \le \eta_\lmax(u^\exact_\lmax) \le \eta_\lmax(u^k_\lmax) + \Cstab[3M] \, \enorm{u^\exact_\lmax - u^k_\lmax} \to 0 \quad \text{ as } \quad k \to \infty.
	\end{align*}
	This concludes the proof of {\rm (i)}.
	
	{\bf (ii). }Note that the stopping criterion in Algorithm~\ref{algorithm:idealized}{\eqref{eq:stopp:alternative}} implies that $ \enorm{u_{\ell}^{\kmax}-u_{\ell}^{\kmax-1}} \le \lambda \,\eta_{\ell}(u_{\ell}^\kmax)=0$ by assumption. Thus, $u^{\kmax}_\ell = u^{\kmax-1}_\ell$. This implies that $u^{\kmax-1}_\ell$ is a fixed point of $\Phi_\ell(\delta; \cdot)$. Since the fixed point is unique, we infer that $u^{\kmax}_\ell = u^{\kmax-1}_\ell = u^\exact_\ell$. With reliability~\eqref{assumption:rel}, we thus obtain that
	\begin{align*}
		\enorm{u^\exact - u_{\ell}^\exact} \eqreff{assumption:rel}{\le} \Crel \, \eta_{\ell}(u^{\exact}_{\ell}) = \Crel \,\eta_{\ell}(u^{\kmax}_{\ell}) =0. 
	\end{align*}
	This concludes the proof.
\end{proof}

Plain convergence is required to obtain results proving weak convergence in the spirit of~\cite[Lemma~28]{bbimp2022}. This is pivotal for achieving quasi-orthogonality along the lines of~\cite[Lemma~29]{bbimp2022}, which can substitute~\eqref{eq:qo:abstract} in the proof of full linear convergence. Details are omitted.

\renewcommand*{\bibfont}{\scriptsize}
\setlength{\biblabelsep}{-5pt}
\printbibliography  


\end{document}